\documentclass[11pt,reqno,a4paper]{amsart}
\usepackage{amsmath,amsthm,verbatim,amscd,amssymb,setspace,enumitem,hyperref}
\usepackage{exscale,color}
\usepackage[colorinlistoftodos,prependcaption,textsize=tiny]{todonotes}

\usepackage{enumitem}
\usepackage{mathrsfs}

\renewcommand{\epsilon}{\varepsilon}

\newcounter{mtheorem}
\newtheorem{mtheorem}[mtheorem]{Theorem}

\newcommand{{\vol}}{\rm vol}

\newcommand{\Ric}{\operatorname{Ric}}

\newcommand{\IP}{\mathbb{P}}
\newcommand{\IC}{\mathbb{C}}
\newcommand{\Rm}{\operatorname{Rm}}

\newtheoremstyle{fancy}{}{}{\itshape}{}{\textbf\bgroup}{.\egroup}{ }{}
\newtheoremstyle{fancy2}{}{}{\rm}{}{\textbf\bgroup}{.\egroup}{ }{}

\theoremstyle{fancy}
\newtheorem{theorem}{Theorem}[section]
\newtheorem{lemma}[theorem]{Lemma}

\newtheorem{prop}[theorem]{Proposition}

\theoremstyle{fancy2}
\newtheorem{definition}[theorem]{Definition}

\newtheorem{remark}[theorem]{Remark}

\newtheorem{claim}[theorem]{Claim}

\setlist{leftmargin=*}

\numberwithin{equation}{section}

\begin{document}
\title{Non-collapsed finite time singularities of the Ricci flow on compact K\"ahler surfaces are of Type I}
\date{\today}

\author{Ronan J.~Conlon}
\address{Department of Mathematical Sciences, The University of Texas at Dallas, Richardson, TX 75080}
\email{ronan.conlon@utdallas.edu}

\author{Max Hallgren}
\address{Department of Mathematics, Rutgers University, Piscataway, NJ 08854, USA}
\email{mh1564@scarletmail.rutgers.edu}

\author{Zilu Ma}
\address{Department of Mathematics, Rutgers University, Piscataway, NJ 08854, USA}
\email{zilu.ma@rutgers.edu}

\date{\today}

\begin{abstract}
We show that any non-collapsed finite time singularity of the Ricci flow on a compact K\"ahler surface is of Type I. Combined with a previous result of the first author, Cifarelli, and Deruelle, it follows that any such singularity is modeled on the shrinking Ricci soliton of Feldman-Ilmanen-Knopf on the total space of the line bundle $\mathcal{O}_{\mathbb{P}^1}(-1)\to\mathbb{P}^{1}$.
\end{abstract}

\maketitle

\markboth{Ronan J.~Conlon, Max Hallgren, and Zilu Ma}{Non-collapsed finite time singularities of the Ricci flow on compact K\"ahler surfaces are of Type I}

\section{Introduction}

\subsection{Background}
Let $(M,\,g)$ be a compact K\"ahler manifold with K\"ahler metric $g$. The K\"ahler-Ricci flow is a one-parameter family $(g_t)_{t\in [0,T)}$ of K\"ahler metrics on $M$ with $g_{0}=g$ evolving under the evolution equation
\begin{equation*}
\frac{\partial g_{t}}{\partial t}=-\operatorname{Ric}_{g_{t}},
\end{equation*}
where $\operatorname{Ric}_{g_{t}}$ denotes the Ricci curvature of $g_{t}$. This flow can be reformulated in terms of a strictly parabolic complex Monge-Amp\`ere equation, from which short time existence and uniqueness of solutions follow from standard theory. However, being non-linear, the flow may develop singularities in finite time. Let $[0,\,T), T>0,$ denote the maximal time interval of existence of $(g_{t})$ and assume that $T<\infty$. Then we know that the norm of the curvature $|{\operatorname{Rm}_{g_{t}}}|_{g_{t}}\to\infty$ as $t\to T^{-}$ \cite{Hamilton}. If the curvature does not blow up too fast, i.e., if \begin{equation*} 
\sup_{t\,\in\,[0,\,T)} \sup_{M} |{\operatorname{Rm}_{g_{t}}}|_{g_{t}}(T-t)<\infty, \end{equation*}
then we say that the singularity is of \emph{Type I}. Otherwise, it is of \emph{Type II}.

Finite time Type I singularities of the K\"ahler-Ricci flow are of interest because they are modeled on non-flat smooth shrinking gradient K\"ahler-Ricci solitons \cite{topping, nabersoliton} with bounded curvature.
K\"ahler-Ricci flows beginning from any K\"ahler metric in a positive multiple of the first Chern class of a del Pezzo surface develop finite time Type I singularities \cite{TianEinstein,Koiso,XuJiaZhu,perltian,TianZhuSoliton}, and the corresponding singularity model is the unique \cite{BandoMabuchi,TianZhuUniqueness} shrinking K\"ahler-Ricci soliton on the del Pezzo surface in question. Moreover, in this dimension, all possible Type I singularity models have been classified \cite{CCD}. On the other hand, in complex dimensions three and above, there exist finite time singularities of the flow that are not of Type I \cite{li2, miao}.

We say that the flow $(M,\,(g_{t})_{t\in[0,\,T)}),\,T<\infty,$ is \emph{non-collapsed} if
\begin{equation}\label{VolDef}
\liminf_{t\to T^{-}}\operatorname{vol}_{g_{t}}(M)>0.
\end{equation}
In \cite{SongTypeI}, it was shown that non-collapsed finite time singularities of the K\"ahler-Ricci flow beginning from any $U(n)$-invariant metric on the blowup of $\mathbb{P}^{n}$ at one point are of Type I, and subsequently in \cite{GuoSong} that any resulting singularity model is the shrinking Ricci soliton of Feldman-Ilmanen-Knopf on the total space of $\mathcal{O}_{\mathbb{P}^n}(-1)\to\mathbb{P}^{n}$ \cite{FIK}. Similar results were obtained by M\'aximo \cite{maximo} for $n=2$. In \cite{CCD},
it was shown that any non-collapsed finite time singularity of Type I on a compact K\"ahler surface is modeled on the Feldman-Ilmanen-Knopf shrinking soliton
on the total space of $\mathcal{O}_{\mathbb{P}^1}(-1)\to\mathbb{P}^{1}$ (without any symmetry assumptions). In this article, we prove that any non-collapsed finite time singularity of the K\"ahler-Ricci flow on a compact K\"ahler surface is of Type I.

\subsection{Main result}
Our main result is the following.

\begin{mtheorem}\label{noncollapsedtheorem} 
Let $(M,\,(g_{t})_{t\in[0,\,T)}),\,T\in(0,\,\infty),$ be a K\"ahler-Ricci flow on a compact K\"ahler surface $M$ with $t=T$ the first singular time. If 
$\liminf_{t\to T^{-}}\operatorname{vol}_{g_{t}}(M)>0$, then $$\sup_{t\in [0,\,T)} \sup_{M} |{\Rm_{g_{t}}}|_{g_{t}}(T-t)<\infty.$$ In particular, $g_{t}$ develops a finite time Type I singularity at $t=T$.
\end{mtheorem}

From \cite{SoWe1}, it is known that such a singularity corresponds to the contraction of finitely many disjoint $(-1)$-curves. As a consequence of Theorem \ref{noncollapsedtheorem} and \cite[Theorem B]{CCD}, any singularity model of the flow based at a contracted $(-1)$-curve must be the flow corresponding to the shrinking soliton of Feldman-Ilmanen-Knopf \cite{FIK} on the total space of the line bundle $\mathcal{O}_{\mathbb{P}^1}(-1)\to\mathbb{P}^{1}$. In general, we expect that all finite time singularities of the K\"ahler-Ricci on a compact K\"ahler surface are of Type I. 

\subsection{Outline of Proof}
Let $(\widetilde{M},\,\widetilde{J})$ be a compact K\"ahler surface with complex structure $\widetilde{J}$ and let $(\widetilde{g}_t)_{t\in [0,\,T)},\,T<\infty$, be a K\"ahler-Ricci flow on $\widetilde{M}$ with $t=T$ the first singular time satisfying \eqref{VolDef}. Then by \cite{SoWe1}, 
there is a blowdown map $\pi:\widetilde{M} \to N$ contracting finitely many disjoint $(-1)$-curves. For simplicity, we assume that $\pi$ contracts only one such curve $\widetilde{E}$. By Bamler's compactness theory \cite{Bam3}, a tangent flow based at a point of $\widetilde{E}$ at the singular time $T$ is a shrinking gradient K\"ahler-Ricci soliton surface $(X,\,g,\,J,\,f)$ with (a priori infinitely many) isolated orbifold singularities. This theory is introduced, along with notation and other preliminaries, in Section $2$. 

In Section $3$, we show that $X$ has only finitely many orbifold singularities and that $\widetilde{E}$ limits to a compact holomorphic curve $E$ in $X$. To achieve this, we fix a volume form $\Omega_N$ on $N$ and work with the function $\widetilde{u}_t = \log \left( \frac{\pi^{\ast}\Omega_N}{\widetilde{\omega}_t^2} \right)$ defined on $\widetilde{M}$. Here, $\widetilde{\omega}_{t}$ denotes the K\"ahler form of the evolving K\"ahler metric $\widetilde{g}_{t}$.
We first demonstrate in Proposition \ref{functionu}(i)--(iii) that $\widetilde{u}_t$ approximately solves the heat equation away from $\widetilde{E}$ and has a logarithmic singularity along $\widetilde{E}$. By using these facts and analyzing the singular term along $\widetilde{E}$, we derive the gradient estimate $|\nabla^{\widetilde{g}_{t}}\widetilde{u}_t|_{\widetilde{g}_{t}} \lesssim (T-t)^{-\frac{1}{2}}$ away from $\widetilde{E}$, in a sense made precise by Proposition \ref{functionu}(iv). Combining this estimate with the fact that $\sqrt{-1}\partial \bar{\partial} \widetilde{u}_t \approx {\operatorname{Ric}_{\widetilde{g}_{t}}}$ away from $\widetilde{E}$, we show in Proposition \ref{NoncollapsedTypeI} that any Type II singularity model based at spacetime points converging to $\widetilde{E}$ slowly must be a non-flat smooth ALE Calabi-Yau\footnote{A Calabi-Yau manifold for us is a Ricci-flat K\"ahler manifold.} surface. We identify this as flat $\mathbb{C}^{2}$ by using the topological classification of ALE Calabi-Yau surfaces, together with the fact that a punctured neighborhood of $\widetilde{E}$ is diffeomorphic to $\mathbb{C}^2\setminus \{0\}$. This yields a contradiction, and so we deduce that a Type I curvature bound holds away from time-dependent shrinking neighborhoods of $\widetilde{E}$, in a sense made quantitative in Proposition \ref{NoncollapsedTypeI}. From this, it follows in Proposition \ref{dichotomy} that any orbifold point of $X$ is the limit of spacetime points converging quickly to $\widetilde{E}$, from which it follows that all orbifold points of $X$ lie along the limit curve $E\subseteq X$ of $\widetilde{E}$, whose existence is guaranteed by Lemma \ref{Hausdorff}. For this reason, we focus our analysis on $E$. 

In Proposition \ref{finiteorb}, we prove that $E$ is compact in the following way. For the sake of a contradiction we assume that $E$ is unbounded. Then Proposition \ref{bigcurve} shows that the connectedness of $\widetilde{E}$ would give an unbounded connected component of $E$. Combining Bamler's partial regularity theory for $X$ (cf.~Proposition \ref{highcurvaturesmall}) together with lower volume bounds for holomorphic curves (cf.~Proposition \ref{arealowerbound}) then shows that $E$ would have infinite volume. This contradicts the fact that $E$, as a limit of Type I rescalings of $\widetilde{E}$, has finite volume, leading to the desired conclusion. In particular, since the singularities of $X$ are isolated and are contained in $E$, $X$ has only finitely many orbifold singularities.
 
In Section 4, we first show in Proposition  \ref{nosteadies} that every steady gradient K\"ahler-Ricci soliton surface that arises as a Ricci flow singularity model is ALE Calabi-Yau. This is used in the proof 
of Theorem \ref{bddcurve}, where $X$ is shown to have bounded curvature. To this end, one 
considers the curvature blowup rates of $X$
near spatial infinity. If the curvature blows up faster than the distance squared from a basepoint, then a rescaled pointed limit at infinity is a non-flat ALE Calabi-Yau manifold due to the finiteness of the orbifold singularities established in Section 3. Thus, there is a corresponding domain in $X$ diffeomorphic to an ALE Calabi-Yau manifold asymptotic at infinity to $\mathbb{C}^2/\Gamma$ for some $\Gamma \subseteq U(2)$. { Thus, an unbounded number of disjoint copies of this ALE manifold embed smoothly into $\widetilde{M}$. A topological argument due to \cite{CFSZ} then yields a contradiction.}
 
If the curvature blows up precisely at a quadratic rate, then a rescaled pointed limit is a non-collapsed steady gradient K\"ahler-Ricci soliton. 
Proposition \ref{nosteadies} rules out this case. If the curvature blows up slower than quadratically, then we show that a limit at infinity splits isometrically. Appealing to the classification of three-dimensional $\kappa$-solutions rules out this case as well. In conclusion, we find that $X$ has bounded curvature. In passing, we note that the proof of Theorem \ref{bddcurve} can be combined with the fact that smooth shrinking gradient K\"ahler-Ricci solitons are quasiprojective \cite[Theorem 1.1]{JunshengSong} to give a new proof of the fact that all shrinking gradient K\"ahler-Ricci soliton surfaces have bounded curvature
\cite{LiWang}. The details we provide in Theorem \ref{2d}.

In Section 5, we prove that $X$ is smooth. To achieve this, we first recall the estimates from Section 3 which imply that the function $|\nabla^{\widetilde{g}_{t}} \widetilde{u}_t|_{\widetilde{g}_{t}}$ on $\widetilde{M}$ is bounded away from $\widetilde{E}$ after a Type I rescaling. Using this, we show in Proposition \ref{ExcSet} that $\widetilde{u}_t$ limits to a function $u: X\to [-\infty,\infty)$ with $u+f$ plurisubharmonic everywhere and strictly plurisubharmonic away from $E$. This forces every compact analytic curve in $X$ to be contained in $E$. In Proposition \ref{nocones}, we apply techniques from \cite{CCD} and the curvature bound established in Section 4 to ascertain that $X$ has quadratic curvature decay, hence is asymptotic at spatial infinity to a K\"ahler cone with smooth link, the Remmert reduction of $X$. It follows from the above that $E$ is precisely the exceptional set of the Remmert reduction. On the other hand, we show that the asymptotic cone is biholomorphic to $\mathbb{C}^2$, so in particular, if $E = \emptyset$, then $X$ would be isometric to Euclidean space \cite{CDS}. This would contradict a result of Bamler \cite[Theorem 2.37]{Bam3} characterizing the singular set of a Ricci flow, and so we deduce that $E\neq\emptyset$.

We conclude Section 5 by showing in Proposition \ref{minusone} that the proper transform of $E$ under the minimal resolution of $X$ comprises finitely many disjoint $(-1)$-curves (so that two irreducible components of $E$ necessarily  intersect at a singular point of $X$), and in Proposition \ref{irreducible}(i) 
that $E$ is irreducible. The former is proved using algebraic-geometric methods, whereas for the latter, we use a linking number argument to show that if two irreducible components of $E$ intersect (necessarily at an orbifold point), then there are two distinct second homology classes of $\widetilde{M}$, each positive multiples of $[\widetilde{E}]$, intersecting positively. This contradicts the fact that they both lie in a set diffeomorphic to the unit disk bundle in $\mathcal{O}_{\mathbb{P}^1}(-1)\to\mathbb{P}^{1}$. Using similar methods and the classification of minimal symplectic fillings of spherical sphere forms, we show in Proposition \ref{irreducible}(ii) that all of the singularities of $X$ are Wahl singularities. In particular, the exceptional set of the minimal resolution of these singularities has a well-known characterization \cite{Kollar}. Combined with the above, we can identify 
the exceptional set of the minimal 
resolution $\widetilde{X}$ of $X$. On the other hand, we know that the minimal resolution is a composition of blowups of $\mathbb{C}^{2}$. These facts allow us to compute the trace of the intersection matrix of the blowdown map $\widetilde{X}\to X\to \mathbb{C}^2$ in two different ways, and in doing so, derive a contradiction to $X$ having singular points. This is carried out in Proposition \ref{ShrinkerisFIK}. The results of \cite{CDS} now allow us to identify $X$ as the two-dimension shrinking K\"ahler-Ricci soliton of Feldman-Ilmanen-Knopf \cite{FIK}. Combined with the Type I curvature bound for $\widetilde{M}$ away from $\widetilde{E}$ from Section 3, we finally obtain Theorem \ref{noncollapsedtheorem}.

In Appendix A, we give a short proof of the fact that for any sequence of smoothly converging K\"ahler structures $(g_i,J_i)\to (g,J)$, there exist $J_i$-holomorphic coordinates converging locally smoothly to $J$-holomorphic coordinates. This is used in Sections 3 and 5, where it allows us to apply certain known theorems concerning sequences of holomorphic curves with respect to a fixed K\"ahler structure to the setting of smoothly converging K\"ahler structures. In Appendix B, we prove a technical result that ensures that certain holomorphic curves can be cut in a topologically controlled manner, allowing such curves to be capped by smooth 2-chains with controlled area. This is used in the proof of Proposition \ref{irreducible}.

\subsection{Acknowledgements} The authors would like to thank Bennett Chow and Pak-Yeung Chan for organizing the Workshop on Ricci flow in May 2023 at UCSD. This paper would not have been possible without the lively discussions with the participants during the workshop. The authors would also like to thank Richard Bamler, Charles Cifarelli, Ilyas Khan, Chi Li, Alec Payne, Jian Song, Song Sun, and Junsheng Zhang for many helpful discussions. We are grateful to Charlie Cifarelli, Jiangtao Li, and Yuan Liao for their useful comments and corrections, and to Junsheng Zhang for pointing out a mistake in Lemma \ref{intersectiongraph}. The second author is supported in part by the National Science Foundation under Grant No.~DMS-2202980. 

This material is based upon work supported by the National Science Foundation under Grant No.~DMS-1928930 while the first author was in residence at the Simons
Laufer Mathematical Sciences Institute (formerly MSRI) in Berkeley, California, during the Fall 2024 semester. He wishes to thank SLMath for their excellent working conditions and hospitality during this time.

\section{Preliminaries}

\subsection{Ricci flow convergence theory}
Let $(M^n,(g_t)_{t\,\in \,[0,\,T)},J),\,T<\infty,$ be a K\"ahler-Ricci flow on a compact K\"ahler manifold of dimension $n$. Throughout, we follow notation from \cite{Bam3} and write $\mathbb{S}^{3}$ for the unit sphere in $\mathbb{R}^{4}$.

\subsubsection{Basic definitions and estimates}
For a Riemannian manifold $(M,\,g)$, we let $R_{g}$, $\Ric_{g}$, and $\Rm_{g}$ denote the scalar curvature, Ricci curvature, and Riemannian curvature tensor, respectively.
For a K\"ahler manifold $(M,\,g,\,J)$, we write $\omega:= g(J\cdot\,,\cdot)$ for the K\"ahler form and $\Ric_{\omega}:=\Ric_{g}(J\cdot\,,\cdot)$ for the Ricci form. 

Given $(x,t)\in M\times [0,\,T)$, let $K(x,t;\cdot\,,\cdot):M\times [0,t)\to(0,\infty)$ denote the conjugate heat kernel based at $(x,\,t)$. Define probability measures $d\nu_{x,t;s}:=K(x,\,t;\cdot\,,s)\,dg_s$, where $dg_s$ denotes the Riemannian volume measure of $(M,\,g_s)$ for $s\in [0,\,t)$. We also use the notation
$$B(x,t,r):= \{y\in M\,|\, d_t(x,y)<r\}$$
for the geodesic ball with respect to the Riemannian distance $d_t:= d_{g_t}$ of $(M,\,g_t)$. For $(x,\,t)\in M\times[0,T)$, define the curvature scale
$$r_{{\Rm}}(x,t):=\sup\{r>0\,|\,|{\Rm}|\leq r^{-2}\textrm{ on }B(x,t,r)\times([t-r^{2},t+r^{2}]\cap[0,T))\}.$$
If $(M,(g_t)_{t\in [0,T)})$ is instead a smooth complete orbifold K\"ahler-Ricci flow, the definition of $B(x,t,r)$ still makes sense, so we can define $r_{\Rm}(x,t)$ denote the supremum of $r>0$ such that both \linebreak $B(x,t,r)\cap X_{\textnormal{reg}}=\emptyset$ and $|{\Rm}|\leq r^{-2}$ on $B(x,t,r)\times [t-r^2,t+r^2]\cap [0,T)$. Given $t_0 \in [0,T)$ and a curve $\gamma:[t_1,t_2]\to M$ with $t_2\le t_0$, recall the definition of the $\mathcal{L}$-length:
\begin{equation} \label{Llengthdefinition} \mathcal{L}_{t_0}(\gamma):= \int_{t_1}^{t_2} \sqrt{t_0-t}\left( R_{g_t}(\gamma(t))+|\dot{\gamma}(t)|_{g_t}^2 \right)dt.\end{equation}
Given $(p,t_0)\in M\times [0,T)$, the reduced distance $\ell_{(p,t_0)}:M\times [0,t_0) \to \mathbb{R}$ based at $(p,t_0)$ is defined by
$$\ell_{(p,t_0)}(x,t) := \inf \left\{ \mathcal{L}_{t_0}(\gamma) \: | \: \gamma:[t,t_0] \to M \text{ is a piecewise smooth curve with } \gamma(t) = x, \gamma(t_0)=p \right\}.$$
The reduced distance was first defined in \cite{perelman}. The reader is directed to \cite[Section 7.6]{Chow1} for a summary of its most important properties. If $(M,\,g)$ is a (possibly incomplete) Riemannian manifold, let $r_{{\Rm}}^{(M,g)}(x)$ denote the supremum over all $r>0$ such that $B_{g}(x,r)$, the ball of radius $r$ centered at $x$ with respect to $g$, has compact closure in $M$ and $|{\Rm}|_{g}\leq r^{-2}$ in $B_{g}(x,r)$.

A $P^{\ast}$-parabolic neighborhood \cite[Definition 9.2]{Bam1} centered at $(x_0,t_0)\in M \times [0,T)$ is defined as follows: $P^{\ast}(x_0,t_0;A,-T^-,T^+)$ is the set of $(x,t) \in M\times [0,T)$ with $t \in [t_0 - T^-,t_0+T^+]$ and
$$d_{W_1}^{g_{t_0-T^-}}(\nu_{x_0,t_0;t_0-T^-},\nu_{x,t;t_0-T^-})<A,$$
where $d_{W_1}^{g_t}$ is the 1-Wasserstein distance between probability measures on the metric space $(M,d_{g_t})$ (cf.~\cite[Section 2]{Bam2}). 

We next summarize the relevant points of Bamler's weak compactness and partial regularity theory that we need. We omit some definitions for the sake of brevity, but will refer the reader to the appropriate sections in \cite{Bam2,Bam1,Bam3}. The notation in this section will be used throughout.

Given $x_0 \in M$ and any sequence of times $T_i \nearrow T$, we can pass to a subsequence to obtain convergence in $C_{\operatorname{loc}}^{\infty}(M\times (0,T))$ to a positive solution $$K(x_0,\,T;\cdot,\,\cdot):= \lim_{i\to \infty} K(x_0,\,T_i;\cdot,\,\cdot):M \times [0,T)\to\mathbb{R}$$ of the conjugate heat equation \cite[Lemma 2.2]{mantegazza} such that if $d\nu_{x_0,T;t} = K(x_0,T;\cdot,t)\,dg_t$ for $t\in [0,T)$, then
$$\lim_{i\to \infty} \sup_{t\in [0,T-\epsilon]} d_{W_1}^{g_t}(\nu_{x_0,T_i;t},\nu_{x_0,T;t})=0$$
for any $\epsilon>0$ \cite[Lemma 2.35]{Bam3}. By the $W_1$-convergence $\nu_{x_0,T_i;t}\to \nu_{x_0,T;t}$, we moreover have that 
$$\int_{M} \int_{M} d_t^2(x,y)\,d\nu_{x_0,T;t}(x)\, d\nu_{x_0,T;t}(y) \leq H_{2n}(T-t),$$
where $H_n:= \frac{(n-1)\pi^2}{2}+4$. This implies that for each $t\in [0,T)$, there exists $z \in M$ such that $\int_X d_t^2(z,y) d\nu_{x_0,T;t} \leq H_{2n}(T-t)$ \cite[Section 3]{Bam1}. The point $(z,\,t)$ is called an \emph{$H_{2n}$-center of $(x_0,T)$}. 

We now recall for future reference some heat kernel estimates proved in \cite{Bam1}. 
\begin{prop} 
[\protect{\textnormal{\cite[Theorem 7.2 \& Theorem 7.5]{Bam1}}}]\label{bamheatkernel} 
Suppose $(M^n,\,(g_t)_{t\in [0,T)})$ is a Ricci flow on a closed $n$-dimensional manifold. Then there exists $C>0$ such that for all $(x,t)\in M\times [0,T)$, $(y,s)\in M\times [0,t)$, and for any $H_n$-center $(z,s)$ of $(x,t)$, the following holds.
\begin{enumerate}
    \item $K(x,t;y,s)\leq \frac{C}{(t-s)^{\frac{n}{2}}} \exp \left( -\frac{d_s^2(z,y)}{9(t-s)} \right)$;

    \item $|\nabla_x K(x,t;y,s)|_{g_t} \leq \frac{C}{(t-s)^{\frac{n+1}{2}}} \exp \left( -\frac{d_s^2(z,y)}{10(t-s)} \right)$.
\end{enumerate}
\end{prop}

\begin{proof} 
\begin{enumerate}
    \item This follows directly from \cite[Theorem 7.2]{Bam1} using the Nash entropy bound \cite[Proposition 4.35]{Bam3}. 
    \item By \cite[Theorem 7.5]{Bam1} and the Nash entropy bound \cite[Proposition 4.35]{Bam3}, there exists $C>0$ such that
    \begin{align*} |\nabla_x K(x,t;y,s)|_{g_t} &\leq \frac{C}{\sqrt{t-s}} \sqrt{ \log\left( \frac{C}{(t-s)^{\frac{n}{2}}K(x,t;y,s)} \right)}K(x,t;y,s) \\
    &= \frac{C^2}{(t-s)^{\frac{n+1}{2}}}\eta \left( C^{-1}(t-s)^{\frac{n}{2}}K(x,t;y,s) \right) \left( C^{-1} (t-s)^{\frac{n}{2}}K(x,t;y,s) \right)^{\frac{9}{10}},
    \end{align*}
where for $\rho \in (0,1)$, we define
$$\eta(\rho):=\rho^{\frac{1}{10}}\sqrt{\log(1/\rho)}.$$
Because $\eta$ is bounded, the claim follows from (i). 
\end{enumerate}
\end{proof}

The notion of pointed Nash entropy is used to measure the non-collapsedness of a Ricci flow solution. We next recall the definition. 

\begin{definition} 
[\protect{\textnormal{\cite[Section 2.6]{Bam3}}}]
\label{conjsingtime} $K(x_0,T;\cdot,\cdot)$ and $\nu_{x_0,T;t}$ are called \emph{conjugate heat kernels based at $(x_0,T)$}. Define $f \in C^{\infty}(M\times [0,t))$ by $K(x_0,T;\cdot\,,t)=(2\pi \tau)^{-n}e^{-f}\,dg_t$, where $\tau:= T-t$. Then the \textit{pointed Nash entropy} at $(x_0,T)$ is
given by $$\mathcal{N}_{x_0,T}(\tau):= \int_M f\,d\nu_{x_0,T;t}-n.$$
\end{definition}

\begin{remark} Note that $K(x_0,T;\cdot,\cdot)$ and $\nu_{x_0,T;t}$ may in general depend on the sequence $T_i \nearrow T$. However, $\mathcal{N}_{x_0,T}(\tau)$ is independent of the sequence $T_i \nearrow T$ \cite[Corollary 5.11]{Bam1}, so that
$$\mathcal{N}_{x_0,T}(0) := \lim_{\tau \searrow 0} \mathcal{N}_{x_0,T}(\tau) \in (-\infty,0]$$
depends only on $x_0 \in M$.
\end{remark}

\subsubsection{Ricci flow spacetimes}

The notion of a Ricci flow spacetime was introduced in \cite{KLSing} in the context of three-dimensional Ricci flow, but was shown in \cite{Bam2} to be the natural structure on the regular part of a singular limit of Ricci flows. We make a modification of this definition from \cite{HJ} so as to take into account the K\"ahler structure.

\begin{definition}[\protect{\textnormal{\cite[Definition 1.2]{KLSing}}}] \label{KRFspacetimedef}
A \emph{K\"ahler-Ricci flow spacetime} is a tuple $(\mathcal{M},\mathfrak{t},\partial_{\mathfrak{t}},g,\mathcal{J})$ comprising the following data.
\begin{enumerate} 
    \item A smooth manifold $\mathcal{M}$;
\item A smooth submersion $\mathfrak{t} \in C^{\infty}(\mathcal{M})$;
\item A vector field $\partial_{\mathfrak{t}}\in \mathfrak{X}(\mathcal{M})$ with $\partial_{\mathfrak{t}}\mathfrak{t}=1$;
\item A bundle metric $g$ on the subbundle $\ker(d\mathfrak{t})\subseteq T\mathcal{M}$ satisfying $\mathcal{L}_{\partial_{\mathfrak{t}}}g=-{\operatorname{Ric}_{g}}$, where $\mathcal{M}_t:= \mathfrak{t}^{-1}(t)$, and ${\operatorname{Ric}_{g}}|_{\mathcal{M}_t}$ is defined to be the Ricci curvature of $g_t := g|_{\mathcal{M}_t}$;
\item A bundle endomorphism $\mathcal{J}$ on $\ker(d\mathfrak{t})$ satisfying $\mathcal{L}_{\partial_{\mathfrak{t}}}\mathcal{J}=0$ that restricts to an almost-complex structure $\mathcal{J}_t$ on each $\mathcal{M}_t$ such that $(\mathcal{M}_t, g_t,\mathcal{J}_t)$ is a K\"ahler manifold.
\end{enumerate}
\end{definition}  

We now define a K\"ahler-Ricci flow spacetime structure on the regular set of certain K\"ahler orbifolds. We refer the reader to  \cite[Chapter 4]{sasakian} for the definition of a K\"ahler orbifold and an overview of their basic properities. Given any orbifold $X$, let $X_{\textnormal{reg}} \subseteq X$ denote the (dense, open) set of points whose isotropy groups are trivial. 

\begin{definition} An \emph{orbifold shrinking gradient K\"ahler-Ricci soliton} $(X,\,g,\,J,\,f)$ is a K\"ahler orbifold $(X,\,g)$ with complex structure $J$, together with a smooth real-valued function $f\in C^{\infty}(X)$ satisfying $${\operatorname{Ric}_{g}}+\nabla^2f=g,\qquad\mathcal{L}_{\nabla f}J=0.$$
We say $X$ has \emph{isolated singularities} if $X \setminus X_{\textnormal{reg}}$ is discrete. The isotropy group of an isolated singularity is a non-trivial finite subgroup of $U(n)$ acting freely on $\mathbb{C}^{n}\setminus\{0\}$. 
\end{definition}

\begin{remark} \label{holoflow} If $(X,\,g,\,J,\,f)$ is an orbifold gradient shrinking K\"ahler-Ricci soliton with isolated singularities, then $\nabla f$ vanishes at each point of $X \setminus X_{\textnormal{reg}}$ because the isotropy groups fix the lift of $\nabla f$ in any orbifold chart and the isotropy action on $\mathbb{C}^{n}\setminus\{0\}$ is free. We can therefore define a one-parameter family of biholomorphisms of $(X,\,J)$ that preserves $X_{\textnormal{reg}}$ by defining $\varphi_{-1}=\text{id}_X$ and
$$\partial_t \varphi_t(x) = \frac{1}{|t|}\nabla f(\varphi_t(x)).$$
Setting $g_t := |t|\varphi_t^{\ast}g$, the one-parameter family of metrics $(X,(g_t)_{t\in (-\infty,0)},\,J)$ defines an orbifold K\"ahler-Ricci flow.
\end{remark}

By this remark, any orbifold gradient shrinking K\"ahler-Ricci soliton corresponds to a natural K\"ahler-Ricci flow spacetime, which we now define.
\begin{definition}\label{standardsoliton}Given an orbifold shrinking gradient K\"ahler-Ricci soliton $(X,g,J,f)$ with isolated singularities, the \emph{induced K\"ahler-Ricci flow spacetime} $(\mathcal{R},\mathfrak{t},\partial_{\mathfrak{t}},g,\mathcal{J})$ is defined by the following data.
\begin{enumerate}
    \item $\mathcal{R} := X_{\textnormal{reg}} \times (-\infty,0)$;
\item $\mathfrak{t}:\mathcal{R}\to (-\infty,0)$ is projection onto the second factor;
\item $\partial_{\mathfrak{t}} \in \mathfrak{X}(\mathcal{R})$ is the standard vector field on $(-\infty,0)$;
\item $g:= \mathfrak{t} \Phi^{\ast}g$, where $\Phi(x,t):= \varphi_t(x)$;
\item $\mathcal{J}_t = J$ for all $t\in (-\infty,\,0)$.
\end{enumerate}
\end{definition}

\subsubsection{Compactness and partial regularity theory}

For the following theorem, we fix a conjugate heat kernel $(\nu_{x_0,T;t})_{t\in [0,T)}$ based at $(x_0,T)$. Metric flows are defined in \cite[Definition 3.2]{Bam2} and metric solitons \cite[Definition 3.57]{Bam2} are a special class of metric flows. The natural topology of a metric flow is defined in \cite[Definition 3.32]{Bam2} and the notions of correspondence and $\mathbb{F}$-convergence are defined in \cite[Section 5.1]{Bam2}. 

\begin{theorem}
[\protect{\textnormal{\cite[Theorem 2.37]{Bam2} \& \cite[Theorem 2.5]{HJ}}}]
\label{bamconvergence} Let $(\widetilde{M}^2,(\widetilde{g}_t)_{t\in [0,T)},\widetilde{J})$ be a K\"ahler-Ricci flow on a compact K\"ahler surface $(\widetilde{M}^2,\,\widetilde{J})$ with $T<\infty$. Given any sequence $t_i \nearrow T$, set $\widetilde{g}_{i,t}:=(T-t_i)^{-1}\widetilde{g}_{T+(T-t_i)t}$ and $\widetilde{\nu}_t^i := \nu_{x_0,T;T+(T-t_i)t}$ for $t\in [-(T-t_i)^{-1}T,0)$. After passing to a subsequence, there exists a metric soliton $(\mathcal{X},(\nu_{x_{\infty};t})_{t\in (-\infty,0]})$ along with a correspondence $\mathfrak{C}$ such that the following $\mathbb{F}$-convergence within the correspondence on compact time intervals holds:
$$(\widetilde{M},(\widetilde{g}_{i,t})_{t\in (-(T-t_i)^{-1}T,0)},(\widetilde{\nu}_t^i)_{t\in (-(T-t_i)^{-1}T,0)})\xrightarrow[i\,\to\,\infty]{\mathbb{F},\,\mathfrak{C}} (\mathcal{X},(\nu_{x_{\infty};t})_{t\,\in\,(-\infty,0)}).$$
Moreover, the metric flow pair $({\mathcal{X}},({\nu}_{x_{\infty};t})_{t\in (-\infty,0)})$ is modeled on an orbifold shrinking gradient K\"ahler-Ricci soliton $(X,\,g_X,\,J_X,\,f_X)$ with isolated singularities in the following precise sense:
\begin{enumerate}
    \item ${\mathcal{X}} = X \times (-\infty,0)$, where the natural topology and product topology coincide;
\item The K\"ahler-Ricci flow spacetime structure $({\mathcal{R}},{\mathfrak{t}},{\partial_{\mathfrak{t}}},{g},{\mathcal{J}})$ coincides with the induced spacetime $X_{\textnormal{reg}}\times (-\infty,0)$ associated to $(X,{g}_X,{J}_X,{f}_X)$;
\item Each time slice $({\mathcal{X}}_t,{d}_t)$ is the completion of the Riemannian length metric on $(X_{\textnormal{reg}},d_{g_t})$;
\item $d{\nu}_{x_{\infty};t}=(2\pi|t|)^{-2}e^{-{f}}d{g}_t$ for some $f\in C^{\infty}(\mathcal{R})$ which is identified with $\Phi^{\ast}f_X$ under the above correspondence, where $\Phi$ is as in Definition \ref{standardsoliton};
\item There is an exhaustion of ${\mathcal{R}}$ by precompact open sets $(U_i)$  along with time-preserving open embeddings $\psi_i :U_i \to \widetilde{M}_i \times (-\infty,0)$ such that the following hold:
\begin{enumerate}[label=\textnormal{(\alph{*})}, ref=(\alph{*})]
    \item $\psi_i^{\ast}\widetilde{g}_i \to {g}$ in $C_{\operatorname{loc}}^{\infty}({\mathcal{R}})$;
\item  $(\psi_i^{-1})_{\ast} \partial_t \to {\partial_{\mathfrak{t}}}$ in $C_{\operatorname{loc}}^{\infty}({\mathcal{R}})$;
\item  $\psi_i^{\ast}\widetilde{J} \to {\mathcal{J}}$ in $C_{\operatorname{loc}}^{\infty}({\mathcal{R}})$;
\item  If we write $d\widetilde{\nu}_t^i = (2\pi |t|)^{-2}e^{-\widetilde{f}_i}d\widetilde{g}_{i,t}$, then $\psi_i^{\ast}\widetilde{f}_i\to f$ in $C_{\operatorname{loc}}^{\infty}({\mathcal{R}})$;
\end{enumerate}
\item $|t|(R_g+|\nabla f|_g^2) = {f} - W$ on ${\mathcal{R}}$, where $W:=\mathcal{N}_{x_0,T}(0)$.
\end{enumerate}
We call the orbifold K\"ahler-Ricci flow $(X,(g_t)_{t\in [0,T)},J)$ a tangent flow of $(\widetilde{M},(\widetilde{g}_t)_{t\in [0,T)},\widetilde{J})$.
\end{theorem}

\begin{proof} By \cite[Theorems 2.37]{Bam3}, we can pass to a subsequence to obtain a metric soliton $(\mathcal{X},(\nu_{x_{\infty};t})_{t\in (-\infty,0]})$ satisfying (iii), and a correspondence $\mathfrak{C}$ realizing the desired $\mathbb{F}$-convergence. By \cite[Theorem 15.69  \& 2.46]{Bam3}, there is an orbifold shrinking gradient K\"ahler-Ricci soliton $(X,g_X,J_X,f_X)$, along with a time-preserving homeomorphism \nolinebreak $\Theta: \mathcal{X} \to X\times (-\infty,0)$ whose restrictions $\Theta_t:(\mathcal{X}_t,|\mathfrak{t}|^{-\frac{1}{2}}d_t)\to (X,d_{g_X})$ are isometries, where we identify $d_X$ with its pullback by the projection map $X\times (-\infty,0)\to X$. Let $(\mathcal{R},\mathfrak{t},\partial_{\mathfrak{t}},g)$ be the Ricci flow spacetime structure on the regular part of $\mathcal{X}$. By \cite[Theorem 15.69]{Bam3} and \cite[Theorem 2.5]{HJ}, there is a function $f\in C^{\infty}(\mathcal{R})$ such that (vi) holds, there exists $\mathcal{J}$ as in Definition \ref{KRFspacetimedef}(v) making $(\mathcal{R},\mathfrak{t},\partial_{\mathfrak{t}},g,\mathcal{J})$ a K\"ahler-Ricci flow spacetime, and there is an open exhaustion $(U_i)$ of $\mathcal{R}$ and time-preserving open embeddings $\psi_i$ satisfying (v). Moreover, $\Theta$ restricts to a diffeomorphism $\mathcal{R}\to X_{\text{reg}}\times (-\infty,0)$ with $\Theta^{\ast}t=\mathfrak{t}$, $\Theta^{\ast}g_X=|\mathfrak{t}|^{-1}g$, $\Theta_{\ast}(\partial_{\mathfrak{t}}-\nabla f)=\partial_t$, $\Theta^{\ast}f_X=f$, and $\Theta^{\ast} J_X = \mathcal{J}$, where we identify $g_X,J_X,f_X$ with their pullbacks by the projection $X\times (-\infty,0)\to X$ and where $\partial_t \in \mathfrak{X}(X_{\text{reg}}\times (-\infty,0))$ is the standard vector field on the second factor. 

On the other hand, if we define $\Psi:X\times (-\infty,0)\to X \times (-\infty,0)$ by $(x,t)\mapsto (\Phi(x,t),t)$, then $\Psi$ is a diffeomorphism of orbifolds which restricts to a diffeomorphism of manifolds $X_{\text{reg}}\times (-\infty,0) \to X_{\text{reg}}\times (-\infty,0)$. Moreover, if $f':=\Phi^{\ast}f_X$ and $(\mathcal{R}',\mathfrak{t}',\partial_{\mathfrak{t}'},g',\mathcal{J}')$ is the K\"ahler-Ricci flow spacetime induced by $(X,g_X,J_X,f_X)$ as in Definition \ref{standardsoliton}, then we have that $\Psi^{\ast}t=\mathfrak{t}'$, $\Psi^{\ast}g_X=|\mathfrak{t}'|^{-1}g'$, $\Psi_{\ast}(\partial_{\mathfrak{t}}'-\nabla f')=\partial_t$, and $\Psi^{\ast}J_X=\mathcal{J'}$. It follows that $\Psi^{-1} \circ \Theta$ is an identification of $\mathcal{X}$ with $X\times (-\infty,0)$ such that (i),(ii),(iv)--(vi) hold, and (iii) holds as well under this identification by \cite[Theorem 2.4]{Bam3}.

It remains only to verify that $(X,\,J_X)$ is a complex orbifold. Indeed, \cite[Theorem 2.46]{Bam3} only guarantees that $X$ is a smooth Riemannian orbifold. We must show that the orbifold charts at the singular points may be chosen to be compatible with the complex structure on $X_{\textnormal{reg}}$, or equivalently, that for any $x_{\ast} \in X\setminus X_{\textnormal{reg}}$, there is a smooth orbifold chart $\check{U}\to U\subset X$ centered at $x_{\ast}$ such that the pullback of $J_X$ to $\check{U}\setminus \{0\}$ extends smoothly to a complex structure on $\check{U}$. To this end, let $\eta:\widehat{B}\to B$ be a smooth orbifold chart centered at $x_{\ast}$, where $\widehat{B} \subseteq \mathbb{C}^2$ is a ball. Let $\Gamma \subseteq O(4,\mathbb{R})$ be the (finite) isotropy group at $x_{\ast}$, $\widehat{J}_{X}$ the natural $\Gamma$-invariant complex structure on $\widehat{B} \setminus \{0\}$ that descends to the complex structure $J_X$ on $B\setminus \{ x_{\ast} \}$, and set $\widehat{g}_{X}:= \eta^{\ast}g_X$ on $\widehat{B}$. For fixed $z_0 \in \widehat{B} \setminus \{0\}$, let $\overline{J}$ be the $\widehat{g}_X$-parallel transport of $\widehat{J}_X|_{z_0}$ along straight lines emanating from $z_0$. This defines a smooth complex structure  
on $\widehat{B}$ that agrees with $\widehat{J}_{X}$ on a dense open subset (the complement of a line)
by standard ODE theory, hence by continuity agrees on all of $\widehat{B}\setminus \{0^2\}$.  $\overline{J}$ is the desired extension of $\widehat{J}_{X}$ over $0$ with $\nabla^{\widehat{g}_X}\overline{J}\equiv 0$ on $\widehat{B}$. An application of the Newlander-Nirenberg theorem then gives the desired holomorphic orbifold chart. 
\end{proof}

\subsection{$H_n$-centers and holomorphic maps}

Let $(\widetilde{M}^n,(\widetilde{g}_t)_{t\in [0,T)}),\,T<\infty,$ be a K\"ahler-Ricci flow on a compact $n$-dimensional K\"ahler manifold $\widetilde{M}^n$ developing a finite time singularity at $t=T$ and suppose that there exists a holomorphic map $\pi:\widetilde{M} \to N$ onto a compact K\"ahler manifold $(N,\,g_{N})$ of possibly lower dimension.
Let $\widetilde{\omega}_t$ and $\omega_{N}$ denote the K\"ahler form of $\widetilde{g}_{t}$ and $g_{N}$, and suppose
that $[\widetilde{\omega}_0]-Tc_1(\widetilde{M})=\pi^{\ast} [\omega_N]$ in $H^{1,\,1}(\widetilde{M},\mathbb{R})$. The following is a consequence of the parabolic Schwarz lemma.

\begin{lemma}
[\protect{\cite[Lemma 3.7.3]{songnotes}}] \label{schwartz} 
There exists $C_0 >0 $ such that for all $t\in [0,T)$,
$$\widetilde{\omega}_t \geq \frac{1}{C_0}\pi^{\ast}\omega_{N}.$$
In particular, $\pi:(\widetilde{M},\widetilde{g}_t) \to (N,g_{N})$ is $\sqrt{C_0}$-Lipschitz for each $t\in [0,\,T)$.
\end{lemma}

Using this estimate, we next show that the image under $\pi$ of the $H_{2n}$-centers of conjugate heat kernels on $\widetilde{M}$ are close to those of their basepoints. Note that this result also applies to conjugate heat kernels based at the singular time. The proof uses the argument of \cite[Lemma 6.7]{JST}.  

\begin{lemma} \label{nodrift}
There exists $A>0$ such that for all $(x_0,t_0)\in \widetilde{M}\times [0,T]$ and $t\in [0,t_0)$, any $H_{2n}$-center $(z_{t},t)$ of 
$(x_0,t_0)$
is contained
in $\pi^{-1}\left(B_{g_N}(\pi(x_0),A\sqrt{t_0-t})\right)$. 
\end{lemma}
\begin{proof} First let $t_0 <T$ and let $(z_t',t)$ be an $\ell$-center of $(x_0,t_0)$, by which we mean
$\ell_{(x_{0},t_0)}(z_t',t)\leq n$. Such a point exists by \cite[Lemma 7.50]{Chow1}. Let $\gamma:[t,t_0]\to \widetilde{M}$ be a minimizing $\mathcal{L}_{t_0}$-geodesic
from $(x_{0},t_0)$ to $(z_t',t)$. 
Then by \eqref{Llengthdefinition}, together with Lemma \ref{schwartz}, followed by an application of
Cauchy's inequality, we find that
\begin{equation*}
\begin{split}
n\geq &\  \frac{1}{2\sqrt{t_0-t}}\mathcal{L}_{t_0}(\gamma) \\ = &\ \frac{1}{2\sqrt{t_{0}-t}}\int_{t}^{t_{0}}
\sqrt{t_0
-s}\left(R_{\widetilde{g}_{s}}(\gamma(s))+|\dot{\gamma}(s)|_{\widetilde{g}_{s}}^{2}\right)ds\\
\geq & \ T\min_{\widetilde{M}}R_{\tilde{g}_{0}}+\frac{1}{2C_0^2\sqrt{t_{0}-t}}\int_{t}^{t_{0}}\sqrt{t_0-s}|\dot{\gamma}(s)|_{\pi^{\ast}g_N}^{2}ds\\
= & -C(T,\widetilde{g}_{0})+\frac{1}{4C_0^2(t_{0}-t)}\left(\int_{t}^{t_{0}}\frac{1}{\sqrt{t_0-s}}ds\right)\left(\int_{t}^{t_{0}}\sqrt{t_{0}-s}|\dot{\gamma}(s)|_{\pi^{\ast}g_N}^{2}ds\right)\\
\geq & -C(T,\widetilde{g}_{0})+\frac{1}{4C_0^2(t_{0}-t)}\left(\int_{t}^{t_{0}}|\dot{\gamma}(s)|_{\pi^{\ast}g_N}ds\right)^{2},
\end{split}
\end{equation*}
where $C_0$ is as in Lemma \ref{schwartz}. Because $\pi\circ\gamma$ is a curve from $\pi(x_0)$ to $\pi(z_t')$, it follows that
\begin{equation*}
\begin{split}
d_{g_N}(\pi(z_t'),\pi(x_0))\leq & \ \operatorname{length}_{g_N}(\pi\circ\gamma)=\int_{t}^{t_0}|\dot{\gamma}(s)|_{\pi^{\ast}g_N}\,ds\\
\leq &\ C(T,\widetilde{g}_{0},C_0)\sqrt{t_0-t}.
\end{split}
\end{equation*}
By \cite[Proposition 5.6]{CMZ}, there exists $C=C(\widetilde{g}_0)>0$ such that any $H_{2n}$-center
$(z_t,t)$ of $(x_0,t_0)$ satisfies $d_{\widetilde{g}_{t}}(z_t',z_t)\leq C\sqrt{t_0-t}$. Because $\pi:(\widetilde{M},\widetilde{g}_{t})\to(N,g_N)$ is $\sqrt{C_0}$-Lipschitz by Lemma \ref{schwartz}, this yields the bound
$$d_{g_{N}}(\pi(z_t'),\pi(z_t))\leq C(\widetilde{g}_0)\sqrt{C_0} \sqrt{t_0-t}.$$
Combining estimates gives the claim for $t_0<T$. If instead $t_0 = T$, then by Definition \ref{conjsingtime}, there exist $t_i \nearrow T$ such that $\lim_{i\to \infty} d_{W_1}^{\widetilde{g}_t}(\widetilde{\nu}_{x_0,t_i;t},\widetilde{\nu}_{x_0,T;t})=0$. Letting $(z_{i,t},t)$ be an $H_{2n}$-center of $(x_0,t_i)$ and $(z_t,t)$ an $H_{2n}$-center of $(x_0,T)$, we conclude that 
$$d_{\widetilde{g}_t}(z_t,z_{i,t})\leq d_{W_1}^{\widetilde{g}_t}(\delta_{z_t},\widetilde{\nu}_{x_0,t_i;t}) + d_{W_1}^{\widetilde{g}_t}(\widetilde{\nu}_{x_0,t_i;t},\widetilde{\nu}_{x_0,T;t}) + d_{W_1}^{\widetilde{g}_t}(\delta_{z_{i,t}},\widetilde{\nu}_{x_0,t_i;t})\leq C(n)\sqrt{T-t}$$ for all $i \geq \underline{i}(t)$, so because $\pi$ is uniformly Lipschitz, the claim follows from the case $t_0<T$.
\end{proof}

For the next result, we show that any tangent flow $X$ of a finite time singularity of a K\"ahler-Ricci flow as described above based at a point $x_0 \in \widetilde{M}$ corresponds to regions of the rescaled flow $(\widetilde{M},\widetilde{g}_{i,t})$ lying entirely in any given tubular neighborhood of the exceptional set. 

\begin{lemma} \label{containedinV} Let $(\widetilde{M}^2,(\widetilde{g}_t)_{t\in [0,T)}),\,T<\infty,$ be a K\"ahler-Ricci flow on a compact K\"ahler surface $\widetilde{M}^2$ developing a finite time singularity at $t=T$. Suppose that $\pi:\widetilde{M} \to N$ is a holomorphic map with connected fibers, where $(N,\,g_{N})$ is a compact K\"ahler manifold. Let $\mathcal{V} \subseteq \widetilde{M}$ be a neighborhood of a fiber $\widetilde{E}$ of $\pi$. Fix $x_0 \in \widetilde{E}$, and let $X$ be a tangent flow based at $(x_0,T)$ as in Theorem \ref{bamconvergence}. 
For any compact subset $L \subseteq X_{\operatorname{reg}} \times [-2,-1]$, the image $\psi_i(L) \subseteq \mathcal{V}\times \mathbb{R}$ for all $i \geq \underline{i}(L,\mathcal{V})$ sufficiently large.
\end{lemma}
\begin{proof} For each $i\in \mathbb{N}$ and $t\in [-2,-1]$, let $(z_{i,t},t)$ be an $H_4$-center of $(\widetilde{\nu}_{s}^i)_{s\in [-(T-t_i)^{-1}t_i,0)}$, which is the parabolic rescaling of the conjugate heat kernel $(\widetilde{\nu}_t)_{t\in [0,T)}$ based at $(x_0,T)$. As the next claim shows, the $H_4$-centers of $\widetilde{\nu}_s^i$ are not too far from the image $\psi_{i}(L)$. 

\begin{claim} \label{claim-usedincontainedinV} There exists $D=D(L)>0$ such that whenever $i\geq \underline{i}(L)$, it holds that
$$\sup_{(x,t)\in \psi_i(L)}d_{\widetilde{g}_t}(x,z_{i,t})\leq D(L).$$
\end{claim}

\begin{proof} Recalling from Theorem \ref{bamconvergence} that $d\widetilde{\nu}_t^i = (2\pi |t|)^{-2}e^{-\widetilde{f}_i}d\widetilde{g}_{i,t}$, we know that $\psi_i^{\ast}\widetilde{f}_i \to f$ in $C_{\operatorname{loc}}^{\infty}(X_{\textnormal{reg}}\times (-\infty,0))$, so for sufficiently large $i\geq \underline{i}(L)$, we have that
$$\sup_{\psi_i(L)} \widetilde{f}_{i} \leq 2\sup_{L} f<\infty.$$
By Proposition \ref{bamheatkernel}(i), we know that
$$\frac{1}{(2\pi |t|)^{2}} e^{-\widetilde{f}_{i,t}}(x) \leq \frac{C}{|t|^2} \exp \left( -\frac{d_{\widetilde{g}_{i,t}}^2(x,z_{i,t})}{10|t|} \right)$$
for all $x\in \widetilde{M}$ and $t\in [-2,-1]$.
Hence for all 
$(x,t)\in\psi_i(L),$
it follows that
$$2\sup_L f \geq \widetilde{f}_{i,t}(x)\geq \frac{d_{\widetilde{g}_{i,t}}^2(x,z_{i,t})}{20}-C$$
whenever $i\geq \underline{i}(L)$.
\end{proof}

Let $A>0$ be as in Lemma \ref{nodrift}. Because $\pi$ has connected fibers, we can fix a neighborhood $\mathcal{W}\subseteq \widetilde{M}$ of $\widetilde{E}$ such that $\mathcal{W} \Subset \mathcal{V}$ and $\pi(\mathcal{W})\cap \pi(\widetilde{M} \setminus \mathcal{V}) = \emptyset$. Moreover, we have $\pi^{-1}(B_{g_N}(\pi(x_0),r))\subseteq \mathcal{W}$ for sufficiently small $r>0$, so Lemma \ref{nodrift} gives that
$$z_{i,t} \in \pi^{-1}\left(B_{g_N}(\pi(x_0),A\sqrt{(T-t_i)|t|})\right) \subseteq \mathcal{W}$$
for all $t\in [-2,-1]$ when $i\geq \underline{i}(\mathcal{W})$ is sufficiently large. Then Claim \ref{claim-usedincontainedinV} gives
$$ \sup_{(x,t)\in \psi_i(L)} d_{\widetilde{g}_{i,t}}(x,\mathcal{W}) \leq D(L)$$
whenever $i\geq \underline{i}(L)$ is large. By Lemma \ref{schwartz}, we have that
$$d_{\widetilde{g}_{i,t}}(x,\mathcal{W}) = \frac{d_{\widetilde{g}_{T+(T-t_i)t}}(x,\mathcal{W})}{\sqrt{T-t_i}} \geq \frac{1}{\sqrt{C_0(T-t_i)}} d_{g_N}(\pi(x),\pi(\mathcal{W})).$$
Combining expressions then gives us that

$$\sup_{(x,t)\in\psi_i(L)} d_{g_N}(\pi(x),\pi(\mathcal{W})) \leq C(C_0,L)\sqrt{T-t_i},$$
so the claim holds for all $i\in \mathbb{N}$ sufficiently large such that
$$C(C_0,L)\sqrt{T-t_i} \leq d_{g_N}(\pi(\mathcal{W}),\pi(\widetilde{M}\setminus \mathcal{V})).$$
\end{proof}

\subsection{Topological results}
We next recall some topological definitions and results that will be used in later sections. 

\begin{definition} \label{ALEdef} 
Suppose that $\Gamma\subseteq U(2)$ is a finite subgroup acting freely on $\mathbb{C}^2\setminus\{0\}$ and let $(g_0,J_0)$ be the flat K\"ahler orbifold structure on $\mathbb{C}^2/\Gamma$. We say that a K\"ahler manifold $(M,\,g,\,J)$ is an \emph{ALE K\"ahler manifold asymptotic to $\mathbb{C}^2/\Gamma$} if there exists a compact subset $K \subseteq M$ and a diffeomorphism $\Phi: M\setminus K \to (\mathbb{C}^2/\Gamma)\setminus B_{g_0}(0^2,\,R)$ for some $R \in (0,\infty)$ such that if $r$ denotes the radius function on $\mathbb{C}^2/\Gamma$, then there exist $C_i >0$ such that
$$|(\nabla^{g_0})^i ((\Phi^{-1})^{*}g-g_0)|_{g_0} + |(\nabla^{g_0})^i ((\Phi^{-1})^{*}J-J_0)|_{g_0} \leq C_i r^{-4-i}\qquad\textrm{for all $i\in\mathbb{N}$}.$$ 
\end{definition}

ALE Calabi-Yau manifolds are potential singularity models for finite time singularities of the K\"ahler-Ricci flow. However, there are several known topological obstructions restricting how such singularities can develop. These we now recall.

\begin{prop} \label{topologicalrestrictions}
\begin{enumerate}
    \item If $\Sigma$ is a smooth manifold diffeomorphic to $\mathbb{S}^3/\Gamma$ for some $\Gamma \subseteq U(2)$ acting freely on $\mathbb{C}^2\setminus \{0\}$ and if there exists a smooth embedding $\Sigma \hookrightarrow \mathbb{R}^4$, then $\Gamma \subseteq SU(2)$. 
    \item If $(M,g,J)$ is an ALE Calabi-Yau manifold asymptotic to $\mathbb{C}^2/\Gamma$ for some $\Gamma \subseteq SU(2)$ acting freely on $\mathbb{C}^{2}\setminus\{0\}$, then there is an embedded two-sphere $E \hookrightarrow M$ with $E\cdot E=-2$.
\end{enumerate}
\end{prop}

\begin{proof}
\begin{enumerate}
    \item Because $\mathbb{S}^3/\Gamma$ is a rational homology sphere, this follows from \cite[Theorem 2.1(1)]{GilLiv}. 
    \item Because $M$ is hyperK\"ahler, it is diffeomorphic to the minimal resolution of $\mathbb{C}^2/\Gamma$ \cite[Theorem 1.2]{Kron}, and so contains at least one two-sphere $E$ with self-intersection $-2$ \cite[Theorem 7.5.1]{ishii}
\end{enumerate}
\end{proof}

\begin{remark} One can strengthen Proposition \ref{topologicalrestrictions}(i) to state that $\Gamma$ is either trivial or the quaternion group $Q_{8} = \{ \pm 1, \pm i, \pm j, \pm k\} \subseteq SU(2)$ \cite[second paragraph of p.296]{crisp}, but this is a much deeper result and does not simplify any of our arguments. The assumption in Proposition \ref{topologicalrestrictions}(ii) that $\Gamma \subseteq SU(2)$ is essential, since for example the $\mathbb{Z}_2$-quotient of the Eguchi-Hanson metric does not contain any two-sphere with non-trivial self-intersection. 
\end{remark}

We will also use the following consequence of the generalized Jordan-Brouwer separation theorem.

\begin{lemma} \label{JordanBrouwer} Let $M$ be a smooth four-dimensional manifold and $\mathcal{V} \subseteq M$ an open subset diffeomorphic to the total space of the tautological bundle $\mathcal{O}_{\mathbb{P}^{1}}(-1)\to\mathbb{P}^{1}$. Let $\Sigma \hookrightarrow \mathcal{V}$ be a smoothly embedded compact connected hypersurface. Then:
\begin{enumerate}
    \item $\mathcal{V}\setminus \Sigma$ has exactly two connected components.
    \item $M \setminus \Sigma$ has exactly two connected components, one of which is in common with one of (i). 
\end{enumerate}
\end{lemma}

\begin{proof} 
\begin{enumerate}
    \item Because $\mathcal{V}$ is simply connected, this is a consequence of the generalized Jordan-Brouwer theorem \cite[Theorem 4.4.5]{Hirsch}. 
    
    \item Let $\mathcal{C}_{1},\,\mathcal{C}_{2}$ denote the two connected components from part (i). Then because $\Sigma$ is compact, if we identify $\mathcal{V}$ with the total space of $\mathcal{O}_{\mathbb{P}}(-1)\to\mathbb{P}^{1}$ and let $h$ denote the standard Hermitian metric on $\mathcal{O}_{\mathbb{P}}(-1)\to\mathbb{P}^{1}$, then for $R>0$ sufficiently large, the set 
    $$A:= \{ v\in \mathcal{O}_{\mathbb{P}}(-1)| |v|_h \geq R \}$$
    satisfies $A \cap \Sigma = \emptyset$. Then $A \subseteq \mathcal{V} \setminus \Sigma$ is connected, and so we can assume without loss of generality that $\mathcal{C}_1 \subseteq \mathcal{V}\setminus A \Subset \mathcal{V}$. Thus, the boundary of $\mathcal{C}_1$ in $M$ is also $\Sigma$. Set $\widehat{\mathcal{C}}_2 := M\setminus \overline{\mathcal{C}_1}$ so that $\mathcal{C}_1,\widehat{\mathcal{C}}_2$ are open sets sharing a common boundary $\Sigma$ and $M\setminus \Sigma = \mathcal{C}_1 \cup \widehat{\mathcal{C}}_2$. Because $\Sigma = \partial \mathcal{C}_1$ is connected, we also know $\widehat{\mathcal{C}}_2$ is connected, hence $\mathcal{C}_1,\widehat{\mathcal{C}}_2$ are the two connected components of $M\setminus \Sigma$.
\end{enumerate}
\end{proof}

Finally, we recall the definition of the linking number between two curves in $\mathbb{S}^3/\Gamma$. First, suppose that $\sigma_1, \sigma_2$ are disjoint smooth 1-chains in $\mathbb{S}^3/\Gamma$. Because $H_1(\mathbb{S}^3/\Gamma;\mathbb{Z})$ is torsion, we can choose $k\in \mathbb{N}_{>0}$ such that $k\sigma_1$ is null-homologous in $\mathbb{S}^3/\Gamma$. We can therefore find a smooth 2-chain $S$ in $\mathbb{S}^3/\Gamma$ such that $\partial S_1 = k \sigma$. The linking number of $\sigma_1$ and $\sigma_2$ is then defined to be
$$\operatorname{link}_{\mathbb{S}^3/\Gamma}(\sigma_1,\sigma_2):= \frac{1}{k}S_1\cdot \sigma_2.$$
We now recall some well-known properties of the linking number.

\begin{prop} [\protect{\textnormal{\cite[Section 77]{seifert}}}] \label{linkingnumberproperties}
The following hold true.
\begin{enumerate}
    \item $\operatorname{link}_{\mathbb{S}^3/\Gamma}(\sigma_1,\sigma_2)$ is independent of the choice of $k\in \mathbb{N}_{>0}$ and $S_1$.
    \item $\operatorname{link}_{\mathbb{S}^3/\Gamma}$ is symmetric.
    \item If $\sigma_1,\sigma_1'$ are 1-chains which are homologous in $(\mathbb{S}^3/\Gamma)\setminus \sigma_2$, then $\operatorname{link}_{\mathbb{S}^3/\Gamma}(\sigma_1,\sigma_2)=\operatorname{link}_{\mathbb{S}^3/\Gamma}(\sigma_1',\sigma_2)$.
\end{enumerate}
\end{prop}

Suppose that $\gamma_1,\gamma_2: \mathbb{S}^1 \to \mathbb{S}^3/\Gamma$ are smooth immersions that have disjoint images. Define $\rho: [0,1] \to \mathbb{S}^1$, $t\mapsto e^{2\pi\sqrt{-1}t}$, so that $\widetilde{\gamma}_i := \gamma_i \circ \rho: [0,1]\to \mathbb{S}^3$ are smooth 1-chains. We then define the linking number of $\gamma_1,\gamma_2$ to be $$\operatorname{link}_{\mathbb{S}^3/\Gamma}(\gamma_1,\gamma_2):=\operatorname{link}_{\mathbb{S}^3/\Gamma}(\widetilde{\gamma}_1,\widetilde{\gamma}_2) .$$

\begin{remark} \label{remark-homotopy}
By Proposition \ref{linkingnumberproperties}(iii), it follows that if $\gamma_1,\gamma_1'$ are homotopic in $(\mathbb{S}^3/\Gamma)\setminus \gamma_2(\mathbb{S}^1)$, then $\operatorname{link}_{\mathbb{S}^3/\Gamma}(\gamma_1,\gamma_2)=\operatorname{link}_{\mathbb{S}^3/\Gamma}(\gamma_1',\gamma_2)$. 
\end{remark}

For any normal complex analytic surface $X$, there is a well-defined intersection form defined on Weil divisors of $X$ \cite[Section II(b)]{mumford}. In \cite{ReTy}, it was shown that this intersection number $E_1 \cdot E_2$ of Weil divisors in $X$ has a topological characterization in terms of linking numbers. This we now recall. 

\begin{prop} \label{linkingequalsintersection} Let $X$ 
be a complex orbifold surface with isolated singularities and let $p\in X$ be an (isolated) singular point. For any curves $E_1,E_2$ in $X$ with $E_1 \cap E_2 =\{p\}$, it holds true that
$$E_1 \cdot E_2 = \operatorname{link}_{\mathbb{S}}( E_1 \cap \mathbb{S}, E_2 \cap \mathbb{S}),$$
where $\mathbb{S}\subseteq X$ is the intersection of $X$ with the boundary of a small ball centered at $p$. Here, $E_j \cap \mathbb{S}$ is equipped with the orientation induced from the embedding $E_j \cap \mathbb{S} \hookrightarrow \mathbb{S}$ as in \cite[p.100]{GuiPol}.
\end{prop}

\begin{proof} Because $E_1 \cdot E_2$ and $\operatorname{link}_{\mathbb{S}}(E_1 \cap \mathbb{S}_1,E_2 \cap \mathbb{S}_2)$ only depend on the analytic germs of $X,E_1,E_2$ at $p$, and the germ of $X$ at $p$ is complex algebraic, we can assume that $X$ is algebraic. The proposition then follows from \cite[Theorem 7.2]{ReTy} and the fact that the notion of intersection multiplicity for curves in normal surfaces used in \cite{ReTy}
agrees with that defined by Mumford \cite{mumford} (cf.~\cite[p.624]{ReTyAdd}). A direct proof of the proposition is given in \cite[Theorem 2.4]{Hunt}. 
\end{proof}

\section{Finitude of orbifold singularities}

Let $(\widetilde{M}, (\widetilde{g}_{t})_{t\in[0,T)},\widetilde{J}),\,T<\infty,$ be a K\"ahler-Ricci flow on
a compact K\"ahler surface developing a finite time singularity at $t=T$ with $$\lim_{t\nearrow T}\text{vol}_{\widetilde{g}_{t}}(\widetilde{M})>0.$$
In this section, we will show that any tangent flow of $\widetilde{M}$ has at most finitely many orbifold singularities.  

By \cite[Theorem 3.8.3]{songnotes}, there are finitely many pairwise-disjoint $(-1)$-curves $\widetilde{E}_{1},...,\widetilde{E}_{m}$ such that $\widetilde{g}_{t}\to \widetilde{g}_{T}$ in $C_{\text{loc}}^{\infty}(\widetilde{M}\setminus\cup_{j=1}^{m}\widetilde{E}_{j})$ as $t\to T$.
For each $j\in\{1,...,m\}$, there is a neighborhood $\mathcal{V}_{j}$ of $\widetilde{E}_{j}$ which is biholomorphic to the unit disk subbundle of the total space of $\mathcal{O}_{\mathbb{P}^{1}}(-1)\to\mathbb{P}^{1}$ via a biholomorphism mapping $\widetilde{E}_{j}$ to the zero section. Moreover, there is a smooth K\"ahler surface $N$ together with a blowdown map $\pi:\widetilde{M}\to N$ whose exceptional set is precisely $\cup_{j=1}^m \widetilde{E}_j$, and such that $\pi|_{\mathcal{V}_j}$ is holomorphically equivalent to the restriction to a neighborhood of the zero section of the blowdown map $\mathcal{O}_{\mathbb{P}^1}(-1) \to \mathbb{C}^2$.

Given any $(x,t)\in \widetilde{M} \times [0,T)$, recall that $(\widetilde{\nu}_{x,t;s})_{s\in [0,t)}$ denotes the conjugate heat kernel based at $(x,t)$. Fix $x_{0}\in \widetilde{E}_{j}$ and let $(\tilde{\nu}_{x_{0},T;t})_{t\in[0,T)}$ be a conjugate
heat kernel based at $(x_{0},T)$ as in Definition \ref{conjsingtime}. Then there exists $Y_0>0$ such that for all $(x,t) \in \widetilde{M}\times [0,T)$ and $\tau\in (0,t]$, we have $\mathcal{N}_{x_0,T}(\tau)\geq -Y_0$ \cite[Proposition 4.35]{Bam3}.

Fix any sequence $t_{i}\nearrow T$, and set $\widetilde{g}_{i,t}:=(T-t_{i})^{-1}\widetilde{g}_{T+(T-t_{i})t}$
and $\widetilde{\nu}_{t}^{i}:=\widetilde{\nu}_{x_{0},T;T+(T-t_{i})t}$ for $t\in[-(T-t_{i})^{-1}T,0)$. After passing to a subsequence, we can find a correspondence $\mathfrak{C}$ such that 
\begin{equation} \label{Fconverge}
(\widetilde{M},(\widetilde{g}_{i,t})_{t\in[-(T-t_{i})^{-1}T,0)},(\widetilde{\nu}_{t}^{i})_{t\in[-(T-t_{i})^{-1}T,0)})\xrightarrow[i\to\infty]{\mathbb{F},\mathfrak{C}}(X,(g_t)_{t\in(-\infty,0)},(\nu_{t})_{t\in(-\infty,0)})
\end{equation}
as in Theorem \ref{bamconvergence}, where $(X,g,J,f)$ is an orbifold shrinking gradient K\"ahler-Ricci soliton with
smooth Cheeger-Gromov convergence on $X_{\textnormal{reg}} \times (-\infty ,0)$, and $d \nu_t = (4\pi |t|)^{-2}e^{-f}dg_t$. This means that there is a precompact exhaustion $(U_i)$ of $X_{\textnormal{reg}} \times (-\infty ,0)$ along with time-preserving diffeomorphisms $\psi_i:U_i \to V_i \subseteq \widetilde{M} \times (-\infty,0)$ realizing the smooth convergence in the sense that
\begin{equation} \label{smoothconverge} \psi_i^{\ast}\widetilde{g}_i \to g, \qquad\psi_i^{\ast}\widetilde{J} \to J, \qquad \psi_i^{\ast} \widetilde{f}_i \to f, \end{equation}
in $C_{\text{loc}}^{\infty}(X_{\textnormal{reg}}\times (-\infty,0))$ as $i\to \infty$, where $\widetilde{f}_i$ are defined by $d\widetilde{\nu}_t^i = (4\pi (T-t))^{-2} e^{-\widetilde{f}_i} d\widetilde{g}_i$. We let $\widetilde{\nu}_{x,t;s}^i := \widetilde{\nu}_{x,T+(T-t_i)t;T+(T-t_i)s}$ be the conjugate heat kernels corresponding to the rescaled flows.

For simplicity, we assume that $j=1$ and write $\widetilde{E} := \widetilde{E}_1$, $\mathcal{V}:=\mathcal{V}_1$, although the argument easily extends to the general case. Set $V_i := \psi_{i}(U_i)\subseteq \widetilde{M} \times \mathbb{R},$ $V_{i,-1}:= V_i \cap (\widetilde{M} \times \{-1\})$, and 
\begin{equation} \label{eq: Edef} E_i := \psi_{i,-1}^{-1}(V_{i,-1} \cap \widetilde{E}),\end{equation} so that $E_i$ is a $(\psi_{i,-1}^{\ast}\widetilde{J})$-holomorphic curve in $U_{i,-1}$ for each $i\in \mathbb{N}$. After passing to a subsequence, we can apply \cite[Theorem 7.3.8]{Burago} and a diagonal argument to obtain a locally closed subset $E^{\circ}$ of $X_{\textnormal{reg}}$ such that for any compact subset $K \subseteq X_{\textnormal{reg}}$, we have $E_i \cap K \to E^{\circ} \cap K$ in the Hausdorff distance with respect to $g_{-1}$. This is an analytic set in $X_{\textnormal{reg}}$.

\begin{lemma} \label{Hausdorff} $E^{\circ}$ is an analytic set in $(X_{\textnormal{reg}},J)$.
\end{lemma}

\begin{proof} Because the claim is local, it suffices to show that for some neighborhood $\mathcal{U}\subseteq X_{\textnormal{reg}}$ of a given point $y_0 \in E^{\circ}$, $\mathcal{U}\cap E^{\circ}$ is an analytic subset of $\mathcal{U}$. To this end, choose an open coordinate neighbourhood $B\subset \mathcal{U}$ of $y_{0}$ and let $(z_{1},\,z_{2})$ denote $J$-holomorphic coordinates on $B$. Write $g_{0}$ for the standard Euclidean metric in these coordinates, with complex structure $\widehat{J}$ and K\"ahler form $\omega_0$. Because $(\psi_{i,-1}^{\ast}\widetilde{g}_{i,-1},\psi_{i,-1}^{\ast}\widetilde{J})\to (g,J)$ in $C_{\text{loc}}^{\infty}(X_{\textnormal{reg}}\times \{-1\})$, Lemma \ref{coordlemma} allows us to construct holomorphic coordinates $z^{(i)}=(z_1^{(i)},z_2^{(i)})$  with respect to $J_i :=\psi_{i,-1}^{\ast}\widetilde{J}$ on $B_g(y_0,r) \subseteq B$ for some $r>0$, which converge locally smoothly on this ball to $(z_{1},\,z_{2})$ as $i\to\infty$.
Let $\omega^{(i)}_{0}$ denote the flat metric with respect to the coordinates $z_{i}$, i.e., $$\omega^{(i)}_{0}=\frac{\sqrt{-1}}{2}\partial_{J_{i}}\bar{\partial}_{J_{i}}\left(|z^{(i)}_{1}|^{2}+|z_{2}^{(i)}|^{2}\right).$$ Then $\omega^{(i)}_{0}\to\omega_{0}$ locally smoothly on $B_g(y_0,r)$ as $i\to\infty$. In particular, because $g$ is equivalent to the flat metric on $B$, there exists $C>0$ such that for all $i$ sufficiently large, 
$$C^{-1}g_{0}^{(i)} \leq \psi_{i,-1}^{\ast}\widetilde{g}_{i,-1} \leq C g_{0}^{(i)},$$
where $g_{0}^{(i)}$ denotes the K\"ahler metric associated to $\omega_{0}^{(i)}$. Via the holomorphic coordinates just introduced,
we can define a holomorphic isometry $z^{(i)}:(B_g(y_0,r),\,J_i,\, g_0^{(i)})\to(z^{(i)}(B_g(y_0,r)),\, \widehat{J},\,g_{0})$ for sufficiently large $i\in \mathbb{N}$.

We can now consider 
 $E_i' := z^{(i)}(B_g(y_0,r)\cap E_i)\subseteq \mathbb{C}^2$ as a sequence of smooth $\widehat{J}$-holomorphic submanifolds of an open subset of $\mathbb{C}^2$. The aforementioned equivalence of metrics then yields for $i$ sufficiently large,
$$\int_{E_i'} \omega_0 = \int_{E_i \cap B_{g}(y_0,2r)} \omega_0^{(i)} \leq C \int_{E_i \cap B_g(y_0,2r)} \psi_{i,-1}^{\ast} \widetilde{\omega}_{i,-1} \leq C \int_{\widetilde{E}} \widetilde{\omega}_{i,-1} = 2\pi C,$$
which in turn allows us to assert that 
$$\limsup_{i \to \infty} \mathcal{H}_{g_{0}}^2(E_i') <\infty.$$
Because $E_i$ Hausdorff converge to $E^{\circ}$ with respect to $g$,
$E_i'$ Hausdorff converge to $z(E^{\circ}\cap B_g(y_0,r)) \subseteq  \mathbb{C}^2$ with respect to the flat metric $g_0$. \cite[Theorem 1]{bishop} then tells us from the above Hausdorff measure bound that this latter variety is an analytic subset of $z(B_g(y_0,r))$, so that $E^{\circ} \cap B_g(y_0,r)$ is an analytic subset of $B_g(y_0,r)$. 
\end{proof}

Fix a smooth volume form $\Omega_N \in \mathcal{A}^{2,2}(N)$ such that $\operatorname{Ric}(\Omega_N)=0$ in a neighborhood of $\pi(\widetilde{E})$. By shrinking $\mathcal{V}$, we can thus assume that $\vartheta := \operatorname{Ric}(\Omega_N)$ satisfies $\pi^{\ast}\vartheta =0$ on $\mathcal{V}$. Define 
\begin{equation} \label{defineu} \widetilde{u}_t := \log \left( \frac{\pi^{\ast}\Omega_N}{\widetilde{\omega}_t^2} \right),\end{equation}
which is smooth on $(\widetilde{M}\setminus \widetilde{E})\times [0,T]$, but has a logarithmic singularity along $\widetilde{E}$. We now prove some crucial properties of $\widetilde{u}_{t}$.

\begin{prop} \label{functionu}
\begin{enumerate}
\item There is a defining section $s\in H^0 (\widetilde{M},\mathcal{O}_{\widetilde{M}}(\widetilde{E}))$ and a smooth time-dependent Hermitian metric $h_t$ on $\mathcal{O}_{\widetilde{M}}(\widetilde{E})$ such that $\widetilde{u}_t = \log |s|_{h_t}^2$.
\item $\sqrt{-1} \partial \bar{\partial} \widetilde{u}_t =\operatorname{Ric}_{\widetilde{\omega}_t} - \pi^{\ast} \vartheta +2\pi [\widetilde{E}]$, where $[\widetilde{E}]$ is the current of integration along $\widetilde{E}$.
\item $(\partial_t - \Delta) \widetilde{u}_t = \mathrm{tr}_{\widetilde{\omega}_t}(\pi^{\ast}\vartheta) - 2\pi\mathcal{H}_{\widetilde{g}_t}^2|_{\widetilde{E}}$, where $\mathcal{H}_{\widetilde{g}_t}^2$ denotes the two-dimensional Hausdorff measure with respect to the metric $d_{\widetilde{g}_t}$.
\item For any $\epsilon>0$, there exists $C(\epsilon)>0$ such that any $(x,t) \in (\widetilde{M}\setminus \widetilde{E})\times [\frac{T}{2},T)$ with 
$$P^{\ast-}(x,t;\epsilon \sqrt{T-t}) \cap (\widetilde{E} \times [0,T)) =\emptyset$$ satisfies 
$|\nabla \widetilde{u}_t|(x) \leq \frac{C(\epsilon)}{\sqrt{T-t}}$.
\end{enumerate}
\end{prop}

\begin{proof} 
\begin{enumerate}[label=\textnormal{(\roman{*})}, ref=(\roman{*})]
    \item Because $\pi|_{\mathcal{V}}$ is holomorphically equivalent to the blowdown map near the exceptional divisor, we
can choose local holomorphic coordinates on $\widetilde{M}$ and $N$ based at
any point of $\widetilde{E}$ such that the coordinate representation of $\pi$
becomes $\pi((z_{1},z_{2}))=(z_{1},z_{1}z_{2})$, where $\widetilde{E}=\{z_{1}=0\}$.
In particular, we see that
\[
(\pi^{\ast}\omega_{N})^{2}=|z_{1}|^{2}v_t \,\widetilde{\omega}_{t}^{2}
\]
for some locally defined nonvanishing smooth function $v_{t}$. We also see that
$$\log |s|_h^2 = \log (|z_1|^2 w)$$
for any smooth Hermitian metric $h$ on $\mathcal{O}_{\widetilde{M}}(\widetilde{E})$, where $w$ is a locally defined positive smooth function. On the given coordinate chart, there is a unique choice of $h_t$ such that $\widetilde{u}_t=\log |s|_{h_t}^2$. Thus, the definition of $h_t$ is independent of the choice of coordinates.
\item This follows from (i) and the Poincar\'{e}-Lelong formula.
\item For any $\phi \in C_c^{\infty}(\widetilde{M})$, we have (in the sense of distributions)
$$\int_{\widetilde{M}} (\Delta \widetilde{u}_t)\phi\, \widetilde{\omega}_t^2 = 2 \int_{\widetilde{M}} \sqrt{-1} \partial \bar{\partial} \widetilde{u}_t \wedge \phi \widetilde{\omega}_t = \int_{\widetilde{M}} \phi(R_{\widetilde{\omega}_t}-\text{tr}_{\widetilde{\omega}_t}(\pi^{\ast}\vartheta)) + 2\pi \int_{\widetilde{E}} \phi \,\widetilde{\omega}_t.$$
We also compute that
$$\partial_t \widetilde{u}_t =R_{\widetilde{\omega}_t},$$
from which the claim follows.
\item We can integrate (iii) against a conjugate heat kernel to obtain
$$\widetilde{u}_t(x) = \int_{\widetilde{M}}\widetilde{u}_0\,d\nu_{x,t;0}+\int_{0}^{t}\int_{\widetilde{M}}\text{tr}_{\widetilde{\omega}_t}(\pi^{\ast}\vartheta)\,d\nu_{x,t;s}ds-2\pi\int_{0}^{t}\int_{\widetilde{E}}K(x,t;y,s)\omega_{s}(y)ds.$$
We first use \cite[Proposition 4.2]{Bam1} to estimate
\begin{equation*}
\begin{split}
\left|\nabla_{x}\int_{\widetilde{M}}\widetilde{u}_0 d\nu_{x,t;0}\right|_{\widetilde{g}_t}&\leq  \int_{\widetilde{M}} |\widetilde{u}_0| (y)|_{\widetilde{g}_t}\nabla_{x}\log K(x,t;y,0)|_{\widetilde{g}_t}\,d\nu_{x,t;0}(y)\\
&\leq \left(\int_{\widetilde{M}} \widetilde{u}_0^2 d\nu_{x,t;0}\right)^{\frac{1}{2}} \left(\int_{\widetilde{M}}|\nabla_{x}\log K(x,t;y,0)|_{\widetilde{g}_t}^{2}d\nu_{x,t;0}(y)\right)^{\frac{1}{2}}\\
&\leq  \frac{2}{\sqrt{t}}\left(\frac{C}{t^{2}}\int_{\widetilde{M}} \widetilde{u}_0^2 \widetilde{\omega}_{t}^{2}\right)^{\frac{1}{2}}\\
&\leq  C
\end{split}
\end{equation*}
for all $(x,t)\in \widetilde{M}\times[\frac{T}{2},T)$. Next, we recall that $|\text{tr}_{\widetilde{\omega}_t}(\pi^{\ast}\vartheta)|\leq C$ by Lemma \ref{schwartz},
so another application of \cite[Proposition 4.2]{Bam1} yields the bound
$$\left|\nabla_x \int_{0}^{t}\int_{\widetilde{M}}\text{tr}_{\omega_{t}}(\pi^{\ast}\vartheta)d\nu_{x,t;s}ds\right|_{\widetilde{g}_t}\leq  C\int_{0}^{t}\int_{\widetilde{M}}|\nabla_{x}\log K(x,t;y,s)|_{\widetilde{g}_t}d\nu_{x,t;s}ds\leq C$$
for all $(x,\,t)\in \widetilde{M}\times[\frac{T}{2},T)$. For $\delta \in (0, T-t]$ to be
determined, we have
$$T-s=(T-t)+(t-s) \leq \frac{(T-t)}{\delta}(t-s) + (t-s) \leq 2\frac{(T-t)}{\delta}(t-s)$$
for all $s\in [0,t-\delta]$, and $T-s\leq 2(T-t)$ for all $s\in [t-\delta,t]$. We use Proposition \ref{bamheatkernel} to estimate:
\begin{equation*}
\begin{split}
\left|\nabla_{x}\int_{0}^{t}\int_{\widetilde{E}}K(x,t;y,s)\omega_{s}(y)ds\right|_{\widetilde{g}_t} &\leq  \int_{t-\delta}^{t}\frac{C}{(t-s)^{\frac{5}{2}}}\exp\left(-\frac{d_{s}^{2}(z_{s},\widetilde{E})}{10(t-s)}\right)\int_{\widetilde{E}}\omega_{s}ds\\
&\qquad+\int_{0}^{t-\delta}\frac{C}{(t-s)^{\frac{5}{2}}}\int_{\widetilde{E}}\omega_{s}ds\\
&\leq  \frac{C(T-t)}{\inf_{s\in [t-\delta,t]}d_{s}^{5}(z_{s},\widetilde{E})}\int_{t-\delta}^{t} \left(\frac{d_{s}^{2}(z_{s},\widetilde{E})}{10(t-s)}\right)^{\frac{5}{2}} \exp\left(-\frac{d_{s}^{2}(z_{s},\widetilde{E})}{10(t-s)}\right)ds\\ &\qquad+C\int_{0}^{t-\delta} \frac{T-s}{(t-s)^{\frac{5}{2}}}ds \\
&\leq   \frac{C(T-t)\delta}{\inf_{s\in [t-\delta,t]}d_{s}^{5}(z_{s},\widetilde{E})} + C\frac{T-t}{\delta}\int_0^{t-\delta} \frac{1}{(t-s)^{\frac{3}{2}}}\\
&\leq  C\left(\frac{(T-t)\delta}{\inf_{s\in [t-\delta,t]} d_{s}^{5}(z_{s},\widetilde{E})} +\frac{T-t}{\delta^{\frac{3}{2}}} \right)
\end{split}
\end{equation*}
for any $H_{4}$-centers $(z_{s},s)$ of $(x,t)$ satisfying $\inf_{s\in [t-\delta,t]} d_{s}(z_{s},\widetilde{E})>0$. In particular, if $$P^{\ast-}(x,t;\epsilon \sqrt{T-t})\cap (\widetilde{E}\times[0,T)) =\emptyset,$$ then we can choose $\delta = \delta_0^2 \epsilon^2 (T-t)$, where $\delta_0>0$ is independent of time and $\epsilon$, such that
$$\inf_{s\in [t-\delta,t]}d_s(z_s,\widetilde{E}) \geq \delta_0 \epsilon \sqrt{T-t}.$$
Combining expressions then yields
$$|\nabla \widetilde{u}_t|(x)\leq \frac{C}{\delta_0^3 \epsilon^3 \sqrt{T-t}}. $$
\end{enumerate}

\end{proof}

It is well-known that the curvature is bounded uniformly on the complement of any neighborhood of $\widetilde{E}$ \cite{SoWe1}. In the following proposition, we quantify this by showing that outside a time-dependent family of neighborhoods of $\widetilde{E}$ which shrink at a specified rate, there is a Type I curvature bound. This will be used to show that any orbifold singularity of $X$ is a limit of points in $\widetilde{E}$. The idea is to use the estimates from Proposition \ref{functionu} to show that if there is a sequence of points outside of these neighborhoods whose curvature blows up at the Type II rate, then there exists a non-flat singularity model that is Ricci-flat. This leads to a contradiction due to the topological restrictions discussed in Proposition \ref{topologicalrestrictions}.

\begin{prop} \label{NoncollapsedTypeI}
For any $\epsilon>0$, there exists $c=c(\epsilon)>0$ such that
for all $(x,t)\in \widetilde{M}\times[0,T)$ with $P^{\ast}(x,t;\epsilon\sqrt{T-t})\cap (\widetilde{E} \times [0,T))=\emptyset$,
it holds that
$$r_{\Rm}(x,t)\geq c\sqrt{T-t}.$$
\end{prop}

\begin{proof} Suppose by way of contradiction that there exist $(\bar{x}_{i},\overline{t}_{i})\in \widetilde{M}\times[0,T)$
such that 
$$P^{\ast}(\bar{x}_{i},\overline{t}_{i};\epsilon\sqrt{T-\overline{t}_{i}})\cap (\widetilde{E} \times [0,T))=\emptyset\qquad\textrm{and}\qquad\lim_{i\to \infty} \frac{r_{\Rm}(\bar{x}_i,\overline{t}_i)}{\sqrt{T-\overline{t}_i}}=0.$$
Then we can choose $\Lambda_i \to \infty$ such that 
$$\lim_{i\to \infty} \frac{\Lambda_i r_{\Rm}(\bar{x}_i,\overline{t}_i)}{\sqrt{T-\overline{t}_i}}=0.$$
We can apply \cite[Claim 10.5]{Bam1} with $A$ replaced by 
$\Lambda_i$, and $(x,t)$ replaced by $(\bar{x}_i,\overline{t}_i)$, to obtain a sequence of points $(x_i',t_i')\in P^{\ast -}(\bar{x}_i,\overline{t}_i, 10\Lambda_i r_{\Rm}(\bar{x}_i,\overline{t}_i))$ such that $r_{\Rm}(x_i',t_i')\leq r_{\Rm}(\bar{x}_i,\overline{t}_i)$ and
$$\inf_{P^{\ast-}(x_i',t_i';\Lambda_i r_{\Rm}(x_i',t_i'))} r_{\Rm} \geq \tfrac{1}{10} r_{\Rm}(x_i',t_i').$$
Because $\lim_{i\to \infty} (T-\bar t_i)^{-\frac{1}{2}}\Lambda_i r_{\Rm}(\bar{x}_i,\overline{t}_i)=0$, we can apply \cite[Proposition 9.4]{Bam1} to conclude that
$$P_i^{\ast} := P^{\ast-}(x_i',t_i';\Lambda_i r_{\Rm}(x_i',t_i')) \subseteq P^{\ast-}(\bar{x}_i,\overline{t}_i; \tfrac{\epsilon}{2}\sqrt{T-\overline{t}_i}).$$
In particular, we have that $P_i^{\ast}\cap (\widetilde{E} \times [0,T)) = \emptyset$ for sufficiently large $i\in \mathbb{N}$. 
Set $Q_i := r_{\Rm}^{-2}(x_i',t_i')$, and define $\widehat g_{i,t}:=Q_{i} \widetilde{g}_{t_{i}'+Q_{i}^{-1}t}$, $\widehat{\omega}_{i,t}:= Q_i \widetilde{\omega}_{t_i'+Q_{i}^{-1}t}$, so that after passing to a subsequence, we have the convergence
\[
(\widetilde{M},(\widehat g_{i,t})_{t\in[-Q_it_i',0]},x_{i}')\to(M_{\infty},(g_{\infty,t})_{t\in(-\infty,0]},x_{\infty})
\]
in the smooth pointed Cheeger-Gromov sense, where
$(M_{\infty},(g_{\infty,t})_{t\in(-\infty,0]})$ is an ancient Ricci
flow which is $\kappa$-noncollapsed at all scales for some $\kappa>0$, where
$r_{\Rm}^{g_{\infty}} \geq \frac{1}{10}$, and where $r_{\Rm}^{g_{\infty}}(x_{\infty},0)=1$. Let $\zeta_i:U_i \to V_i \subseteq \widetilde{M}$ be diffeomorphisms realizing the convergence, where $(U_i)$ is a precompact open exhaustion of $M_{\infty}$. That is to say, $\zeta_i^{-1}(x_i')\to x_{\infty}$, $\zeta_i^{\ast}\widehat{g}_i \to g_{\infty}$ in $C_{\operatorname{loc}}^{\infty}(M_{\infty}\times (-\infty,0])$, and $\zeta_i^{\ast}\widetilde{J} \to J_{\infty}$ in $C_{\operatorname{loc}}^{\infty}(M_{\infty})$, where $J_{\infty}$ is a complex structure on $M_{\infty}$ that makes $(M_{\infty},g_{\infty,t},J_{\infty})$ K\"ahler for each $t\in (-\infty,0]$. 

Recall that $\widetilde{u}_{t}$ is defined by \eqref{defineu}, and set  $$\widehat{u}_{i,t}:=\log\left(\frac{\lambda_{i}(\pi^{\ast}\omega_{N})^{2}}{\widehat{\omega}_{i,t}^{2}}\right),$$
where $\lambda_{i}>0$ are chosen so that $\widehat{u}_{i}(x_{i},0)=0$. The following claim shows that $\widehat{u}_{t}$ converge to a constant as $i\to \infty$. In conjunction with Proposition \ref{functionu}, this will be used to show that the limit $M_{\infty}$ is Calabi-Yau.

\begin{claim} \label{claim-Pstarparabolicdistance}
For any $D>0$, we have 
\[
\lim_{i\to\infty}\sup_{P^{\ast-}_{\hat g_i}(x_{i}',0;D)} |\nabla \widehat{u}_i|_{\widehat{g}_i} =0.
\]
\end{claim}

\begin{proof}[Proof of Claim \ref{claim-Pstarparabolicdistance}]
When $i\geq i(D)$ is sufficiently large, we have
$$P_{\widehat{g}_i}^{\ast -}(x_i',0;D) = P_{\widetilde{g}}^{\ast -}(x_i',t_i';r_{\Rm}^{\widetilde{g}}(x_i',t_i')D) \subseteq P_i^{\ast}$$
and $P_i^{\ast} \cap (\widetilde{E}\times [0,T))=\emptyset$. Thus, Proposition \ref{functionu}(iv) gives that
$$\sup_{P^{\ast-}_{\hat g_i}(x_{i}',0;D)}\left|\nabla \widehat{u}_i\right|_{\hat g_i} 
= \sup_{P^{\ast-}_{\widetilde{g}}(x_{i}',t_i';DQ_i^{-\frac{1}{2}})}Q_i^{-\frac{1}{2}}\left|\nabla \widetilde{u}\right|_{\widetilde{g}} \leq \frac{C}{\sqrt{Q_i(T-t_i')}},$$
and the claim follows from the fact that $\lim_{i\to \infty} Q_i(T-t_i')=\infty$.
\end{proof}

Because $P_{\widehat{g}_i}^{\ast -}(x_i',0;D) \cap (\widetilde{E}\times [0,T)) = \emptyset$ for $i\geq i(D)$, we know that
$$
\sqrt{-1}\partial\bar{\partial}\widehat{u}_{i,t}=\Ric_{\hat\omega_{i,t}}\qquad\textrm{and}\qquad
(\partial_{t}-\Delta)\widehat{u}_{i}=0$$
on $P_{\hat g_{i}}^{\ast-}(x_{i}',0;D)$. We also know that
$$|\partial_t \widehat{u}_{i}| =|R_{\widehat{g}_{i}}| \leq C$$
on $P_{\widehat{g}_i}^{\ast -}(x_i',0;D)$ for sufficiently large $i\geq i(D)$.
Combining expressions and using the fact that $\widetilde{u}_{i,0}(x_i')=0$, we conclude that $\widehat{u}_i$ is uniformly bounded on $P_{\widehat{g}_i}^{\ast -}(x_i',0,D)$ for any $D>0$. Thus, local parabolic
regularity theory allows us to extract a subsequence such that $\zeta_i^{\ast}\widehat{u}_{i}$
converge in $C_{\text{loc}}^{\infty}(M_{\infty}\times (-\infty,0])$
to a solution $u_{\infty}$ of the heat equation which is constant
on each time slice. This means that
$u_{\infty}\equiv 0$, so that ${\operatorname{Ric}}_{g_{\infty}}\equiv0$ and
$(M_{\infty},g_{\infty},J_{\infty})$ is a smooth 
ALE Calabi-Yau manifold, diffeomorphic to a finite quotient of a hyperK\"ahler ALE manifold \cite{Suv}. Let $\mathbb{C}^{2}/\Gamma$ denote its tangent cone at infinity, where $\Gamma \leq U(2)$ is a finite subgroup acting freely on $\mathbb{C}^2\setminus \{0\}$. Then there exists a real hypersurface $\Sigma \subseteq M_{\infty}$ diffeomorphic to $\mathbb{S}^{3}/\Gamma$. Moreover, for $i\in \mathbb{N}$ sufficiently large, we see that $\Sigma \subseteq U_i$ and $\zeta_i(\Sigma) \subseteq P_i^{\ast}$, so that $\zeta_i(\Sigma)\cap \widetilde{E} = \emptyset$. We now show that $\zeta_i$ maps $\Sigma$ into the neighborhood $\mathcal{V}$ of $\widetilde{E}$. 

\begin{claim} \label{claim-containedinV} For sufficiently large $i\in \mathbb{N}$, $\zeta_i(\Sigma) \subseteq \mathcal{V}.$
\end{claim}

\begin{proof} 
Because $M_{\infty}$ is not flat, there exists $y_{\infty}\in M_{\infty}$ such that $|\Rm_{g_{\infty,0}}|_{g_{\infty,0}}(y_{\infty})\neq 0$. Then
$$\lim_{i \to \infty} |\Rm_{\widetilde{g}_{t_i'}}|_{\widetilde{g}_{t_i'}}(\zeta_i(y_{\infty}))= \lim_{i \to \infty} Q_i|\Rm_{\widehat{g}_{i,0}}|_{\widehat{g}_{i,0}}(\zeta_i(y_{\infty}))  = \infty.$$
Let $\mathcal{W} \Subset \mathcal{V}$ be a neighborhood of $\widetilde{E}$. Then because 
$$\sup_{t\in [0,T)} \sup_{x\in \widetilde{M}\setminus \mathcal{W}} |\Rm_{\widetilde{g}_t}|_{\widetilde{g}_t}(x)<\infty,$$ we see that $\zeta_i(y_{\infty}) \in \mathcal{W}$ for sufficiently large $i\in \mathbb{N}$. By Lemma \ref{schwartz}, there exists $C_0>0$ such that $\widetilde{g}_t \geq C_0^{-1}\pi^{\ast}g_N$ for all $t\in [0,T)$. Since $d_{g_N}(\pi(\mathcal{W}),\pi(\widetilde{M}\setminus \mathcal{V}))>0,$
we consequently find that
$$\liminf_{t \to T} d_{\widetilde{g}_t}(\mathcal{W},\widetilde{M}\setminus \mathcal{V})>0.$$
Combining this with the fact that
$$\limsup_{i\to \infty} \sup_{x\in \Sigma} d_{\widetilde{g}_{t_i'}}(\zeta_i(x),\zeta_i(y_{\infty})) = \limsup_{i \to \infty} Q_i^{-\frac{1}{2}} \sup_{x\in \Sigma} d_{\widehat{g}_{i,0}}(\zeta_i(x),\zeta_i(y_{\infty})) = 0,$$
and $\zeta_i(y_{\infty})\in \mathcal{W}$, we arrive at the desired claim.
\end{proof}

For large $i\in \mathbb{N}$, we thus have that $\zeta_i(\Sigma) \subseteq \mathcal{V} \setminus \widetilde{E}$. However, $\mathcal{V}\setminus \widetilde{E}$ is diffeomorphic to $\mathbb{R}^4 \setminus \{0\}$, so that $\mathbb{S}^3/\Gamma$ smoothly embeds into $\mathbb{R}^4$. By Proposition \ref{topologicalrestrictions}(i), it follows that $\Gamma \subseteq SU(2)$. Moreover, by Proposition \ref{topologicalrestrictions}(ii), $M_{\infty}$ contains at least one embedded two-sphere $E_{\infty}$ with self-intersection $-2$. Note that $\zeta_i(E_{\infty}) \subseteq P_i^{\ast}$ for large $i\in \mathbb{N}$, so that $\zeta_i(E_{\infty})\cap \widetilde{E}=\emptyset$. We can in addition argue as in Claim \ref{claim-containedinV} to obtain the inclusion $\zeta_i(E_{\infty}) \subseteq \mathcal{V} \setminus \widetilde{E}$. That is, $\zeta_i(E_{\infty})$ represents a homology class in $H_{2}(\mathcal{V};\mathbb{Z})$
with self-intersection $-2$. However, we know that $H_{2}(\mathcal{V};\mathbb{Z})$
is spanned by the class of a $(-1)$-curve, and so we reach a contradiction. 
\end{proof}

Our next goal is to apply Proposition \ref{NoncollapsedTypeI} to show that either the curve $E^{\circ}$ from Lemma \ref{Hausdorff} is empty or all orbifold singularities of $X$ lie along its closure $E:= \overline{E}^{\circ}$. The former case should correspond to $\widetilde{E}$ being entirely pulled into one of the orbifold singularities forming along the rescaled sequence $(\widetilde{M},\widetilde{g}_{i,t_i})$. To make this precise, we will now study the behavior of these time slices where the orbifold singularities develop.

Suppose that $x_{\ast}\in X\setminus X_{\textnormal{reg}}$. By \cite[Theorem 6.45]{Bam2}, there exist $x_{i}'\in \widetilde{M}$ such that $(x_{i}',-1)\xrightarrow[i\to\infty]{\mathfrak{C}}
(x_{\ast},-1)
$ with respect to the convergence (\ref{Fconverge}).
By Proposition \ref{NoncollapsedTypeI}, there exist $\epsilon_{i}\searrow0$ such that $P_{\widetilde{g}_{i}}^{\ast-}(x_{i}',-1;\epsilon_{i})\cap (\widetilde{E} \times \mathbb{R}) \neq\emptyset$ for all $i\in \mathbb{N}$.
Thus, any choice of points $(x_{i},s_{i})\in P_{\widetilde{g}_{i}}^{\ast-}(x_{i}',-1;\epsilon_{i})\cap (\widetilde{E} \times \mathbb{R})$
satisfy
\[
(x_{i},s_{i})\xrightarrow[i\to\infty]{\mathfrak{C}}(x_{\ast},-1).
\]
For $r>0$ small, choose $y_{\ast}\in\partial B(x_{\ast},r) \cap X_{\text{reg}}$ and
set $y_{i}:=\psi_{i,-1-\frac{1}{2}r^{2}}(y_{\ast})$, so that 
\[
(y_{i},-1-\tfrac{1}{2}r^{2})\xrightarrow[i\to\infty]{\mathfrak{C}}(y_{\ast},-1-\tfrac{1}{2}r^{2}).
\]
Then since
\[
(y_{\ast},-1-\tfrac{1}{2}r^{2})\in P^{\ast}(x_{\ast},-1;\tfrac{1}{2}Cr,-\tfrac{3}{4}r^2,0)
\]
for some $C>0$ independent of $r\in (0,1]$, \cite[Lemma 15.8]{Bam3}
gives us that 
\[
(y_{i},-1-\tfrac{1}{2}r^{2})\in P_{\widetilde{g}_i}^{\ast}(x_{i},s_{i};Cr,-r^2,0)
\]
for sufficiently large $i\geq i(r)$. In particular, if $(w_{i},-1-r^2)$
is an $H_{4}$-center of $(y_{i},-1-\frac{r^{2}}{2})$ and $(z_{i},-1-r^{2})$
is an $H_{4}$-center of $(x_{i},s_{i})$, then 
\[
d_{\widetilde{g}_{i,-1-r^{2}}}(w_{i},z_{i})\leq Cr+d_{W_{1}}^{\widetilde{g}_{i,-1-r^{2}}}(\widetilde{\nu}_{y_{i},-1-\frac{r^{2}}{2};-1-r^{2}}^{i},\widetilde{\nu}_{x_{i},s_{i};-1-r^2}^{i})\leq 2Cr
\]
for all $i\ge\underline{i}(r)$. By \cite[Proposition 5.6]{CMZ}, this is also true if 
$(w_{i},-1-r^2)$ is instead an $\ell$-center of $(y_{i},-1-\frac{r^{2}}{2})$
and $(z_{i},-1-r^{2})$ is an $\ell$-center of $(x_{i},s_{i})$,
so we henceforth assume instead that these points are $\ell$-centers. 

As a consequence of Lemma \ref{containedinV}, we know that $\psi_{i,-1}(\partial B(x_{\ast},r))\subseteq \mathcal{V}$ for all $r>0$ sufficiently small and $i\geq \underline{i}(r)$ sufficiently large. For each small $r>0$, we can then let $\mathcal{C}_{r,i}$ be the connected
component of $\mathcal{V}\setminus \psi_{i,-1}(\partial B(x_{\ast},r))$ containing $\psi_{i,-1}(\partial B(x_{\ast},\frac{r}{2}))$
when $i\geq \underline{i}(r)$ is sufficiently large. We now show that the sequence $x_i \in \widetilde{M}$ approximating the singular point $x_{\ast}$ must be contained in the corresponding ``bubble" $\mathcal{C}_{r,i}$.

\begin{lemma} \label{disappearing}
There exists $A>0$ such that for any $r>0$ sufficiently small,
if $i\geq \underline{i}(r)$ is sufficiently large, then $x_{i} \in \mathcal{C}_{Ar,i}$.
\end{lemma}

\begin{proof}
For $A>0$ to be determined, suppose by way of contradiction
that $x_{i}\in\mathcal{V}\setminus\mathcal{C}_{Ar,i}$ for infinitely
many $i\in\mathbb{N}$. In the next claim we show that the contradiction hypothesis would imply that the $\ell$-centers of $(x_i,s_i)$ lie outside a slightly smaller bubble $\mathcal{C}_{\frac{A}{2}r,i}$. 

\begin{claim} \label{claim-usedindisappearing}
If $A>0$ is sufficiently large, then any $\ell$-center $(z_i,-1-r^2)$ of $(x_i,s_i)$ satisfies $z_{i}\notin\mathcal{C}_{\frac{A}{2}r,i}$.
\end{claim}

\begin{proof}[Proof of Claim \ref{claim-usedindisappearing}]
Suppose by way of contradiction
this is not the case, and let $\gamma:[-1-r^{2},s_{i}]\to \widetilde{M}$ be a
minimizing $\mathcal{L}_{s_i}$-geodesic from $(x_{i},s_{i})$ to $(z_{i},-1-r^{2})$. By Lemma \ref{containedinV}, we see that
$$\psi_{i,-1}\left(B(x_{\ast},100Ar)\setminus B(x_{\ast},\tfrac{1}{100}r)\right) \subseteq \mathcal{V}$$
for sufficiently small $r>0$ and large $i\geq \underline{i}(A,r)\in \mathbb{N}$. Thus, $\psi_{i,-1}(\partial B(x_{\ast},Ar))$ and $\psi_{i,-1}(\partial B(x_{\ast},\frac{A}{2}r))$
disconnect $\mathcal{V}$, there is a subinterval $[t_{-},t_{+}]\subseteq[-1-r^{2},s_{i}]$ such that 
$$\gamma([t_{-},t_{+}])\subseteq\psi_{i,-1}(\overline{B}(x_{\ast},Ar)\setminus B(x_{\ast},\tfrac{1}{2}Ar)),$$
and $$\gamma(t_{-})\in\psi_{i,-1}(\partial B(x_{\ast},\tfrac{1}{2}Ar)),\qquad 
\gamma(t_{+})\in\psi_{i,-1}(\partial B(x_{\ast},Ar)).$$ 
Then the convergence $\psi_{i}^{\ast}\widetilde{g}_{i}\to g$
in $C_{\text{loc}}^{\infty}(X_{\textnormal{reg}} \times (-\infty,0))$ implies that $\psi_{i}^{\ast}\widetilde{g}_{i,t}\geq \frac{1}{2} g_{-1}$ on \linebreak $\overline{B}(x_{\ast},Ar)\setminus B(x_{\ast},\frac{1}{2}Ar)$ for all $t\in[-1-r^{2},-1]$ if $r=r(A)>0$ is sufficiently small and $i\geq \underline{i}(r,A)$. Recalling that $\liminf_{i\to\infty}\inf_{\widetilde{M}\times[-2,-1]}R_{\widetilde{g}_{i}}\geq0$,
we can apply the Cauchy-Schwarz inequality to get that $\widetilde{\gamma}:=\psi_{i,-1}^{-1}\circ\gamma|_{[t_{-},t_{+}]}$
satisfies 
\begin{align*}
Cr\geq & \mathcal{L}_{s_i}(\gamma)\\
\geq & \int_{t_{-}}^{t_{+}}\sqrt{s_i-s}|\dot{\gamma}(s)|_{\widetilde{g}_{i,s}}^{2}ds-Cr\\
\geq & \frac{1}{2}\int_{t_{-}}^{t_{+}}\sqrt{s_i-s}|\dot{\widetilde{\gamma}}(s)|_{g}^{2}ds-Cr\\
\geq & \frac{1}{2r}\left(\int_{t_{-}}^{t_{+}}\frac{1}{\sqrt{s_i-s}}ds\right)\left(\int_{t_{-}}^{t_{+}}\sqrt{s_i-s}|\dot{\widetilde{\gamma}}(s)|_{g}^{2}ds\right)-Cr\\
\geq & \frac{1}{2r}\left(\int_{t_{-}}^{t_{+}}|\dot{\widetilde{\gamma}}(s)|_{g}ds\right)^{2}-Cr\\
\geq & \frac{1}{2r}\left(\frac{1}{2}Ar\right)^{2}-Cr,
\end{align*}
so that $A^{2}r\leq 16Cr$, a contradiction if $A>0$ is chosen
sufficiently large.
\end{proof}

The next claim shows that the $\ell$-centers of $(y_i,-1-\frac{1}{2}r^2)$ lie in a bubble smaller than that from Claim \ref{claim-usedincontainedinV}.

\begin{claim}
\label{claim-theproofissimilar} If $A>0$ is sufficiently large, then any $\ell$-center $(w_i,-1-r^2)$ of $(y_i,-1-\frac{1}{2}r^2)$ satisfies $w_{i}\in\mathcal{C}_{\frac{A}{4}r,i}$.
\end{claim}

\begin{proof}[Proof of Claim \ref{claim-theproofissimilar}] 
Let $\sigma, \sigma_i:[-1-\frac{1}{2}r^2,-1]\to X_{\text{reg}}\times [-1-\frac{1}{2}r^2,-1]$ be the integral curves of $\partial_{\mathfrak{t}}, (\psi_i^{-1})_{\ast}\partial_t$ respectively which satisfy $\sigma(1-\frac{1}{2}r^2) =\sigma_i(1-\frac{1}{2}r^2) = (y_{\ast},-1-\frac{1}{2}r^2)$. Then as $(\psi_{i}^{-1})_{\ast}\partial_t \to \partial_{\mathfrak{t}}$ in $C_{\operatorname{loc}}^{\infty}(X_{\text{reg}}\times (-\infty,0))$, the continuous dependence of ODE solutions on parameters implies that $\sigma_i$ is well-defined for sufficiently large $i\in \mathbb{N}$  (in particular, its trajectory stays in $U_i$), and $\sigma_i \to \sigma$ uniformly as $i\to \infty$. Since $\sigma(-1)=(y_{\ast},-1)$ and $\sigma_i(-1)=\left( (\psi_{i,-1}^{-1}\circ \psi_{i,-1-\frac{1}{2}r^2})(y_{\ast}), -1\right)$, it therefore follows that $(\psi_{i,-1}^{-1}\circ \psi_{i,-1-\frac{1}{2}r^2})(y_{\ast}) \to (y_{\ast},-1)$. Consequently, 
$$\psi_{i,-1-\frac{1}{2}r^2}(y_{\ast}) \in \psi_{i,-1}(B(x_{\ast},2r)\setminus \overline{B}(x_{\ast},\tfrac{1}{2}r)),$$
and in particular $y_i \in \mathcal{C}_{\frac{A}{8}r,i}$ for $i\geq \underline{i}(A,r)$ sufficiently large. As $\psi_{i,-1}(\partial B(x_{\ast},\frac{A}{4}r))$
and $\psi_{i,-1}(\partial B(x_{\ast},2r))$ disconnect $\mathcal{V}$,
the remainder of the proof 
follows that of Claim \ref{claim-usedindisappearing}. 
\end{proof}

Again using Lemma \ref{containedinV},
observe that since $\psi_{i,-1}(\partial B(x_{\ast},\frac{A}{4}r))$
and $\psi_{i,-1}(\partial B(x_{\ast},Ar))$ disconnect $\mathcal{V}$
for sufficiently large $i\in\mathbb{N}$, and because $\psi_{i}^{\ast}\widetilde{g}_{i}\to g$
in $C_{\text{loc}}^{\infty}(X_{\textnormal{reg}} \times (-\infty,0))$, we have that
\[
d_{\widetilde{g}_{i,-1-r^{2}}}\left(\mathcal{C}_{\frac{A}{4}r,i},\mathcal{V}\setminus\mathcal{C}_{Ar,i}\right)\geq cAr,
\]
for $r =r(A)>0$ sufficiently small, $i\geq \underline{i}(A,r)$ is sufficiently large, and $c>0$ independent of $r,A,i$. By Claims \ref{claim-usedindisappearing} and \ref{claim-theproofissimilar}, we thus find that $d_{\widetilde{g}_{i,-1-r^{2}}}(z_{i},w_{i})\geq cAr$,
which contradicts that fact that $d_{\widetilde{g}_{i,-1-r^{2}}}(w_{i},z_{i})\leq Cr$ if we
choose $A>0$ large, $r=r(A)>0$ small, and then $i\geq \underline{i}(A,r)$. 
\end{proof}

Using Lemma \ref{disappearing}, we now show that either every orbifold singularity lies on $E$, or $E$ is pulled entirely into some orbifold singularity. Recall that $E^{\circ}$ is the Hausdorff limit of $E_i$ in $X_{\textnormal{reg}}$ as defined in \eqref{eq: Edef}, and that $E$ is the closure of $E^{\circ}$ in $X$.

\begin{prop} \label{dichotomy}
\begin{enumerate}
    \item If $E^{\circ}=\emptyset$, then $X \setminus X_{\textnormal{reg}}$ is a single point $x_{\ast}$, and for any $r>0$, $\widetilde{E}\subseteq\mathcal{C}_{r,i}$
for sufficiently large $i\geq \underline{i}(r)$, where $\mathcal{C}_{r,i}$
is the connected component of $\mathcal{V}\setminus\psi_{i,-1}(\partial B(x_{\ast},r))$
containing $\psi_{i,-1}(\partial B(x_{\ast},\frac{1}{2}r))$. 

\item If $E^{\circ}\neq\emptyset$, then $X \setminus X_{\textnormal{reg}} \subseteq E = \overline{E^{\circ}}$. 
\end{enumerate}
\end{prop}

\begin{proof}
\begin{enumerate}
    \item Fix any $x_{\ast}\in X \setminus X_{\textnormal{reg}}$. For small $r>0$, set $\Omega_{r}:=\overline{B}(x_{\ast},r^{-1})\setminus B(X \setminus X_{\textnormal{reg}},r)$.
Then $E^{\circ}=\emptyset$ implies that for $i\geq \underline{i}(r)$
sufficiently large, we have $\psi_{i,-1}(\Omega_{r})\cap \widetilde{E}=\emptyset.$
After slightly modifying $\Omega_r$, we may assume that $\Omega_{r}$ has smooth boundary. By Lemma \ref{containedinV}, we see that $\psi_{i,-1}(\Omega_r) \subseteq \mathcal{V}$ for all $r>0$ small and $i\geq \underline{i}(r)$ large, so that each of its connected components disconnects $\mathcal{V}$ by Lemma \ref{JordanBrouwer}. By Lemma \ref{disappearing}, we also have that $\widetilde{E} \cap\mathcal{C}_{r,i}\neq\emptyset$ when $i\geq \underline{i}(r)$. Because $\widetilde{E} \cap\psi_{i,-1}(\Omega_{r})=\emptyset$, this means that $\widetilde{E}\subseteq\mathcal{C}_{r,i}$
for sufficiently large $i \geq \underline{i}(r)$. 

Suppose by way of contradiction that there exists $x_{\ast}' \in X\setminus X_{\text{reg}}$ distinct from $x_{\ast}$. Letting $\mathcal{C}_{r,i}'$ be the connected component of $\mathcal{V}\setminus \psi_{i,-1}(\partial B(x_{\ast}',r))$ containing $\psi_{i,-1}(\partial B(x_{\ast}',\frac{1}{2}r))$ for $r>0$ sufficiently small, we then find that $\widetilde{E} \cap  \mathcal{C}_{r,i}' \neq \emptyset$ by replacing $x_{\ast}$ with $x_{\ast}'$ in Lemma \ref{disappearing}. That is, $\mathcal{C}_{r,i}'$, $\mathcal{C}_{r,i}$ are disjoint connected components of $\mathcal{V}\setminus \psi_{i,-1}(\Omega_{r})$ which each intersect the connected set $\widetilde{E}$. This is a contradiction.

\item Fix $x_{\ast}\in X \setminus X_{\textnormal{reg}}$. By Lemma \ref{disappearing}, for
each $r>0$, we know that $\mathcal{C}_{r,i}\cap \widetilde{E} \neq\emptyset$ for all
large $i\geq \underline{i}(r).$ If for all small $r>0$, it was true that $\widetilde{E} \subseteq\mathcal{C}_{r,i}$ when $i\geq \underline{i}(r)$ is sufficiently large, we would have that
$$E_i=\psi_{i,-1}^{-1}(V_{i,-1} \cap \widetilde{E}) \subseteq \psi_{i,-1}^{-1}(V_{i,-1} \cap \mathcal{C}_{r,i}) \subseteq B(x_{\ast},r),$$
since $\psi_{i,-1}^{-1}(V_{i,-1} \cap \mathcal{C}_{r,i})$ is contained in the connected component of $X\setminus \partial B(x_{\ast},r)$ containing $\partial B(x_{\ast},\frac{1}{2}r)$. This implies that $E^{\circ} = \emptyset$, a contradiction. Hence, for any sufficiently small $r>0$, we conclude that $\widetilde{E}\cap\psi_{i,-1}(\partial B(x_{\ast},r))\neq\emptyset$
for sufficiently large $i\geq \underline{i}(r)$. This means that $E^{\circ}\cap\partial B(x_{\ast},r)\neq\emptyset$
for all sufficiently small $r>0$, so that $x_{\ast}\in E$. 
\end{enumerate}
\end{proof}

Let $(\varphi_{t})_{t\in(-\infty,0)}$ be the diffeomorphisms defined in Remark \ref{holoflow} so that $g_t=|t|\varphi_{t}^{\ast}g_{-1}$. We will often identify the $t=-1$ time slice of the metric flow from \eqref{Fconverge} with the orbifold soliton $X$.  Recall that $Y_0>0$ was chosen so that for all $(x,t) \in \widetilde{M}\times [0,T)$ and $\tau\in (0,t]$, we have $\mathcal{N}_{x_0,T}(\tau)\geq -Y_0$. Fix $x_{\ast}\in\text{argmin}_{X}f$. We will next show that the flow of $\partial_t -\nabla f \in \mathfrak{X}(X \times (-\infty ,0))$ maps geodesic balls in a given time slice into a corresponding $P^{\ast}$-parabolic neighborhood.

\begin{lemma} \label{Pstarcontainment}
There exists $c=c(Y_0)>0$ such that the following holds. For any $x\in X\setminus B_{g}(x_{\ast},-1,1)$
and $r\in(0,cd_{-1}^{-1}(x,x_{\ast})]$, 
\[
\bigcup_{t\in[-1-r^{2},-1]}\left(\varphi_{t}\left(B_{g}(x,-1,cr)\right)\times\{t\}\right)\subseteq P^{\ast}(x,-1;c^{-1}r,-r^{2},0).
\]
\end{lemma}

\begin{proof}
For $y\in B_{{g}}(x,-1,cr)$, we define $\gamma(t):=\varphi_{t}(y)$ for
$t\in[-1-r^{2},-1]$. Recall from Theorem \ref{bamconvergence} that
$$|t|(R_{{g}_t}+|\nabla {f}|_{{g}_t}^{2})=f_t-W$$
on $X \times (-\infty,0)$, where $W\geq -Y_0$. Then the fact that $R_{{g}_t}(\gamma(t))=\frac{1}{|t|}R_{{g}_{-1}}(y)$ and $|\dot{\gamma}(t)|_{{g}_{t}}^{2}=\frac{1}{|t|}|\nabla^{{g}_{-1}}f|_{{g}_{-1}}^{2}(y)$
allows us to estimate: 
\begin{equation*}
    \begin{split}
\mathcal{L}_{-1}(\gamma)&= \int_{-1-r^{2}}^{-1}\sqrt{-1-t}\left(R_{{g}_t}(\gamma(t))+|\dot{\gamma}(t)|_{g_{t}}^{2}\right)dt\\
&\leq  r\int_{-1-r^{2}}^{-1}\left(R_{{g}_{-1}}(y)+|\nabla^{{g}_{-1}}{f}|_{{g}_{-1}}^{2}(y)\right)dt\\
&=  r^{3}\left({f}_{-1}(y)-W\right) \\
&\leq  C(Y_0)r^3 {d}_{-1}^2(y,x_{\ast}),
\end{split}
\end{equation*}
where we have used the fact that $f_{-1}(y)-W\leq \frac{1}{4}(d_{-1}(y,x_{\ast})+C)^2$ for all $y\in X$, where $C>0$ is universal \cite[Theorem 1.1]{CMZ}. The reader is directed to \cite[Definition 22.1]{Bam3} for the definition of $\mathcal{L}$-length on a singular soliton. By \cite[Lemmas 15.9 \& 22.2]{Bam3}, we thus have the bounds
$$\frac{C(Y_0)^{-1}}{(-1-t)^2} \leq K(y,-1;\varphi_t(y),t) \leq \frac{C(Y_0)}{(-1-t)^{2}}\exp \left( -\frac{d_t^2(\varphi_t(y),y_t)}{C(Y_0)(-1-t)} \right)$$
for any $t\in [-1-r^2,-1]$, where $(y_t,t)$ is an $H_4$-center of $(y,-1)$. That is, 
$$d_{t}(y_{t},\varphi_{t}(y))\leq C(Y_0)\sqrt{-1-t}\leq C(Y_0)r$$
for all $t\in[-1-r^{2},-1]$. Therefore for all $t\in[-1-r^{2},-1]$, we have
\begin{equation*}
    \begin{split}
d_{W_{1}}^{g_{-1-r^{2}}}(\nu_{x,-1;-1-r^{2}},\nu_{\varphi_{t}(y),t;-1-r^{2}})&\leq   d_{W_{1}}^{g_t}(\nu_{x,-1;t},\delta_{\varphi_{t}(y)})\\
&\leq \ d_{W_{1}}^{g_t}(\nu_{x,-1;t},\delta_{x_{t}})+d_{t}(x_{t},\varphi_{t}(x))+d_{t}(\varphi_{t}(x),\varphi_{t}(y))\\
&\leq   C(Y_0)r,
\end{split}
\end{equation*}
where $(x_t,t)$ is an $H_4$-center of $(x,-1)$, and where we used the fact that $d_t(\varphi_t(x),\varphi_t(y)) = |t|^{\frac{1}{2}}d_{-1}(x,y)$.
\end{proof}

Using Lemma \ref{Pstarcontainment}, we can show that the high-curvature region of a ball in $X$ has small $\frac{1}{2}$-dimensional Minkowski content. This is roughly a consequence of Bamler's $L^p$-curvature estimates ($p>\frac{3}{2}$) and the self-similarity of the Ricci flow generated by $X$. 

\begin{prop} \label{highcurvaturesmall}
There exist $c=c(Y_0)>0$ and $D=D(Y_0)>0$ such that for all $x\in X \setminus\nolinebreak B_{{g}}(x_{\ast},1)$,
$${\rm vol}_{{g}}\left(\left\{ r_{{\Rm}}^{(X_{\textnormal{reg}},g)}< csr\right\} \cap B(x,r)\right)\leq Ds^{\frac{7}{2}}{\rm vol}_{{g}}\left(B(x,r)\right)$$
for all $s\in(0,1]$, where $r:=d_g^{-1}(x,x_{\ast})$.  
\end{prop}

\begin{proof}
By \cite[Theorem 2.31]{Bam3} with $r:=d_{g}^{-1}(x,x_{\ast})=d_{-1}^{-1}(x,x_{\ast})$,
we have that
$$\int_{-1-r^{2}}^{-1}\int_{P_{{g}}^{\ast}(x,-1;r)\cap(X\times\{t\})}\left( r_{\text{Rm}}^{(X_{\textnormal{reg}},g_t)}(x) \right)^{-\frac{7}{2}}d{g}_{t}(x)dt\leq C(Y_0)r^{\frac{5}{2}}.$$
Because 
\[
\sup_{y\in B_{{g}}(x,r)}R_{{g}_{-1}}(y)\leq C{d}_{-1}^{2}(x,x_{\ast}) = Cr^{-2}
\]
for some universal constant $C>0$, we know that
\begin{equation} \label{kappanoncollapsing}
\mbox{vol}_{{g}}(B(x,r))\geq\kappa(Y_0)r^{4},
\end{equation}
so that by Lemma \ref{Pstarcontainment},
\begin{equation*}
    \begin{split}
C(Y_0)r^{\frac{5}{2}}&\geq  \int_{-1-r^{2}}^{-1}\int_{P_{{g}}^{\ast}(x,-1;r)\cap(X\times\{t\})} \left( r_{\text{Rm}}^{(X_{\textnormal{reg}},g_t)}(y) \right)^{-\frac{7}{2}}dg_{t}(y)dt\\
&\geq  \int_{-1-r^{2}}^{-1}\int_{\varphi_{t}(B_{{g}}(x,-1,cr))} \left( r_{\text{Rm}}^{(X_{\textnormal{reg}},g_t)}(y) \right)^{-\frac{7}{2}}(y,t)dg_{t}(y)dt\\
&=  \int_{-1-r^{2}}^{-1}|t|^{\frac{1}{2}}\int_{\varphi_{t}(B_{{g}}(x,-1,cr))}\left( r_{\text{Rm}}^{(X_{\textnormal{reg}},g_{-1})}(\varphi_{t}(y)) \right)^{-\frac{7}{2}} {d}(\varphi_{t}^{\ast}{g}_{-1})(y)dt\\
&\geq  \int_{-1-r^{2}}^{-1}\int_{B_{{g}}(x,-1,cr)}\left( r_{\text{Rm}}^{(X_{\textnormal{reg}},g_{-1})}(y) \right)^{-\frac{7}{2}}d{g}_{-1}(y)dt\\
&=  r^{2}\int_{B_{{g}}(x,cr)} \left( r_{\text{Rm}}^{(X_{\textnormal{reg}},g)}(y) \right)^{-\frac{7}{2}}d{g}(y).
\end{split}
\end{equation*}
In other words, we have that
\begin{equation*}
    \begin{split}
\text{vol}_{{g}}\left(\left\{ r_{\text{Rm}}^{(X_{\textnormal{reg}},g)}<csr\right\} \cap B_{{g}}(x,r)\right) &\leq  (csr)^{\frac{7}{2}}\int_{B_{{g}}(x,r)} \left(r_{\text{Rm}}^{(X_{\textnormal{reg}},g)}(y) \right)^{-\frac{7}{2}}d{g}(y)\\
&\leq  C(Y_0)(sr)^{\frac{7}{2}}r^{\frac{1}{2}}\\
&\leq  C(Y_0)s^{\frac{7}{2}}\text{vol}_{{g}}(B(x,r)).
\end{split}
\end{equation*}
\end{proof}

If a connected holomorphic curve $\Sigma \subseteq X_{\textnormal{reg}}$ passes through a point $x \in X_{\textnormal{reg}}$ and intersects $\partial B(x,r)$, where $r \lesssim {d}^{-1}(x,x_{\ast})$, then a covering argument shows that $\Sigma$ has one-dimensional Minkowski content bounded below, so it must contain a point which is not too close to the high-curvature region. This is made precise in the following proposition, which makes use of a well-known lower bound for the area of holomorphic curves in $\mathbb{C}^2$.

\begin{prop} \label{arealowerbound}
Let $\Sigma\subseteq X_{\textnormal{reg}}$ be a ${J}$-holomorphic curve
and suppose that $x\in\Sigma\setminus B_{{g}}(x_{\ast},1)$ satisfies $\partial(\overline{\Sigma}\cap B_{{g}}(x,r))\subseteq\partial B_{{g}}(x,r)$
for some $r\in(0,\frac{c}{{d}(x,x_{\ast})}]$, where $x_{\ast}\in {\rm argmin}{f}$ and $c=c(Y_0)$. Then 
\[
\int_{\Sigma \cap B_{{g}}(x,r)} {\omega}
\geq c(Y_0)r^{2}.
\]
\end{prop}

\begin{proof}
We begin with the following claim which shows that some point in $\Sigma$ has a curvature bound in a ball of controlled radius.
\begin{claim} \label{claim-containmentwithcurvaturebounds}
For $\alpha\leq\overline{\alpha}(Y_0)$, there exists
$y\in\Sigma$ such that $B_{{g}}(y,\alpha r)\subseteq B_{{g}}(x,r)$ and
$r_{\text{Rm}}^{(X_{\textnormal{reg}},g)}(y)\geq 2\alpha r$. 
\end{claim}

\begin{proof}[Proof of Claim \ref{claim-containmentwithcurvaturebounds}]
Suppose that this fails so that $\Sigma\cap B_{{g}}(x,(1-\alpha)r)\subseteq\{r_{\text{Rm}}^{(X_{\textnormal{reg}},g)}<2\alpha r\}$.
By Proposition \ref{highcurvaturesmall}, we have that
\[
\text{vol}_{{g}} \left( \left\{ r_{\text{Rm}}^{(X_{\textnormal{reg}},g)}<4\alpha r\right\} \cap B_{{g}}(x,r)\right) \leq D\alpha^{\frac{7}{2}}\text{vol}_{{g}}(B_{{g}}(x,r))
\]
if $\alpha\leq \overline{\alpha}(Y_0)$ is sufficiently small. Choose a maximal set $\{x_{i}\}_{i=0}^{N}\subseteq\Sigma\cap B_{{g}}(x,(1-\alpha)r)$
containing $x$ such that ${d}(x_{i},x_{j})\geq\alpha r$ whenever
$i\neq j$, so that $\bigcup_{i=0}^{N}B_{{g}}(x_{i},2\alpha r)\supseteq\Sigma\cap B_{{g}}(x,(1-\alpha)r)$. 

By assumption, we have $\Sigma\cap\partial B_{{g}}(x,s)\neq\emptyset$ for
all but finitely many $s\in[0,r]$. This means that $$([{d}(x_{i},x)-2\alpha,{d}(x_{i},x)+2\alpha])_{i=1}^N$$
covers $[0,1-\alpha]$, resulting in the lower bound $N\geq (4\alpha)^{-1}$. By combining the volume lower bound \eqref{kappanoncollapsing}, a volume upper bound \cite[Corollary 1.3]{CMZ}, and Proposition \ref{highcurvaturesmall}, we obtain
\begin{equation*}
\begin{split}
\text{vol}_{{g}} \left( \bigcup_{i=0}^{N}B_{{g}}(x_{i},\alpha r)\right) &\geq (4\alpha)^{-1}\kappa(Y_0)(\alpha r)^{4}\geq c'(Y_0)\alpha^{3}r^4>D\alpha^{\frac{7}{2}}\text{vol}_{{g}}(B_{{g}}(x,r))\\ 
& \geq \text{vol}_{{g}} \left( \left\{ r_{\text{Rm}}^{(X_{\textnormal{reg}},g)} <4\alpha r\right\} \cap B_{{g}}(x,r)\right)
\end{split}
\end{equation*}
if $\alpha\leq\overline{\alpha}(Y_0)$. This implies that there must
exist $i\in\{0,\ldots,N\}$ such that 
$$B_{{g}}(x_{i},\alpha r)\cap\{r_{\text{Rm}}^{(X_{\textnormal{reg}},g)} \geq 4 \alpha r\}\neq\emptyset.$$
Since $r_{\text{Rm}}^{(X_{\textnormal{reg}},g)}$ is $1$-Lipschitz, this gives us that $r_{\text{Rm}}^{(X_{\textnormal{reg}},g)}(x_{i})\geq2\alpha r$.
\end{proof}

Let $y\in\Sigma$ be as in Claim \ref{claim-containmentwithcurvaturebounds}. By
\cite[Comment 1, p.178, and Proposition 4.3.1(ii)]{holocurves}
(see also \cite[Lemma 5.2]{CJL}), there exists $c'=c'(Y_0)>0$ such that 
$$
\int_{\Sigma \cap B_{{g}}(y,c'\alpha r)} {\omega}
 \geq \frac{\pi}{4}(c'\alpha r)^{2},
$$
so the proposition follows from the fact that $B_{{g}}(y,c'\alpha r)\subseteq  B_{{g}}(x,r)$. 
\end{proof}

In order to use Proposition \ref{arealowerbound}, we need to show that if $E$ is unbounded, then we can construct a holomorphic curve with finite volume but infinite diameter.  

\begin{prop} \label{bigcurve}
If $E$ is unbounded, then there exists a holomorphic curve $E'\subseteq X_{\textnormal{reg}}$ satisfying
the following:
\begin{enumerate}
    \item $\int_{E'} {\omega} \leq2\pi$.
\item There is an (at most countable) subset $\{x_{i}\,|\,i\in\mathcal{I}\}\subseteq X \setminus X_{\textnormal{reg}}$ such that $E'\cup\{x_{i}\,|\,i\in\mathcal{I}\}$ is connected.
\item $E'$ is unbounded.
\end{enumerate}
Moreover, if $E$ is bounded, then it is connected.
\end{prop}

\begin{proof} We will show that (i)-(iii) hold for a holomorphic curve $E'$ which is a union of connected components of $E^{\circ}$. Recall that $\psi_{i}:U_{i}\to V_{i}\subseteq M$ are the open embeddings realizing
Cheeger-Gromov convergence and that by Proposition \ref{dichotomy}, $E_i$ converges locally in the Hausdorff sense to a holomorphic curve $E^{\circ} \subseteq X_{\textnormal{reg}}$ with $X \setminus X_{\textnormal{reg}} \subseteq E = \overline{E^{\circ}}$. In particular, (iii) is satisfied. Because $\mathcal{H}_{{\widetilde{g}}_{i,-1}}^2(\widetilde{E}) =\int_{\widetilde{E}}\widetilde{\omega}_{i,-1}= 2\pi$ for all $i\in \mathbb{N}$ and $\psi_{i,-1}^{\ast}{\widetilde{g}}_{i,-1} \to {g}$ in $C_{\operatorname{loc}}^{\infty}(X_{\operatorname{reg}})$, we know that for any compact subset $L \subseteq X \setminus X_{\textnormal{reg}}$, we have 
\begin{equation} \label{eq:limsup} \limsup_{i\to \infty} \mathcal{H}_{{g}}^2(E_i \cap L)\leq 2\pi.\end{equation}
Next, recall from \cite[Section 16.1, Proposition 1]{Chirka} that any sequence of holomorphic curves in $\mathbb{C}^2$ with locally uniformly bounded area also subsequentially converges in the sense of currents to a holomorphic 2-chain\footnote{A holomorphic $p$-chain is an integer linear combination of currents of integration along irreducible $p$-dimensional analytic sets.} supported on the Hausdorff limit of these curves. Using this together with the fact that convergence in the sense of currents is local, arguing as in the proof of Lemma \ref{Hausdorff} allows us to conclude that this property also holds for a sequence of K\"ahler structures where the complex structures and metrics are converging locally smoothly. Thus, from \eqref{eq:limsup} we obtain (i).

Set $\{x_i\,|\, i\in \mathcal{I} \} := X \setminus X_{\textnormal{reg}}$. Let $(\mathbb{V},\mathbb{E})$ be the graph with vertex set $\mathbb{V}$ equal to the connected components $E^v$ of $E^{\circ}$, and with edge set $\mathbb{E}$, with an edge between $v_1,v_2 \in \mathbb{V}$ if and only if $\overline{E^{v_1}} \cap \overline{E^{v_2}} \neq \emptyset$. We may assume that each $v\in \mathbb{V}$ corresponds to a bounded connected component $E^v$ of $E^{\circ}$, for otherwise $E^v$ already satisfies (i)--(iii). Suppose by way of contradiction that $(\mathbb{V},\mathbb{E})$ has a connected component $(\mathbb{V}',\mathbb{E}')$ such that $\cup_{v\in \mathbb{V}'}E^v$ is bounded. Then there exists a neighborhood $\mathcal{B}$ of $\overline{\cup_{v\in \mathbb{V}'}E^v}$ with smooth boundary contained in $X_{\textnormal{reg}}$ such that $\mathcal{B}\cap \overline{E^w}=\emptyset$ for all $w \in \mathbb{V}\setminus \mathbb{V}'$. We know that $\mathbb{V}\setminus \mathbb{V}'\neq \emptyset$ since $E^{\circ}$ is unbounded. Hence, for $i\in \mathbb{N}$ sufficiently large, we have $E_i \cap \mathcal{B} \Subset \mathcal{B}$ so that $\psi_{i,-1}(\partial \mathcal{B})$ separates $\mathcal{V}$ into at least two connected components with $\widetilde{E}$ not contained in any one connected component. This contradicts the fact that $\widetilde{E}$ is connected. For any connected component $(\mathbb{V}',\mathbb{E}')$ of $(\mathbb{V},\mathbb{E})$, we may therefore take $E':=\cup_{v\in \mathbb{V}'}E^v$. 

If $E$ is bounded, then the argument of the previous paragraph shows that it is connected.
\end{proof}

We now show that a curve $E'$ as in Proposition \ref{bigcurve} cannot exist using the area lower bound from Proposition \ref{arealowerbound}. This will allow us to conclude that there are at most finitely many orbifold singularities.

\begin{prop} \label{finiteorb} $E$ is compact. Moreover, $X$ has at most finitely many orbifold singularities.
\end{prop}

\begin{proof}
Suppose by way of contradiction that $E$ is unbounded, so that Proposition \ref{bigcurve} can be applied to produce $E'$. Set 
$$r(x):=\frac{{d}(x,x_{\ast})+1}{\min(c,\frac{1}{2})}$$ for $x\in X$, where $c=c(Y_0)$ is as in Proposition \ref{arealowerbound}. By Proposition \ref{bigcurve}(ii) and (iii), we know
that $r(\overline{E'})=[r_{\min},\infty)$ for some $r_{\min}\geq0$.
Choose a maximal collection $\{y_{i}\}_{i\in\mathbb{N}}$ in $E '\cap (X_{\textnormal{reg}} \setminus B(x_{\ast},r_{\min}))$ such that if $s_{i}:=r(y_{i})^{-1}$, then $\{B(y_{i},s_{i})\}$ are pairwise disjoint. Because $x\mapsto \frac{1}{r(x)}$ is $\frac{1}{2}$-Lipschitz, it follows that $E' \subseteq\bigcup_{i}B(y_{i},10s_{i})$,
so that $\bigcup_{i}r\left(B(y_{i},10s_{i})\right)\supseteq[r_{\min},\infty)$.
Fix $j_0\in \mathbb{N}$ such that $r_{\min}\leq2^{j_0}$, and let $\mathcal{I}_{j}\subseteq\mathbb{N}$
be the set of $i\in\mathbb{N}$ satisfying $r\left(B(y_{i},10r_{i})\right)\cap[2^{j},2^{j+1}]\neq\emptyset$.
If $j\geq j_0$, we then have that $|\mathcal{I}_{j}|\geq c2^{2j}$, since $[2^j,2^{j+1}] \subseteq \cup_{i \in \mathcal{I}_j} [r(y_i)-10 s_i,r(y_i)+10s_i]$ and  $s_{i}\leq2^{-j}$
for any $i\in\mathcal{I}_{j}$. Moreover, we see that $\mathcal{I}_j \cap \mathcal{I}_k= \emptyset$ whenever $|j-k|\geq 2$. It subsequently follows from Proposition \ref{arealowerbound} that 
\begin{equation*}
\begin{split}
\int_{E'} \omega\geq  \sum_{i} \int_{E'\cap B(y_i,s_i)} \omega \geq  \frac{1}{3}\sum_{j \geq j_0}\sum_{i\in\mathcal{I}_{j}} \int_{E'\cap B(y_i,s_i)} \omega \geq  \sum_{j\geq j_0}\sum_{i\in\mathcal{I}_{j}}c(Y_0)2^{-2j}\geq  \sum_{j \geq j_0}c(Y_0)=\infty,
\end{split}
\end{equation*}
thereby contradicting Proposition \ref{bigcurve}(i). 
Finally, if $X$ had infinitely many orbifold singularities, then by Proposition \ref{dichotomy}, they would all lie along $E$. The discreteness of the singular set would then imply that $E$ is unbounded. This is a contradiction.
\end{proof}

\section{Boundedness of Curvature}

In this section, we first study steady gradient K\"ahler-Ricci soliton surfaces arising as singularity models of the Ricci flow on compact K\"ahler surfaces and show that they must be Ricci-flat. Indeed, in the general Riemannian setting, the asymptotics of such solitons at spatial infinity are well understood  \cite{BCDMZ}. However, as we shall soon see, the assumption that the soliton is K\"ahler further limits the asymptotics and forces the soliton either to be Calabi-Yau or asymptotic to $\mathbb{P}^1 \times \mathbb{C}$.  
Standard facts about compact holomorphic curves in K\"ahler surfaces then allows us rule out the latter case.

The precise result is the following.
Recall that a singularity model of Ricci flow is defined to be a blow-up model of a compact Ricci flow at a finite time singularity; see \cite[Definition 1.1]{BCDMZ}. 

\begin{prop} \label{nosteadies}
    Every steady gradient K\"ahler-Ricci soliton surface { with bounded curvature} is an ALE Calabi-Yau manifold.
\end{prop}

\begin{proof}
Let $(X,g,f,J)$ be a steady gradient K\"ahler-Ricci soliton surface so that
$$\operatorname{Ric}_g+\nabla^2 f = 0\qquad\textrm{and}\qquad\mathcal{L}_{\nabla f}J=0,$$ and assume that $(X,\,g)$ is a singularity model as in the statement of the proposition. We first apply \cite[Theorem 1.3 \& Proposition 3.1]{BCDMZ} and the fact that any complete shrinking gradient K\"ahler-Ricci soliton splitting $\mathbb{R}$ actually splits $\mathbb{R}^2$ isometrically to deduce that the tangent flow at infinity of the steady soliton is unique and is diffeomorphic to either $\IP^1\times\IC$ or $\IC^2/\Gamma$ for some finite subgroup $\Gamma \subseteq U(2)$ acting freely on $\mathbb{C}^{2}\setminus\{0\}$. If the tangent flow at infinity is $\IC^2/\Gamma$, then $(X,g)$ is an ALE Calabi-Yau manifold by \cite[Proposition 3.1]{BCDMZ}.

Suppose by way of contradiction that the  tangent flow at infinity is diffeomorphic to $\mathbb{P}^1 \times \mathbb{C}$. Being a shrinking gradient K\"ahler-Ricci soliton, \cite[Theorem A]{CCD} guarantees that the tangent flow is biholomorphic to $\IP^1\times\IC$. Let $(g_t)_{t\in \mathbb{R}}$ be the canonical Ricci flow associated to $(X,\,g,\,f)$ with $g_0=g$. Fix any $x_0 \in X$, a sequence $\lambda_i \searrow 0$, and let $(w_i,-\lambda_i^{-1})$ be $H_4$-centers of $(x_0,0)$. Then we can apply \cite[Theorem 9.31]{Bam3} to obtain a precompact exhaustion $(U_i)$ of $\mathbb{P}^1 \times \mathbb{C}$, along with diffeomorphisms $\Psi_i: U_i \to V_i \subseteq X$, such that 
\[
    g_i:=\lambda_i\Psi_i^* g_{-\lambda_i^{-1}} \to \bar g,
    \qquad
    J_i:= \Psi_i^*J
    \to \bar J,
    \qquad
    \Psi_i^{-1}(w_i)\to w_{\infty},
\]
in $C_{\operatorname{loc}}^{\infty}(\mathbb{P}^1 \times \mathbb{C})$, where $\bar g$ and $\bar J$ denote the standard metric and complex structure on $\IP^1\times \IC$ respectively, and $w_{\infty} \in \mathbb{P}^1 \times \{0\}$.

Note that $(\IP^1,j)\hookrightarrow (\IP^1\times \IC,\bar J), z\mapsto (z,0)$ is a smooth $\bar J$-holomorphic curve with trivial self-intersection, where $j$ denotes the standard complex structure on $\IP^1.$ 
By \cite[Corollary 2.3]{CCD}, for $i$ sufficiently large, there are smooth $J_i$-holomorphic curves $v_i:(\IP^1,j)\to (U_i,J_i)$ with trivial self-intersection and $w_{\infty} \in v_i(\mathbb{P}^1)$. Thus, on the steady soliton $(X,J),$ there are $J$-holomorphic curves $\IP^1 \cong C := (\Psi_i \circ v_i)(\mathbb{P}^1)$ with trivial self-intersection.
By the adjunction formula and the steady K\"ahler-Ricci soliton equation, we derive from this that
\begin{align*}
    -2 &= \text{deg}(K_C)
    = K_X \cdot C + C \cdot C\\
    &= -\int_C[\operatorname{Ric}_{\omega}]+0
    = -\int_C [\sqrt{-1}\partial\bar\partial f] = 0,
\end{align*}
where $\operatorname{Ric}_{\omega}$ denotes the Ricci form of $g$. This is a contradiction.
\end{proof}

We now return to the study of the shrinking orbifold soliton $X$ obtained as the $\mathbb{F}$-limit \eqref{Fconverge}. We will show that the curvature of $X$ is bounded. (If $X$ was known to be smooth, then we would already know this by \cite{LiWang}.) The idea is to first rule out faster-than-quadratic curvature growth using a topological argument. Next, we show that the curvature actually grows slower than quadratic, for otherwise an appropriate rescaled sequence converges to a non-Ricci-flat steady gradient K\"ahler-Ricci soliton singularity model, the existence of which contradicts Proposition \ref{nosteadies}. Finally, we show that the curvature cannot blow up at a slower-than-quadratic rate, since then a rescaled sequence converges to $\mathbb{R}$ times a three-dimensional $\kappa$-solution\footnote{A three-dimensional $\kappa$-solution is a complete, ancient, non-flat three-dimensional solution to the Ricci flow with bounded nonnegative curvature, which is $\kappa$-noncollapsed at all scales for some $\kappa>0$.} which can be ruled out using the classification of such flows. 
{ 

\begin{theorem} \label{bddcurve} Let $X$ be an orbifold shrinking gradient K\"ahler-Ricci soliton arising as a tangent flow of a Ricci flow on a compact K\"ahler surface, as in Theorem \ref{bamconvergence}. If $X$ has only finitely many singularities, then $X$ has bounded curvature. 
\end{theorem}
}

\begin{proof} As usual, we identify $X$ with the $t=-1$ time slice of the metric flow limit from \eqref{Fconverge}. Recall that the regular part of the metric flow corresponds to the induced K\"ahler-Ricci flow spacetime given in Definition \ref{standardsoliton}, and that with respect to this identification, ${g}_t =|t|\varphi_{t}^{\ast}{g}$, where $\partial_{t}\varphi_{t}(x)=\frac{1}{|t|}\nabla^{{g}}{f}(\varphi_{t}(x))$. We also set ${f}_{t}:=\varphi_{t}^{\ast}{f}$. Suppose by way of contradiction that there is a sequence $\bar{x}_{i}\in X$ diverging to spatial infinity such that $\lim_{i\to \infty} r_{\Rm}^{g}(\bar{x}_i,-1)=0$. 

The next claim is a point-picking argument which finds appropriate basepoints at which a rescaled sequence of the flow converges to a smooth limit.

\begin{claim} \label{claim-pointpicking}
There are sequences $A_i \nearrow \infty$ and $(x_{i},t_{i})\in X\times(-\infty,-1]$ such that the following hold:
\begin{enumerate}
\item $r_{i}:=r_{{\Rm}}^{{g}}(x_{i},t_{i})\leq r_{{\Rm}}^{{g}}(\bar{x}_{i},-1)$;
\item $\sup_{P_{{g}}^{\ast-}(x_i,t_i;A_i r_i)}r_{{\Rm}}^{{g}}\geq\frac{1}{10}r_{i}$;
\item $(x_{i},t_{i})\in P_{{g}}^{\ast-}(\bar{x}_{i},-1;10 A_{i}r_{\Rm}^{{g}}(\bar{x}_i,-1))$;
\item $\lim_{i\to \infty} A_i r_i = \lim_{i \to \infty} A_i r_{\text{Rm}}^g(\bar{x}_i,-1)=0$.
\end{enumerate}
\end{claim}

\begin{proof}[Proof of Claim \ref{claim-pointpicking}]
Choose $A_i \to \infty$ such that $\lim_{i\to \infty}A_i r_{\Rm}^{g}(\bar{x}_i,-1) =0$. The hypotheses of \cite[Claim 10.5]{Bam1} with $A$ replaced by $A_i$ and $(x,t)$ replaced by $(\bar{x}_i,-1)$ hold for sufficiently large $i\in \mathbb{N}$. This yields $(x_i,t_i)$ satisfying (i)--(iii), and thus also (iv). 
\end{proof}

Define ${g}_{i,t}:=r_{i}^{-2}{g}_{r_{i}^2t+t_{i}}$ for $t\in(-\infty,0]$,
so that by Claim \ref{claim-pointpicking}(ii), we have
\[
\sup_{P_{{g}_{i}}^{\ast -}(x_i,0;A_i)}r_{{\Rm}}^{{g}_{i}}\geq\frac{1}{10}
\]
for all $i\in\mathbb{N}$, and $r_{{\Rm}}^{g_{i}}(x_{i},0)=1$. By no-local collapsing \cite[Theorem 6.1]{Bam1}, we can pass to a subsequence to obtain smooth
pointed Cheeger-Gromov convergence 
\[
(X,({g}_{i,t})_{t\in(-\infty,0]},{J},x_{i})\to(X_{\infty},(g_{\infty,t})_{t\in(-\infty,0]},J_{\infty},x_{\infty})
\]
to a complete K\"ahler-Ricci flow satisfying $r_{{\Rm}}^{g_{\infty}}(x_{\infty},0)=1$
and 
\[
\sup_{X_{\infty}\times(-\infty,0]}r_{{\Rm}}^{g_{\infty}}\geq\frac{1}{10}.
\]
By passing to a subsequence, we may assume that we are in one of the following three cases.\\ 

\begin{description}
 \item[Case 1] $\lim_{i\to\infty}r_{i}^{2}{f}_{t_i}(x_{i})=0$. \\

\noindent In this case there exist $y_i \in X$ such that 
$$\lim_{i\to \infty} |{\Rm}_{g_{t_i}}|_{{g}_{t_i}}^{-1}(y_i) f_{t_i}(y_i) =0.$$ 
However,
$$(\partial_t - \nabla {f})\left( |{\Rm}_{g_t}|_{{g}_t}^{-1}{f_t} \right) \leq 0,$$
since $\mathcal{L}_{\partial_{\mathfrak{t}}-\nabla f}g_t = -\frac{1}{|t|} g_t$. As a result, $z_i := \varphi_{t_i}(y_i)$ satisfy  
$$\lim_{i\to \infty} |{\Rm}_{g_{-1}}|_{{g}_{-1}}^{-1}(z_i){f}_{-1}(z_i)\leq \liminf_{i \to \infty}|{\Rm}_{g_{t_i}}|_{{g}_{t_i}}^{-1}(y_i){f}_{t_i}(y_i)=0.$$
By \cite[Proof of Proposition 8]{CFSZ}, there is a sequence of disjoint domains $\Omega_i \subseteq X$, each diffeomorphic to a fixed non-flat ALE Calabi-Yau manifold $Z$ {
 such that for any compact subset $L\subseteq X$, we have $L\cap \Omega_i =\emptyset$ for sufficiently large $i\geq \underline{i}(L)$. In particular, by passing to a subsequence, we may assume that $\Omega_i \subseteq X_{\textnormal{reg}}$ for all $i\in \mathbb{N}$. For any $N\in \mathbb{N}$, if $i\geq \underline{i}(N)$ is sufficiently large, then $\psi_i$ restricts to a smooth embedding  $\sqcup_{i=1}^N \Omega_i \hookrightarrow \widetilde{M}$. This contradicts \cite[Theorem 12]{CFSZ} (see also \cite[Proposition 1]{CFSZcorrection}).}\\

\item[Case 2] $\lim_{i\to\infty}r_{i}^{2}{f}_{t_i}(x_{i})\in(0,\infty)$.\\

\noindent By the same argument as Case 1, we obtain that $|\Rm_g|_g \leq C(f+C)$ for some $C>0$. Furthermore, replacing $y_i$ with $\varphi_{t_i}(y_i)$ again allows us to assume that $t_i=-1$. Then \cite[Proof of Proposition 10]{CFSZ} implies the existence
of a complete noncompact non-Ricci-flat steady gradient K\"ahler-Ricci soliton which is also a singularity model in the sense of \cite[Definition 1.1]{BCDMZ}. The existence of such a soliton contradicts Proposition \ref{nosteadies}.\\

\item[Case 3] $\lim_{i\to\infty}r_{i}^{2}{f}_{t_i}(x_{i})=\infty$. \\

\noindent Define $\lambda_{i}:=r_{i}|\nabla f_{t_i}|_{g_{t_i}}(x_{i})$, and set
$${f}_{i,t}(x):=\lambda_{i}^{-1}({f}_{t_i+r_i^2 t}(x)-{f}_{t_i}(x_{i})),$$
so that ${f}_{i,0}(x_{i})=0$,
\[
|\nabla {f}_{i,0}|_{{g}_{i,0}}^{2}(x_{i})=r_{i}^{2}\lambda_{i}^{-2}|\nabla {f}_{t_i}|_{{g}_{t_i}}^{2}(x_{i})=1,
\]
and
\[
\lambda_{i}^{-1}\operatorname{Ric}_{{g}_{i,0}}+\nabla^{g_{i,0}} \nabla^{g_{i,0}}{f}_{i,0}=\lambda_{i}^{-1}r_{i}^{2}{g}_{i,0}.
\]
From $R_{{g}}+|\nabla {f}|_{{g}}^{2}={f}$ and Case 1, we have that
\[
\lim_{i\to\infty}\lambda_{i}^{2}=\lim_{i\to\infty}r_{i}^{2}|\nabla {f}|_{{g}}^{2}(x_{i},t_{i})\geq\frac{1}{2}\lim_{i\to\infty}r_{i}^{2}f(x_{i},t_{i})=\infty.
\]
We can therefore pass to a subsequence and appeal to local elliptic regularity
to obtain $C_{\text{loc}}^{\infty}(X_{\infty})$ convergence of ${f}_{i}$
to $f_{\infty}\in C^{\infty}(X_{\infty})$. Moreover, it is clear that
\[
f_{\infty}(x_{\infty},0)=0,\,\,\,\,\,\,\,|\nabla^{g_{\infty}} f_{\infty}|(x_{\infty},0)=1,\,\,\,\,\,\,\,\,\nabla^{g_{\infty}}\nabla^{g_{\infty}}f_{\infty}\equiv0.
\]
Bochner's formula then gives
\[
0=\Delta|\nabla^{g_{\infty}} f_{\infty}|^{2}=2{\operatorname{Ric}_{g_{\infty}}}(\nabla^{g_{\infty}} f_{\infty},\nabla^{g_{\infty}} f_{\infty})=2{\operatorname{Ric}_{g_{\infty}}}(J\nabla^{g_{\infty}} f_{\infty},J\nabla^{g_{\infty}} f_{\infty}).
\]
Because $(X_{\infty},g_{\infty,0})$ splits as a product, \cite[Theorem
1.1]{CK20} and the remark thereafter implies that $(X_{\infty},(g_{\infty,t})_{t\in(-\infty,0]})=(X_{\infty}'\times\mathbb{R},(g_{\infty,t}'+g_{\mathbb{R}})_{t\in(-\infty,0]})$
for some non-flat complete ancient three-dimensional solution $(X_{\infty}',(g_{\infty,t}')_{t\in(-\infty,0]})$
of Ricci flow with bounded curvature which is $\kappa$-noncollapsed
at all scales for some $\kappa>0$. By \cite[Corollary 2.4]{CH09},
$g_{\infty,t}'$ has nonnegative curvature for all $t>0$, so is a
$\kappa$-solution of Ricci flow. In particular, it must be the Bryant
soliton, a quotient of the round cylinder, a quotient of the round $\mathbb{S}^3$, or Perelman's type II oval, by \cite[Theorem 1.5]{BDS21} and \cite[Theorem 1.3]{brendle}.  However, because ${\operatorname{Ric}}_{g_{\infty,t}'}$
has at least one zero eigenvalue at any point of $X_{\infty}'$, we
know that $g_{\infty,t}'$ cannot be the Bryant soliton, Perelman's type II oval, or a quotient of the round $\mathbb{S}^3$, hence must be the cylinder $\mathbb{S}^2 \times \mathbb{R}$ or $\mathbb{S}^{2}\tilde{\times}\mathbb{R}$ equipped with the standard round metric, where 
\[
\mathbb{S}^{2}\tilde{\times}\mathbb{R}=(\mathbb{S}^{2}\times\mathbb{R})/\sim
\]
and $(\theta,x)\sim(\theta',x')$ if and only if $(\theta',x')=(-\theta,-x)$.
In the latter case, the universal cover
\[
\mathbb{S}^{2}\times\mathbb{R}^{2}\to(\mathbb{S}^{2}\tilde{\times}\mathbb{R})\times\mathbb{R}=X_{\infty}
\]
equipped with the complex structure $\overline{J}$ obtained from pulling back $J_{\infty}$ is K\"ahler, so by \cite[Claim 3.3]{CCD}, it must be the standard complex structure $\overline{J}$ on $\mathbb{S}^{2}\times\mathbb{R}^{2}$. However, $\overline{J}$ is not preserved
by deck transformations of the covering map $(\theta,x,y)\mapsto(-\theta,-x,y)$ (in fact, the antipodal
map of $\mathbb{S}^{2}\cong\mathbb{P}^{1}$ in homogeneous coordinates
is $[z_{0}:z_{1}]\mapsto[-\overline{z}_{1}:\overline{z}_{0}]$), which is not possible. We consequently deduce that $(X_{\infty},(g_{\infty,t})_{t\in(-\infty,0]},J_{\infty})$
is biholomorphic and isometric to the round shrinking $\mathbb{P}^1\times\mathbb{C}$. 

Let $\phi_{i}:X_{\infty}\supseteq W_{i}\to X$ be the Cheeger-Gromov
open embeddings, so that $\phi_{i}^{\ast}{g}_{i} \to g_{\infty}$, $\phi_{i}^{\ast}{J}\to J_{\infty}$ in $C_{\text{loc}}^{\infty}(X_{\infty}\times(-\infty,0])$. In particular, for any choice of $\epsilon>0$ and any compact subset $L\subseteq X_{\infty}$, we have that
\[
||r_{i}^{-2}\phi_{i}^{\ast}g_{t_i}-g_{\infty,0}||_{C^{\lfloor\epsilon^{-1}\rfloor}(L,g_{\infty,0})}\leq\epsilon,\,\,\,\,\,\,\,\,\,||\phi_{i}^{\ast} J-J_{\infty}||_{C^{\lfloor\epsilon^{-1}\rfloor}(L,g_{\infty,0})}\leq\epsilon
\]
for sufficiently large $i=i(\epsilon,L) \in\mathbb{N}$. Choose $L$ to be the closure of any neighborhood of the $J_{\infty}$-holomorphic curve $\mathbb{P}^1 \times \{0\}$. 
{ For $i\in \mathbb{N}$ sufficiently large, \cite[Corollary 2.3]{CCD} yields a $\phi_i^{\ast}J$-holomorphic curve $\Sigma_i \subseteq L$ with trivial self-intersection that is homologous to $\mathbb{P}^1 \times \{0\}$. Thus, $\phi_i(\Sigma_i) \subseteq X$ is a $J$-holomorphic curve with trivial self-intersection with
$$r_i^{-2} \text{vol}_{g_{t_i}}(\phi_i(\Sigma_i)) = r_i^{-2} \int_{\phi_i(\Sigma_i)} \omega_{t_i} = r_i^{-2} \int_{\Sigma_i} \phi_i^{\ast} \omega_{t_i} = \int_{\mathbb{P}^1\times\{0\}} r_i^{-2}\phi_i^{\ast} \omega_{t_i}.$$
Because $r_i^{-2}\phi_i^{\ast} \omega_{t_i}\to \omega_{\infty,\,0}$ in $C_{\text{loc}}^{\infty}(X_{\infty})$, we deduce that $\lim_{i\to \infty} \text{vol}_{g_{t_i}}(\phi_i(\Sigma_i)) = 0$. Recalling that $\lim_{i\to \infty} t_i =-1$, this implies that $\lim_{i\to \infty} \text{vol}_{g_{-1}}(\phi_i(\Sigma_i))=0$ so that 
$$\text{vol}_{g_{-\frac{1}{2}}}(\phi_i(\Sigma_i)) \leq \text{vol}_{g_{-1}}(\phi_i(\Sigma_i)) - \pi <0$$
for all $i\in \mathbb{N}$
sufficiently large. This is a contradiction.
}
\end{description}
\end{proof}

We now indicate how to combine the proof of Theorem \ref{bddcurve} and \cite[Theorem 1.1]{JunshengSong} to obtain a new proof of the following result, which was previously established in \cite{LiWang}. 

\begin{theorem}\label{2d}
Every complete two-dimensional shrinking gradient K\"ahler-Ricci soliton has bounded curvature.  
\end{theorem}

\begin{proof}
By the proof of Theorem \ref{bddcurve}, it suffices to show that $$\lim_{x\to\infty}\frac{|{\Rm_g}|_{g}(x)}{f(x)}<\infty.$$ 
Suppose, for sake of a contradiction, that this is not the case. Then, by the proof of Case 1 of Theorem \ref{bddcurve}, there exists a sequence $(\Omega_{i})_{i\,\in\,\mathbb{N}}$ of pairwise disjoint domains, each diffeomorphic to a fixed non-flat Calabi-Yau ALE manifold $Z$. 

By \cite[Theorem 1.1]{JunshengSong}, $X$ is quasiprojective and so is biholomorphic to an open subset of the regular set $\overline{X}_{\operatorname{reg}}$ of a projective variety $\overline{X}$. By choosing the irreducible component of $\overline{X}$ containing $X$, we may assume that $\overline{X}$ has complex dimension two. Taking a resolution of singularities
$\rho:\widehat{X}\to\overline{X}$, we then obtain a smooth compact projective surface $\widehat{X}$.


Since $X\subseteq\overline{X}_{\text{reg}}$, it follows
that $X$ is diffeomorphic to an open subset of $\widehat{X}$,
a smooth, compact, four-dimensional manifold. The sets $\Omega_{i}$ therefore form a sequence of pairwise disjoint open subsets of $\widehat{X}$, each diffeomorphic to $Z$. This contradicts \cite[Theorem 12]{CFSZ} (see also \cite{CFSZcorrection}). 
\end{proof}

\section{Classification of Tangent Flows}

Throughout this section, we assume that $(X,g,J,f)$ is an orbifold shrinking gradient K\"ahler-Ricci soliton arising as a limit as in \eqref{Fconverge}, and that $\mathcal{V}\subseteq \widetilde{M}$ is a neighborhood of $\widetilde{E}$ as in Section 3. We also recall the existence of Cheeger-Gromov diffeomorphisms $\psi_i$ realizing smooth local convergence on $X_{\textnormal{reg}}$ so that \eqref{smoothconverge} holds. Any complex orbifold with isolated singularities may be viewed as a normal complex analytic space with quotient singularities. Recall from Lemma \ref{Hausdorff} that $E \cap X_{\textnormal{reg}} = E^{\circ}$ is an analytic subvariety of $X_{\textnormal{reg}}$. Because $X \setminus X_{\textnormal{reg}}$ comprises finitely many points, we can appeal to the extension theorem for analytic sets \cite[Theorem 2, Section 9.4]{sheaves} to conclude that $E$ is a compact analytic subset of $X$. Also recall from Proposition \ref{dichotomy} that all orbifold points of $X$ are contained in $E$ unless $E=\emptyset$, in which case there is only one orbifold point. 

We show in the next proposition that $E$ is in fact the maximal compact analytic subset of $X$. Our argument uses the function $\widetilde{u}_{t}$ defined in \eqref{defineu}. Using the estimates of Proposition \ref{functionu}, we will show that, after pulling back by the diffeomorphisms $\psi_i$, appropriate normalizations of $\widetilde{u}_t$ locally smoothly converge on $X\setminus E$ to a smooth function $u$. Because $\widetilde{u}_{t}$ is approximately a Ricci potential away from $\widetilde{E}$ and $X$ is a shrinking soliton, $u+f$ will be strictly plurisubharmonic on $X\setminus E$. We will also show that $u+f$ is bounded above near $E$ and extends to a plurisubharmonic function on all of $X$. A straightforward application of the maximum principle then tells us that $E$ is the maximal compact analytic subset of $X$. 

\begin{prop} \label{ExcSet}
Any compact holomorphic curve in $X$ is contained in $E$.
\end{prop}

\begin{proof}
We begin with the following claim stating that points in $X_{\textnormal{reg}}$ correspond to points in the Ricci flow on $\widetilde{M}$ that are not too close to $\widetilde{E}$ in a $P^{\ast}$-parabolic sense.

\begin{claim} \label{claim-nointersection}
For any $x\in X_{\textnormal{reg}}\setminus E$,
there exists $r=r(x)>0$ such that $P_{\widetilde{g}_{i}}^{\ast-}(\psi_{i,-1}(x),-1;r)\cap(\widetilde{E}\times\mathbb{R})=\emptyset$
for sufficiently large $i\in\mathbb{N}$.
\end{claim}

\begin{proof}[Proof of Claim \ref{claim-nointersection}]
Let $A=A(x)>0$ to be determined later. We can find $r>0$ such that $P^{\ast-}(x,-1;2A^2r)\subseteq X_{\textnormal{reg}}$. Let $U\Subset X_{\text{reg}}$ be arbitrary. Recall from Theorem \ref{bamconvergence} that $(\psi_i^{-1})_{\ast} \frac{\partial}{\partial t} \to \partial_{\mathfrak{t}}$ in $C_{\text{loc}}^{\infty}(X_{\text{reg}}\times (-\infty,0))$, where $\partial_{\mathfrak{t}}$ is the standard vector field on the second factor of $X \times (-\infty,0)$. Moreover, $E_{i,-1}:= \psi_{i,-1}^{-1}(\widetilde{E}\cap \psi_{i,-1}(U)) \to E\cap U$ in the Hausdorff sense by Lemma \ref{Hausdorff}, so that for all $t\in [-2,-1]$, the time $(t+1)$-flow $\Theta_{i,t}: E_{i,-1}\to X_{\text{reg}} \times \{t\}$ of $(\psi_{i}^{-1})_{\ast}\frac{\partial}{\partial t}$ is well-defined for $i\geq \underline{i}(U)$ sufficiently large. Consequently, the images $E_{i,t}$ of $\Theta_{i,t}$ converge to $(E \cap U)\times \{t\}$ as $i\to \infty$, uniformly in $t\in [-2,-1]$. However, because $\frac{\partial}{\partial t}$ is tangent to $\widetilde{E} \times \mathbb{R}$, we know that the image of $\Theta_{i,t}$ coincides with $\psi_{i,t}^{-1}(\widetilde{E} \cap \psi_{i,t}(U))\times \{t\}$, and so we find that $\psi_{i,t}^{-1}( \widetilde{E} \cap \psi_{i,t}(U))$ converges in the Hausdorff sense (uniformly in $t\in [-2,-1]$) to $E\cap U$ as $i\to \infty$. In particular, if $x\in X_{\textnormal{reg}}\setminus E$, we can shrink $r>0$ in order to assume that
\[
P^{\ast-}(x,-1;A^2r)\cap\bigcup_{t\in[-2,-1]}\psi_{i,t}^{-1}(V_{i}\cap \widetilde{E})=\emptyset
\]
for sufficiently large $i\in\mathbb{N}$. By further shrinking $r>0$, we can assume in addition that $r_{\Rm}^{g}(x,-1)\geq 2A^2 r$, so that $r_{\Rm}^{\widetilde{g}_i}(\psi_{i,-1}(x),-1) \geq A^2 r$ for sufficiently large $i \in \mathbb{N}$ \cite[Lemma 15.16 \& Corollary 15.47]{Bam3}. 
By \cite[Corollary 9.6]{Bam1} and \cite[Proposition 9.17]{Bam2}, it follows that
\begin{align*}
P_{\widetilde{g}_{i}}^{\ast-}(\psi_{i,-1}(x),-1;r) &\subseteq B_{\widetilde{g}_i}(\psi_{i,-1}(x),-1,Ar)\times [-1-r^2,-1] \\ &\subseteq \psi_i(B_g(x,-1,2Ar)\times [-1-r^2,-1]) \\ &\subseteq \psi_i\left(P^{\ast-}(x,-1;A^2r)\right)
\end{align*} 
if we choose $A>0$ large and $i \geq \underline{i}(A)$ large. This observation yields the desired claim.
\end{proof}

Next suppose that $L\subseteq X_{\textnormal{reg}}$ is any compact and connected set and $B>0$. For any $t\in [-B,-B^{-1}]$ and $x\in\psi_{i,t}(L)$, Claim \ref{claim-nointersection} gives us that 
\[
P_{\widetilde{g}_{i}}^{\ast-}(x,t;r)\cap(\widetilde{E}\times\mathbb{R})=\emptyset
\]
for some $r=r(L)>0$, whenever $i\geq \underline{i}(L)$ is sufficiently large. Recall the function $\widetilde{u}_{t}$ defined by \eqref{defineu}, and define $\widetilde{u}_{i,t}:= \widetilde{u}_{T+(T-t_i)t}+\lambda_i$, where $\lambda_i := - \widetilde{u}_{t_i}(\psi_{i,-1}(x_0))$ for some fixed $x_0 \in L$. By Proposition \ref{functionu}(iv) and the fact that 
$$P_{\widetilde{g}}^{\ast-}(x,T+(T-t_i)t;B^{-\frac{1}{2}}r\sqrt{(T-t_i)|t|}) \cap (\widetilde{E}\times \mathbb{R}) = \emptyset$$
for any $t\in [-B,-1]$ and $x\in \psi_{i,(T-t_i)^{-1}(t-T)}(L)$, we find that
$$\sup_{t\in [-B,-1]} \sup_{y\in \psi_{i,t}(L)} |\nabla \widetilde{u}_{i,t}|_{\widetilde{g}_i,t}(y) = \sqrt{T-t_i} \sup_{t \in [T-(T-t_i)B,t_i]} \sup_{y\in \psi_{i,(T-t_i)^{-1}(t-T)}(L)} |\nabla \widetilde{u}_t|_{\widetilde{g}}(y) \leq C(L,B).$$
Similarly, we have that
$$\sup_{t\in [-B,-1]} |\partial_t \widetilde{u}_{i,t}|(\psi_{i,t}(x_0)) = \sup_{t\in [-B,-1]} |R_{\widetilde{g}_{T+(T-t_i)t}}|(\psi_{i,t}(x_0)) = \sup_{t\in [-B,-1]} |R_{\widetilde{g}_{i,t}}|(\psi_{i,t}(x_0))\leq C(B).$$
In summary, we have just shown that any compact connected subset $L \subseteq X_{\textnormal{reg}} \setminus E$ and any $B>0$, we have 
\begin{equation} \label{upperboundforu}
\sup_{t\in [-B,-1]}\sup_{\psi_{i,t}(L)}|\widetilde{u}_{i}|\leq C(L,B).
\end{equation}

By Lemma \ref{containedinV}, we know that $\psi_i(L\times [-B,-1])\subseteq \mathcal{V}\times \mathbb{R}$ for $i\geq \underline{i}(L,B)$ large. By Proposition \ref{functionu}, it follows that
\[
\sqrt{-1}\partial\bar{\partial}\widetilde{u}_{i,t}={\operatorname{Ric}}_{\widetilde{\omega}_{i,t}},
\]
\[
(\partial_{t}-\Delta)\widetilde{u}_{i}=0,
\]
on $\psi_i(L\times [-B,-1])$ for sufficiently large $i\geq \underline{i}(L,B)$.
Local parabolic regularity theory then allows us to extract a subsequence
such that $\widetilde{u}_{i,-1} \circ \psi_{i,-1}$ converges in $C_{\text{loc}}^{\infty}(X_{\textnormal{reg}} \setminus E)$
to a function $u$ 
satisfying $\sqrt{-1}\partial\bar{\partial}u={\operatorname{Ric}}_{\omega}$
on $X_{\textnormal{reg}}\setminus E$. The function $\varphi:=u+f\in C^{\infty}(X_{\textnormal{reg}}\setminus E)$
then satisfies
\[
\sqrt{-1}\partial\bar{\partial}\varphi={\operatorname{Ric}}_{\omega}+\sqrt{-1}\partial\bar{\partial}f=\omega
\]
on $X_{\textnormal{reg}}\setminus E$. 
The following claim will be used to show that $\varphi$ is locally bounded above away from the singularities of $E$ and $X$. We let $E_{\textnormal{reg}}$ denote the regular set of $E$ (as an analytic set).

\begin{claim} \label{claim-upperboundonPSH}
Every point $x \in E_{\textnormal{reg}} \cap X_{\textnormal{reg}}$ has a neighborhood $U$ such that 
$$\limsup_{i \to \infty} \sup_U \psi_{i,-1}^{\ast}(\widetilde{f}_{i,-1}+\widetilde{u}_{i,-1})<\infty.$$
\end{claim}

\begin{proof}[Proof of Claim \ref{claim-upperboundonPSH}]
Fix $x \in E_{\textnormal{reg}} \cap X_{\textnormal{reg}}$ so that there exist holomorphic slice coordinates $(z^1,z^2)$ for $E_{\textnormal{reg}}$ in some neighborhood $U'$ of $x$. In other words,
$E \cap U' = \{ z\in U'\,|\,z^2=0 \}$. By Lemma \ref{coordlemma}, there exist $\psi_{i,-1}^{\ast}\widetilde{J}$-holomorphic coordinates $(z_i^1,z_i^2)$ converging locally smoothly to $(z^1,z^2)$ in a neighborhood of $x$. For $w\in \mathbb{C}$ with $|w|$ sufficiently small, and $\epsilon>0$ small, we may define
$$S_{i,w} := \{ z\in U'\,|\, z_i^1=w,\quad |z_i^2|=\epsilon \} \subseteq U'.$$
Recalling that $E_i = \psi_{i,-1}^{-1}(\widetilde{E} \cap V_{i,-1})$ satisfies $E_i\cap U' \to E \cap U'$ in the Hausdorff sense (cf.~\eqref{eq: Edef}), we see that
$$E_i \cap \{z \in U'\,|\,  |z_i^1|<\epsilon \text{ and } |z_i^2|=\epsilon \} = \emptyset$$
for sufficiently large $i \geq \underline{i}(\epsilon)$. Furthermore, $\psi_{i,-1}^{\ast}(\widetilde{u}_{i,-1}+\widetilde{f}_{i,-1})$ is $\psi_{i,-1}^{\ast}\widetilde{J}$-plurisubharmonic on $U'$ for large $i \geq \underline{i}(U')$. This is because, by Proposition \ref{functionu},
$$\sqrt{-1}\partial_{\psi_{i,-1}^{\ast}\widetilde{J}} \bar{\partial}_{\psi_{i,-1}^{\ast}\widetilde{J}} \left( \psi_{i,-1}^{\ast}(\widetilde{u}_{i,-1}+\widetilde{f}_{i,-1}) \right) = \psi_{i,-1}^{\ast} (\Ric_{\widetilde{\omega}_{i}}+\sqrt{-1}\partial_{\widetilde{J}} \bar{\partial}_{\widetilde{J}} \widetilde{f}_{i,-1})+ [\psi_{i,-1}^{-1}(\widetilde{E})]$$
in the sense of currents, where $\psi_{i,-1}^{\ast} (\Ric_{\widetilde{\omega}_{i}}+\sqrt{-1}\partial_{\widetilde{J}} \bar{\partial}_{\widetilde{J}} \widetilde{f}_{i,-1}) \to \omega$ and $[\psi_{i,-1}^{-1}(\widetilde{E})]$ is a positive $(1,\,1)$-current with respect to the complex structure $\psi_{i,-1}^{\ast}J$. Therefore upon setting
$$D_{i,w}:= \{ z\in U'\,|\,z_i^1=w,\quad|z_i^2|<\epsilon \},$$
the maximum principle tells us that
$$\sup_{D_{i,w}} \psi_{i,-1}^{\ast}(\widetilde{u}_{i,-1}+\widetilde{f}_{i,-1}) \leq \sup_{S_{i,w}} \psi_{i,-1}^{\ast}(\widetilde{u}_{i,-1}+\widetilde{f}_{i,-1}).$$
By taking the supremum over $|w| < \epsilon$, we thus conclude that for $i\in \mathbb{N}$ sufficiently large, 
\begin{align*} \sup_{|z^1|<\frac{\epsilon}{2}} \sup_{|z^2|<\frac{\epsilon}{2}} \psi_{i,-1}^{\ast} (\widetilde{u}_{i,-1}+\widetilde{f}_{i,-1}) &\leq \sup_{|z_i^1|<\epsilon} \sup_{|z_i^2|<\epsilon} \psi_{i,-1}^{\ast} (\widetilde{u}_{i,-1}+\widetilde{f}_{i,-1}) \\ & \leq \sup_{|w|<\epsilon} \sup_{S_{i,w}} \psi_{i,-1}^{\ast}(\widetilde{u}_{i,-1}+\widetilde{f}_{i,-1}).
\end{align*}
However, the right-hand side is bounded uniformly in $i\in \mathbb{N}$ by \eqref{upperboundforu} and the fact that 
$$\bigcup_{|w|<\epsilon} S_{i,w}\subseteq \left\{ z\in U'\,|\, |z^1|<2\epsilon \text{ and } |z^2|\geq \tfrac{\epsilon}{2}\right\} \Subset  X_{\textnormal{reg}}\setminus E$$
for sufficiently large $i\geq \underline{i}(\epsilon)$. 
\end{proof}

Claim \ref{claim-upperboundonPSH} implies that $\varphi$ is a plurisubharmonic function on $X_{\text{reg}}\setminus E$ which is bounded above in a neighborhood of any point of $X_{\text{reg}}\cap E_{\text{reg}}$. By \cite[Théorème 1]{GraRem}, $\varphi$ uniquely extends to a plurisubharmonic function on $X_{\textnormal{reg}} \cap E_{\textnormal{reg}}$. Because $X \setminus (X_{\textnormal{reg}} \cap E_{\textnormal{reg}})$ has complex codimension 2, $\varphi$ further extends to a plurisubharmonic function on all of $X$ by \cite[Théorème 2]{GraRem}. Suppose by way of contradiction that $A\subseteq X$ is a compact holomorphic curve not entirely contained in $E$. Without loss of generality, we can assume that $A$ is irreducible. Applying the maximum principle for weakly pseudoconvex functions (cf.~\cite[Section 6.3, Theorem]{Chirka}) then allows us to conclude that $\varphi|_A$ is constant, contradicting the fact that at any regular point of $A \setminus E$, we have
$$\sqrt{-1}\partial \bar{\partial}(\varphi|_A) = (\sqrt{-1}\partial \bar{\partial} \varphi)|_A >0.$$
\end{proof}

We now show that $X$ has quadratic curvature decay and use this to determine that $X$ is asymptotic at infinity to a K\"ahler cone which coincides with the Remmert reduction of $X$ (cf.~\cite[p.221]{sheaves}). This observation, together with Proposition \ref{dichotomy}, allow us to rule out flat cones as tangent flows and to show that $E \neq \emptyset$. 

\begin{prop} \label{nocones} 
\begin{enumerate}
    \item $(X,g)$ is not flat.
    \item There is a birational morphism $\rho: X \to \mathbb{C}^2$ with exceptional set $\rho^{-1}(0)=E\neq \emptyset$. In particular, $X$ is quasiprojective. 
\end{enumerate}
\end{prop}

\begin{proof}
\begin{enumerate}
\item Suppose by way of contradiction that $X$ is flat. Then $X=\mathbb{C}^{2}/\Gamma$ for $\Gamma$ a non-trivial finite subgroup of $U(2)$ acting freely on $\mathbb{C}^{2}\setminus\{0\}$ \cite[Theorem 2.46]{Bam3}. Let $o\in X$ denote the vertex. Because $X$ does not contain any compact holomorphic curves, we see that $E=\emptyset$ so that $\psi_{i,-1}(\partial B(o,1))\cap \widetilde{E}=\emptyset$ for $i\in \mathbb{N}$ sufficiently large. By Lemma \ref{containedinV}, we have that $\Sigma_i := \psi_{i,-1}(\partial B(o,1))\subseteq \mathcal{V}\setminus \widetilde{E}$. Also note that $\mathcal{V} \setminus \widetilde{E}$ is diffeomorphic to $\mathbb{R}^4 \setminus \{ 0\}$. By Lemma \ref{topologicalrestrictions}(i), we therefore find that $\Gamma \subseteq SU(2)$. From Lemma \ref{JordanBrouwer}, we know that $\widetilde{M}\setminus \Sigma_i$ and $\mathcal{V} \setminus \Sigma_i$ each have exactly two connected components. One of these connected components is clearly in common. Denote this particular component by $\mathcal{C}_i$, and equip $\Sigma_i$ with the contact structure induced from $\psi_{i,-1}$ and the standard contact structure of $\mathbb{S}^3/\Gamma \cong \partial B(o,1)$. 

First assume that $\widetilde{E}\subseteq \mathcal{C}_i$. Let $\widetilde{\mathcal{C}}_i$ denote the blowdown of $(\mathcal{C}_i,\widetilde{\omega}_{i,-1})$ along $\widetilde{E}$, which admits a symplectic form $\omega'$ such that the blowdown map $(\mathcal{C}_i,\widetilde{\omega}_{i,-1})\to(\widetilde{\mathcal{C}}_i,\omega')$ is a symplectomorphism near $\partial \mathcal{C}_i$ (cf.~\cite[Lemma 3.2]{mcduffminusone} and its proof). As $\omega_{-1}$ is positive on the distribution corresponding to the contact structure of $\partial B(o,1)$, and since $\psi_{i,-1}^{\ast}\widetilde{\omega}_{i,-1} \to \omega_{-1}$ in $C_{\textnormal{loc}}^{\infty}(X\setminus \{o\})$, we see that $\widetilde{\omega}_{i,-1}$ restricts to a positive symplectic form on the contact structure that we are considering on $\Sigma_i$. $(\widetilde{\mathcal{C}}_i,\omega')$ therefore defines a minimal symplectic filling of $\Sigma_i$. By \cite[Main Theorem]{OhtaOno}, $\widetilde{\mathcal{C}}_i$ is diffeomorphic to the minimal resolution of $\mathbb{C}^2/\Gamma$, hence contains a two-sphere with self-intersection $-2$. This then yields such a surface in the blowdown of $\mathcal{V}$ along $\widetilde{E}$, which is diffeomorphic to $\mathbb{R}^4$. This is a clear contradiction.

If instead $\widetilde{E} \cap \mathcal{C}_i = \emptyset$, then arguing as above, we deduce that $(\mathcal{C}_i,\widetilde{\omega}_{i,-1})$ is diffeomorphic to a minimal symplectic filling of $\mathbb{S}^3/\Gamma$ with its standard contact structure. This leads to a contradiction as above.

\item Recall from Proposition \ref{bddcurve} that $X$ has bounded curvature. Fix $x_0 \in X$. We now show that the curvature goes to zero at spatial infinity.

\begin{claim} \label{claim:curvaturetozero} $\lim_{r \to \infty} \sup_{X\setminus B_g(x_0,r)} |{{\Rm_g}}|_g = 0$.
\end{claim}

\begin{proof} Suppose by way of contradiction that there is a sequence $(x_\alpha)_{\alpha \in \mathbb{N}}$ in $X$ diverging to spatial infinity such that 
$$\liminf_{\alpha \to \infty} |{{\Rm_g}}|_g(x_{\alpha}) >0.$$
$(X,(g_t)_{t\in (-\infty,-1]})$ having a uniform curvature bound and $X$ containing only finitely many orbifold singularities allows us pass to a subsequence so that $(X,(g_t)_{t\in (-\infty,-1]},x_{\alpha})$ converges in the smooth pointed Cheeger-Gromov sense to a complete non-flat ancient solution \linebreak $(X_{\infty},(g_{\infty,t})_{t\in (-\infty,-1]})$ of the K\"ahler-Ricci flow with bounded curvature. Furthermore, applying \cite[Lemma 4.1]{nabersoliton} to the $(-1)$-time slice with soliton potential function $f$, using that fact that $|\nabla f|^2\geq f-R$ (cf.~Theorem \ref{bamconvergence}(vi)), we find a metric splitting of $g_{\infty,-1}$ so that by \cite[Theorem 1.1]{CK20}, $g_{\infty,\,t}$ splits for all times $t\leq-1$. In other words, the flow on $X_{\infty}$ splits as a product of a $\kappa$-solution and $\mathbb{R}$. Arguing as in Case 3 in the proof of Theorem \ref{bddcurve}, we conclude that $X_{\infty}$ is biholomorphic to the complex cylinder $(\mathbb{P}^1 \times \mathbb{C},g_{\text{cyl}},J_{\text{cyl}})$.

Let $\zeta_{\alpha}$ be the Cheeger-Gromov diffeomorphisms mapping from an exhaustion of $\mathbb{P}^1\times \mathbb{C}$ into neighborhoods of $x_{\alpha}$ in $X$. By a diagonal argument, for $\alpha \in \mathbb{N}$ large, we can find $i(\alpha)\in \mathbb{N}$ such that
$$\left( (T-t_{i(\alpha)})^{-1} (\psi_{i(\alpha),-1} \circ \zeta_{\alpha})^{\ast} \widetilde{g}_{i(\alpha),-1}, (\psi_{i(\alpha),-1}\circ \zeta_{\alpha})^{\ast} \widetilde{J} \right) \to  (g_{\text{cyl}},J_{\text{cyl}}).$$
By \cite[Proof of Theorem B]{CCD}, there exists a compact subset $L\subseteq \mathbb{P}^1 \times \mathbb{C}$ such that for $\alpha\in \mathbb{N}$ sufficiently large, there is a smooth, rational $(\psi_{i(\alpha),-1}\circ \zeta_{\alpha})^{\ast} \widetilde{J}$-holomorphic curve $C_{\alpha} \subseteq L$ in $\mathbb{P}^1\times \mathbb{C}$ with trivial self-intersection. After possibly increasing $i(\alpha)$ for each $\alpha\in \mathbb{N}$, we can also guarantee that $(\psi_{i(\alpha),-1} \circ \zeta_{\alpha})(C_{\alpha})\subseteq \mathcal{V}$ by Lemma \ref{containedinV}.  However, $\mathcal{V}$ does not contain any holomorphic curve with trivial self-intersection, and we derive a contradiction. 
\end{proof}

We next improve the curvature decay in Claim \ref{claim:curvaturetozero} to quadratic curvature decay.

\begin{claim} \label{claim:fastcurvaturedecay} 
There exists $C>0$ such that  $|{\Rm_g}|_g(x)f(x) \leq C$ for all $x\in X$.
\end{claim}

\begin{proof} Let $(\varphi_{t})_{t\in(-\infty,0)}$ be the biholomorphisms from Remark \ref{holoflow} which satisfy $\partial_{t}\varphi_{t}(x)=\frac{1}{|t|}\nabla f(\varphi_{t}(x))$,
so that $g_{t}=|t|\varphi_{t}^{\ast}g$ is the canonical Ricci flow satisfying
$g_{-1}=g$. Recalling the identity
\[
R+|\nabla f|^{2}=f-W
\]
from Theorem \ref{bamconvergence}(vi), we compute that
\[
\frac{d}{dt}f(\varphi_{t}(x))=\frac{1}{|t|}|\nabla f|_{g}^{2}(\varphi_{t}(x))=\frac{1}{|t|}\left(f(\varphi_{t}(x))-W-R_{g}(\varphi_{t}(x))\right).
\]
Thus, there exists some large $r_{0}>0$ such that for all $x\in X$
and $t\in(-\infty,0)$ with $f(\varphi_{t}(x))\geq r_{0}$, we have 
\[
\frac{d}{dt}f(\varphi_{t}(x))\geq\frac{1}{2|t|}f(\varphi_{t}(x)).
\]
After possibly increasing $r_0$, we can ensure that $\{\nabla f=0\}\subseteq\{f<r_0\}$ and $\{f\geq r_0\} \subseteq X_{\textnormal{reg}}$. Set $\Sigma:=f^{-1}(r_{0})$ and define $\eta:\Sigma\times[-1,0)\to X$,
$(x,t)\mapsto\varphi_{t}(x)$. Since $\frac{d}{dt}f(\varphi_{t}(x))\geq0$
for all $(x,t)\in X\times (-\infty,0)$, $\eta$ maps into $\{f\geq r_{0}\}$.
Conversely, for any $x\in\{f\geq r_{0}\}$ and for any $t\in(-\infty,-1)$
such that $f(\varphi_{t}(x))\geq r_{0}$, we can integrate the lower bound for $\frac{d}{dt}f(\varphi_t(x))$ from $t$ to $-1$ to obtain the bound
\[
\frac{f(x)}{f(\varphi_{t}(x))}\geq\sqrt{|t|}
\]
so that $f(\varphi_{t}(x))\leq\frac{f(x)}{\sqrt{|t|}}$. We can therefore
find $t\in(-\infty,-1)$ such that $f(\varphi_{t}(x))=r_{0}$, hence
$\varphi_{t}(x)\in\Sigma$ and $x=\eta(\varphi_{t}(x),\frac{1}{t})$. This
means that $\eta:\Sigma\times[-1,0)\to \{ f \geq r_0\}$ is surjective. 

Fix $\Lambda>0$ large to be determined. By Claim \ref{claim:curvaturetozero}, after possibly increasing $r_{0}$ if necessary, for all $x\in \{f\geq r_{0}\}$ we have that
\[
\sup_{B(x,\Lambda)}|{\Rm_g}|_{g}\leq\Lambda^{-2}.
\]
A variant of Perelman's pseudolocality theorem \cite[Theorem 1.2]{PengLu} then gives $\epsilon>0$ only depending on the flow $(\widetilde{M},(\widetilde{g}_t)_{t\in [0,T)})$ such that 
\[
|{\Rm_g}|_g(x,t)\leq\frac{1}{(\epsilon \Lambda)^{2}}
\]
for all $x\in \{ f \geq r_0\}$ and $t\in[-1,\min\{0,-1+(\epsilon \Lambda)^{2}\})$. This theorem applies to our situation because the flow $(X,(g_t)_{t\in (-\infty ,0)})$ is a singularity model of a compact Ricci flow with smooth Cheeger-Gromov convergence away from a compact subset of $X$. If we
choose $\Lambda > 2\epsilon^{-1}$ sufficiently large, we obtain a bound of the form $|{\Rm_g}|_g(x,t)\leq1$ for all $(x,\,t)\in\linebreak \{f\geq r_{0}\}\times [-1,0)$.
Recall that $|{\Rm_g}|_g(x,t)=\frac{1}{|t|}|{\Rm_g}|_g(\varphi_{t}(x))$.
Integrating the ODI\linebreak
$\frac{d}{dt}\log(f(\varphi_{t}(x))-W)\leq\frac{d}{dt}\log\frac{1}{|t|}$
also gives us that
\[
f(\varphi_{t}(x))-W\leq\frac{1}{|t|}(f(x)-W).
\]
Combining these expressions yields
\begin{align*}
|{\Rm_g}|_g(\varphi_{t}(x))(f(\varphi_{t}(x))-W)&=|t|\cdot|{\Rm_g}|_g(x,t)\cdot(f(\varphi_{t}(x))-W)\\ &\leq(f(x)-W)|{\Rm_g}|_g(x,t)\leq f(x)-W=r_{0}-W
\end{align*}
for all $(x,\,t)\in \Sigma \times[-1,0)$. In light of the surjectivity proved above, the claim now follows.
\end{proof}

Next, the proof of \cite[Lemma 3.1]{CDS} is point-wise and so works in uniformizing charts of the orbifold. This results in a continuous non-negative function 
$$q := \lim_{t\nearrow 0} |t|(f_t(x)-W):X\to\mathbb{R}_{\geq0}.$$

It therefore follows from \cite[Lemma 3.4]{CDS} that $q$ is proper and bounded from below. For $a>0$, consider the set $$X_{a}:=\{x\in X\,|\,q(x)\geq a\}.$$ This set is preserved by the flow of $\nabla f$. Choose $a$ large enough so that $X_{a}$ doesn't contain any orbifold singularities of $X$. Then \cite[Lemma 3.2]{CDS} tells us that $q$ is smooth on $X_{a}$, and \cite[Lemma 3.6]{CDS} then implies that $X$ has only finitely many ends. 

Along each (of the finite) end of $X$, \cite[Proposition 3.10]{CDS} applies and tells us that there exists a biholomorphism $\rho:X\setminus K\to C_{0}\setminus K'$, where $K\subset X$ is a compact subset of $X$ containing the orbifold singularities, $C_{0}$ is a smooth K\"ahler cone smooth outside a compact subset, and $K'\subset C_{0}$ is a compact subset containing  the apex of $C_{0}$. The fact that $C_{0}$ has only an isolated singularity at the apex is because the 
cone is smooth at infinity, hence has a smooth link. This biholomorphism pushes forward the soliton vector field to the radial vector field $r\partial_{r}$ of the cone, where $r$ is the radial function of the cone, and identifies the soliton metric with the cone metric up to terms of order $O(r^{-2})$; see \cite[Theorem 3.8]{CDS} for details.

Let $\pi:\widetilde{X}\to X$ denote the minimal resolution of $X$. As a result of what has just been said, the radial functions of the cones along each end can be extended to $\widetilde{X}$ in such a way that $\widetilde{X}$ is $1$-convex; see \cite[Lemma 2.15]{ACCY1} for a proof. This implies that $\widetilde{X}$ has a Remmert reduction $\widetilde{X}'$. Because $\widetilde{X}'$ is biholomorphic to a K\"ahler cone along one of its ends at infinity, and since $\widetilde{X}'$ has only one end \cite[p.454]{rossi}, the proof of \cite[Claim 3.12]{CDS} implies that $C_{0}$ is biholomorphic to $\widetilde{X}'$.

Since the soliton vector field $Y := \nabla^{g} f$ vanishes along the orbifold singularities of $X$, it lifts to a real holomorphic vector field $\widetilde{Y}$ on $\widetilde{X}$. The proof of \cite[Claim 3.13]{CDS} now applies with $M$ replaced with $\widetilde{X}$ and demonstrates that $\rho \circ \pi:\widetilde{X}\to C_{0}$ is a resolution, equivariant with respect to the flows of $\widetilde{Y}$ and $r\partial_{r}$, and therefore with respect to 
$J\widetilde{Y}$ and $J_{0}r\partial_{r}$, where $J$ and $J_{0}$ are the complex structure on $\widetilde{X}$ and $C_{0}$, respectively. By construction, it follows that $\rho:X\to C_{0}$ is a partial resolution of $C_{0}$ with, by Proposition \ref{ExcSet}, exceptional set $E$.

Now, in general we have that $R_{g}\geq0$ and 
if $R_{g}=0$ at one point, then the soliton is flat
\cite[Proposition 3.7]{li1}  (see also \cite[Theorem 2.46]{Bam3}). By (i), the latter cannot hold, hence $R_{g}>0$ everywhere. Then
\cite[Theorem 6.2]{CDS} holds true for shrinking orbifolds,
and so \cite[Corollary 6.3]{CDS} tells us that the tangent cone at infinity has strictly positive scalar curvature. By \cite[Theorem 8]{belgun} (see also \cite[Theorem 10.1.3]{sasakian}), it must be biholomorphic to $\mathbb{C}^{2}/\Gamma$ for $\Gamma$ a finite subgroup of $U(2)$ acting freely on $\mathbb{C}^{2}\setminus\{0\}$. Moreover, the induced Sasakian structure on $\mathbb{S}^3/\Gamma \hookrightarrow C_0$ is a deformation of the standard Sasaskian structure on $\mathbb{S}^3/\Gamma$  \cite[Theorem 8]{belgun}. By Gray's Stability Theorem \cite[Theorem 2.2.2]{contact}, $\mathbb{S}^3/\Gamma \hookrightarrow C_0$ is contactomorphic to the standard contact structure on $\mathbb{S}^3/\Gamma$. 

We next show that the finite subgroup $\Gamma$ is actually trivial. To this end, recall that $r$ denotes the radial function of the cone $C_0$. By choosing $r_0>0$ such that $E\subseteq \{\rho^{*}r < \frac{1}{2}r_0 \}$ and then appealing to Lemma \ref{containedinV}, we can ensure that the closed 3-fold $\Sigma_i:=\psi_{i,-1}(\{ \rho^{*}r = r_0\}) \cong \mathbb{S}^3/\Gamma$ is contained in $\mathcal{V} \setminus \widetilde{E} \cong \mathbb{R}^4\setminus \{0\}$ for sufficiently large $i\in \mathbb{N}$. By Proposition \ref{topologicalrestrictions}(i), it follows that $\Gamma \subseteq SU(2)$. From Lemma \ref{JordanBrouwer}, we know that $\widetilde{M}\setminus \Sigma_i$ and $\mathcal{V} \setminus \Sigma_i$ each have exactly two connected components. One of these connected components is clearly in common. Denote this particular component by $\mathcal{C}_i$. Because $(\Sigma_i,\widetilde{\omega}_{i,-1}|_{\Sigma_i})$ is arbitrarily $C^0$-close to a scaling of the link of $C_0$,  after possibly increasing $r_0$, we can guarantee as in part (i) that either $\mathcal{C}_i$ or its blowdown along $\widetilde{E}$ (if $\widetilde{E}\subseteq \mathcal{C}_i$) is diffeomorphic to a symplectic filling of $\Sigma_i \cong \mathbb{S}^3/\Gamma$ equipped with the contact structure induced from $C_0$. Since $\Sigma_i$ is contactomorphic to $\mathbb{S}^3/\Gamma$ with its standard contact structure, the remaining arguments in the proof of (i) yield a contradiction.

Finally, we show that $E\neq \emptyset$. If $E = \emptyset$, then $X$ would contain no compact holomorphic curves, so would be biholomorphic to its Remmert reduction, i.e., to $\mathbb{C}^2$. \cite[Theorem E(1)]{CDS} would then tell us that $g$ is flat which contradicts part (i).
\end{enumerate}
\end{proof}

We next apply several classical results in the theory of complex surfaces to better understand $E$. First, recall that because $(X,J)$ is a two-dimensional normal analytic variety, there is a well-defined intersection product on Weil divisors (see \cite[Section 2(b)]{mumford} or \cite[Section 1]{sakai} for its definition and basic properties) which takes values in $\mathbb{Q}$.

\begin{prop} \label{minusone} If $\pi: \widetilde{X}\to X$ is the minimal resolution of $X$, then the proper transform of any irreducible component of $E$ is a $(-1)$-curve. Moreover, $E\cap X_{\textnormal{reg}}$ is smooth.
\end{prop}

\begin{proof} 
Let $\rho:X\to\mathbb{C}^{2}$ denote the Remmert reduction from Proposition \ref{nocones}, so that the irreducible components of $E$ are exactly the irreducible curves $C\subseteq X$ contracted by $\rho$. Let $K_{X}$ be a canonical divisor for $X$. This is a $\mathbb{Q}$-Cartier divisor because $X$ has quotient singularities. Thus, there exists $k\in \mathbb{N}_{>0}$ such that $K_X^k$ is a line bundle. As $\omega$ is a smooth orbifold metric on $X$, we therefore have that 
$$
K_{X}\cdot C = \frac{1}{k}\text{deg}(K_X^{k}|_C) = \frac{1}{k}\int_C c_1(K_X^k|_C
) =-\int_{C}{\operatorname{Ric}}_{\omega}=-\int_{C}\omega<0.$$

By \cite[Theorem 4-6-2]{matsuki}, we can write
$$
K_{\widetilde{X}}=\pi^{\ast}K_{X}+\sum_{i=1}^{m}a_{i}E_{i}$$
for some $a_{i}\in(-\infty,0]\cap \mathbb{Q}$, where $E_{1},...,E_{m}$ are the exceptional
divisors of $\pi$. 
Using this, we compute that
$$
K_{\widetilde{X}}\cdot\widetilde{C}=(\pi^{\ast}K_{X}+\sum_{i=1}^{m}a_{i}E_{i})\cdot\widetilde{C}=K_{X}\cdot C+\sum_{i=1}^{m}a_{i}\underbrace{E_{i}\cdot\widetilde{C}}_{\geq\,0}\leq K_{X}\cdot C<0,$$
where $\widetilde{C}$ is the proper transform of $C$ with respect
to $\pi$. Since $K_{\widetilde{X}}\cdot\widetilde{C}\in\mathbb{Z}$,
we must have in fact that $K_{\widetilde{X}}\cdot\widetilde{C}\leq-1$. The
adjunction formula for $\widetilde{C}\hookrightarrow\widetilde{X}$
(cf.~\cite[Section II.11, (16)]{surfaces}) then gives us that
$$
2g(\widetilde{C})-2=K_{\widetilde{X}}\cdot\widetilde{C}+\widetilde{C}\cdot\widetilde{C}\leq-1+\widetilde{C}\cdot\widetilde{C},$$
where $g(\widetilde{C}) \geq 0$ \cite[Section II.22(c)]{surfaces} is the arithmetic genus of $\widetilde{C}$. Because $\widetilde{C}$ is an exceptional curve of $\pi$, Grauert's criterion \cite[Theorem III.2.1]{surfaces} tells us that $\widetilde{C} \cdot \widetilde{C}<0$. By combining expressions, we therefore conclude that $\widetilde{C}\cdot \widetilde{C} = -1$ and $K_{\tilde{X}}\cdot \widetilde{C}=-1$, so that $\widetilde{C}$ is a 
$(-1)$-curve \cite[Proposition III.2.2]{surfaces}. In particular, $\widetilde{C}$ is a
nonsingular irreducible rational curve. The fact that two $(-1)$-curves can't intersect in the exceptional set of the map  $\rho\circ\pi:\widetilde{X}\to\mathbb{C}^{2}$ by Grauert's criterion \cite[Theorem 2.1]{surfaces} then implies the last assertion of the proposition. 
\end{proof}

The structure of the irreducible components of $E$ takes the following form.

{
\begin{lemma} \label{intersectiongraph}
$E$ is connected, and any two irreducible components of $E$ intersect in at most one point.
\end{lemma}
}

\begin{proof} 

By Proposition \ref{bigcurve}, $E$ is connected. The map $\rho \circ \pi$, being a birational morphism of smooth complex surfaces restricting to a biholomorphism away from $(\rho \circ \pi)^{-1}(0)$, has connected exceptional set and 
factors as a finite composition of blowdowns along $(-1)$-curves  \cite[Theorem 1-8-2]{matsuki}. The corresponding dual graph is therefore a tree of smooth rational curves.  Because the exceptional curves of $X \to \mathbb{C}^2$ are obtained from those of $\widetilde{X} \to \mathbb{C}^2$ by contracting some of these curves, the claim follows.
\end{proof}

We next collect together some facts to summarize the geometry of $X$ close to its orbifold points and to facilitate notation. These will be used in the proof of Proposition \ref{irreducible} which follows. Recall from Proposition \ref{dichotomy} and Proposition \ref{nocones} that all orbifold singularities of $X$ lie in $E$.

\begin{lemma} \label{gooddomain} Fix $p\in X_{\textnormal{reg}}$, let $x_1,...,x_N \in E$ denote the singular points of $X$, and let $\Gamma_1,...,\Gamma_N$ denote the respective finite subgroups of $U(2)$ acting freely on $\mathbb{C}^2\setminus \{0\}$. Let $(\widehat{g}_{\alpha},\widehat{J}_{\alpha})$ denote the standard flat K\"ahler structure on $C_{\alpha}:=(\mathbb{C}^2\setminus \{0\})/\Gamma_{\alpha}$. For each $\alpha \in \{1,...,N\}$, set $\widehat{W}_{\alpha} := B_{\widehat{g}_{\alpha}}(o_{\alpha},7)\setminus \overline{B}_{\widehat{g}_{\alpha}}(o_{\alpha},\frac{1}{7})$, where $o_{\alpha}$ denotes the vertex of $C_{\alpha}$. Then for all $\epsilon>0$, there exists $r_0=r_0(\epsilon)>0$ such that for all $r\in (0,r_0]$ and $\alpha \in \{1,...,N\}$, there exists a diffeomorphism
$$\zeta_{r,\alpha}: \widehat{W}_{\alpha} \to 
\zeta_{r,\,\alpha}(\widehat{W}_{\alpha}):=
W_{r,\,\alpha} \subseteq X_{\textnormal{reg}}$$
such that the following hold:
\begin{enumerate}
\item $B_g(x_{\alpha},6r) \setminus \overline{B}_g(x_{\alpha},\frac{1}{6}r) \subseteq W_{r,\alpha} \subseteq B_g(x_{\alpha},8r) \setminus \overline{B}_{g}(x_{\alpha},\frac{1}{8}r)$;
\item $\|r^{-2}\zeta_{r,\,\alpha}^{\ast}g - \widehat{g}_{\alpha}\|_{C^{\lfloor \epsilon^{-1} \rfloor}(\widehat{W}_{\alpha},\widehat{g}_{\alpha})} + \|\zeta_{r,\alpha}^{\ast}J - \widehat{J}_{\alpha}\|_{C^{\lfloor \epsilon^{-1} \rfloor}(\widehat{W}_{\alpha},\widehat{g}_{\alpha})} < \epsilon$;

\item If $\mathcal{D}_{r,\alpha}$ denotes the pushforward by $\zeta_{r,\alpha}$ of the standard contact distribution of  $\mathbb{S}^3/\Gamma = \partial B_{\widehat{g}_{\alpha}}(o_{\alpha},1)$ to $\zeta_{r,\alpha}(\partial B_{g_{\alpha}}(o_{\alpha},1))=:\Sigma_{r,\alpha}$, then the restriction of $\omega$ to $\mathcal{D}_{r,\alpha}$ is positive.
\end{enumerate}
\end{lemma}

\begin{proof} Because the tangent cone of $X$ at $x_{\alpha}$ is $(\mathbb{C}^2 \setminus \{0\})/\Gamma_{\alpha}$, and convergence to the tangent cone is in the smooth Cheeger-Gromov sense away from $o_{\alpha}$, (i) and (ii) follow. Because the K\"ahler form $\widehat{\omega}_{\alpha}$ of $C_{\alpha}$ is positive on the standard contact distribution of $\mathbb{S}^3/\Gamma$, (iii) follows from (ii). 
\end{proof}

So far, we have shown that $E$ is a tree of rational holomorphic curves whose singularities and intersections all lie within the set of orbifold points and that all singular points of $X$ are contained in $E$. We now show that $E$ is in fact irreducible and that all singularities of $X$ are of a specific type. The proof of both of these statements crucially relies on the fact that $H_2(\mathcal{V},\mathbb{Z})$ is generated by the fundamental class of the $(-1)$-curve $\widetilde{E}$. Indeed, if $E$ had multiple components, then we show that parts of $\widetilde{E}$ can be excised and small caps glued in to obtain two classes in $H_2(\mathcal{V},\mathbb{Z})$ which are positive multiples of $\widetilde{E}$ and intersect positively, thereby obtaining a contradiction. Such a capping procedure is made possible by the relationship between linking number and intersection number given by Proposition \ref{linkingequalsintersection}. Using a similar idea, along with the classification of minimal symplectic fillings of spherical three-manifolds, we also show that any orbifold singularity of $X$ is one of a special class of cyclic singularities. This class consists of precisely the cyclic singularities which admit a smoothing that is a $\mathbb{Q}$-homology ball, sometimes referred to in the literature as Wahl singularities.

\begin{prop} \label{irreducible}
\begin{enumerate}
 \item  $E$ is irreducible.
 \item  Any orbifold singularity of $X$ is a cyclic singularity of type $\frac{1}{n^2}(1,na-1)$ for some $n>a\geq 1$ with $\textnormal{gcd}(n,a)=1$.   
\end{enumerate}
\end{prop}

\begin{proof} 
Set $Q:= \prod_{\alpha = 1}^N |\Gamma_{\alpha}|$. 
\begin{enumerate}
\item Suppose by way of contradiction that $E$ is reducible and comprises $P$ irreducible components $E^{1},\ldots,E^{P}$. Write $E=\sum_{\beta=1}^{P}E^{\beta}$ for the decomposition of $E$
into these components. Recall the notation of Lemma \ref{gooddomain}. $E\setminus (\{x_{1},...,x_{N}\} \cap E) \subset X_{\textnormal{reg}}$
is smooth by Proposition \ref{minusone} and has $P$ connected components, namely $\mathring{E}^{\beta} := E^{\beta}\setminus(\{x_{1},...,x_{N}\}\cap E^{\beta})\subset X_{\textnormal{reg}}$
for $\beta=1,...,P$. We define $\Omega_r := B_g(p,r^{-1})\setminus \bigcup_{\alpha=1}^N \overline{B}_g(x_{\alpha},r)$, where $r>0$ is sufficiently small so that Lemma \ref{gooddomain} applies.

Recall from \eqref{eq: Edef} that $E_i = \psi_{i,-1}^{-1}(V_{i,-1} \cap \widetilde{E})$ Hausdorff converges to $E \cap X_{\textnormal{reg}}$ and by \eqref{eq:limsup} satisfy
$$\limsup_{i\to \infty} \mathcal{H}_{\psi_{i,-1}^{\ast}\widetilde{g}_{i,-1}}^2(E_i \cap \Omega_r) \leq 2\pi$$
for any $r>0$. We claim that the convergence is also in the weak sense of currents. Indeed, any sequence of holomorphic curves in $\mathbb{C}^2$ with a locally uniform upper bound on their area converges in the weak sense of currents to a holomorphic 2-chain \cite[Section 16.1, Proposition 1]{Chirka}. Because this statement is local, we can use Lemma \ref{coordlemma} to show that this result also applies to the sequence $E_i$ of $\psi_{i,-1}^{\ast}\widetilde{J}$ holomorphic curves, since $(\psi_{i,-1}^{\ast}\widetilde{g}_{i,-1},\psi_{i,-1}^{\ast}\widetilde{J}) \to (g,J)$ in $C_{\text{loc}}^{\infty}(X_{\text{reg}})$. The support of any limiting current is the Hausdorff limit $E\cap X_{\text{reg}}$, and the limit current is some positive integer linear combination of the currents of integration corresponding to each of its irreducible components. We therefore have that
\begin{equation} \label{currentconverge} [E_i] \to \sum_{\beta=1}^P a_{\beta}[E^{\beta}]\quad\textnormal{in  $X_{\textnormal{reg}}$}\end{equation}
in the weak sense of currents for some $a_{\beta}\in \mathbb{N}_{>0}$, as claimed. In particular, it follows that there exists $C>0$ such that for any $r>0$ with $\mathcal{A}_{r,\alpha}:= B(x_{\alpha},3r)\setminus \overline{B}(x_{\alpha},\frac{1}{2}r) \subseteq X_{\textnormal{reg}},$ we have that
$$\limsup_{i \to \infty} \mathcal{H}_{\psi_{i,-1}^{\ast}\widetilde{g}_{i,-1}}^2(E_{i} \cap \mathcal{A}_{r,\alpha}) \leq C \mathcal{H}_g^2(E\cap B(x_{\alpha},4 r)).$$
By choosing $r>0$ small enough so that $B(x_{\alpha},r)$ lies in the domain of an isotropy chart centered at $x_{\alpha}$, we can apply the volume-ratio monotonicity for analytic sets \cite[Section 15.1, Proposition 1]{Chirka} to the lift of $E$ by the orbifold chart, yielding
$$\mathcal{H}_g^2(E\cap B(x_{\alpha},4r))\leq Cr^2$$
for all $r>0$ sufficiently small. Combining expressions then gives that
\begin{equation} \label{areasmall} \limsup_{i\to \infty} \mathcal{H}_{\psi_{i,-1}^{\ast} \widetilde{g}_{i,-1}}^2(E_i \cap \mathcal{A}_{r,\alpha}) \leq Cr^2\end{equation}
for all $r>0$ sufficiently small. 

Using this estimate, we apply Lemma \ref{cappinglemma} to prove that for appropriate choices of (small) $r>0$,
$E_i=\psi_{i,-1}^{-1}(\widetilde{E} \cap V_{i,-1})$, that is,
the pullbacks of $\widetilde{E}$ by the Cheeger-Gromov diffeomorphisms $\psi_{i,-1}$ can be cut along the spheres $\Sigma_{r,\alpha} \cong \mathbb{S}^3/\Gamma_{\alpha}$ defined in Lemma \ref{gooddomain} for $i$ large enough to obtain $\psi_{i,-1}^{\ast}\widetilde{J}$-holomorphic curves, some of whose boundaries are positively linked in $\Sigma_{r,\alpha}$. We moreover show that there is a surface of small area in $\Sigma_{r,\alpha}$ which can be used to cap off the components of this boundary (after replacing $\widetilde{E}$ with an appropriate multiple if necessary) to obtain closed 2-chains which will later be shown to intersect positively.

\begin{claim} \label{claim-upperboundonboundarycurve}
There exists $r>0$ such that 
for all $i\geq \underline{i}(r)$ sufficiently large, the following holds.

    There is a partition $E_i \cap \Omega_r = \sqcup_{\beta=1}^{P} E_{i,r}^{\beta}$, where  $E_{i,r}^{\beta}$ itself is a finite union of connected components, such that $\partial E_{i,r}^{\beta} =\sum_{\{\alpha\,|\,x_{\alpha} \in E^{\beta}\}} \sigma_{i,r,\alpha,\beta}$ is a sum of 1-chains $\sigma_{i,r,\alpha,\beta}$, each of which is a finite union of smoothly embedded closed curves in $\Sigma_{r,\alpha}$, satisfying for any distinct $\beta_1,\,\beta_2 \in \{1,...,P\}$ with $x_{\alpha} \in E^{\beta_1} \cap E^{\beta_2}$,  $$\operatorname{link}_{\Sigma_{r,\alpha}}(\sigma_{i,r,\alpha,\beta_1},\sigma_{i,r,\alpha,\beta_2}) \in \mathbb{Q} \cap (0,\infty).$$ 
    Moreover, there exists $C>0$ such that for each $\alpha \in \{1,...,N\}$ and $\beta \in \{1,...,P\}$ satisfying $x_{\alpha}\in E^{\beta}$, there is a smooth 2-chain  $S_{i,r,\alpha,\beta}\subset\Sigma_{r,\alpha}$ such that $\partial S_{i,r,\alpha,\beta} = Q\cdot \sigma_{i,r,\alpha,\beta}$ as 1-chains, where $$\left| \int_{S_{i,r,\alpha,\beta}} \psi_{i,-1}^{\ast} \widetilde{\omega}_{i,-1} \right| \leq Cr^2.$$ 

\end{claim}

\begin{proof}[Proof of Claim \ref{claim-upperboundonboundarycurve}]

Note that the intersection $E\cap \partial B_{g}(x_{\alpha},r)$ is transverse for all sufficiently small $r>0$ (cf.~the first bullet point in \cite[Section 9.2]{Chirka}). Moreover, because  $E_i \cap \mathcal{A}_{r,\alpha}$ can be identified with $\psi_{i,-1}(E_i \cap \mathcal{A}_{r,\alpha})=\widetilde{E} \cap \psi_{i,-1}(\mathcal{A}_{r,\alpha})$, and $\widetilde{E}$ is compact, $E_i \cap \mathcal{A}_{r,\alpha}$ is properly embedded in $\mathcal{A}_{r,\alpha}$. We may therefore apply Lemma \ref{cappinglemma} to $(E_i \cap \mathcal{A}_{r,\alpha},r^{-1}\psi_{i,-1}^{\ast}\widetilde{g}_{i,-1},\psi_i^{\ast}\widetilde{J})$ to conclude that (after replacing $r$ with some $r'\in (r,2r)$) $E_i \cap \Sigma_{r,\alpha}$ is a finite union of smoothly embedded circles, each converging smoothly to $\Sigma_{r,\alpha} \cap E^{\beta}$ for some $\beta \in \{1,...,P\}$ with $x_{\alpha}\in E^{\beta}$. We can thus pass to a subsequence so that the number of components of each $E_i \cap \Sigma_{r,\alpha}$ is constant for $i\geq \underline{i}(r)$ sufficiently large. This implies that the number of connected components of $E_i \cap \Omega_r$ is uniformly bounded, so we may pass to a further subsequence in order to ensure that the number of such components is also constant for $i\geq \underline{i}(r)$ sufficiently large. As $i\to \infty$, each connected component of $E_i \cap \Omega_r$ converges in the Hausdorff sense to a connected subset of $E\cap \Omega_r$. Because the collection $\{ E^{\beta}\cap X_{\textnormal{reg}} | 1\leq \beta \leq P \}$ is pairwise disjoint with union $E \cap X_{\textnormal{reg}}$, it follows that the connected components of $E_i \cap \Omega_r$ each Hausdorff converge to a subset of some $E^{\beta}$. We may therefore partition $E_i \cap \Omega_r$ accordingly, letting $E_{i,r}^{\beta}$ denote the union of the connected components of $E_i \cap \Omega_r$ converging to a subset of $E^{\beta}$. In particular, the boundary components of $E_{i,r}^{\beta}$ converge to circles in $E^{\beta} \cap \Sigma_{r,\alpha}$. Let $\sigma_{i,r,\alpha,\beta}$ denote the union of the connected components of $\partial E_{i,r}^{\beta}$ contained in $\Sigma_{r,\alpha}$, oriented by any outward-pointing normal vector field of $E_{i,r}^{\beta}$ in $E_i$. By Lemma \ref{cappinglemma}(c), the embeddings of $\sigma_{i,r,\alpha,\beta}$ converge smoothly to an $a_{\beta}$-sheeted covering map of the circle $E^{\beta}\cap \Sigma_{r,\alpha}$. In particular, for sufficiently large $i\geq \underline{i}$, the corresponding circles $\sigma_{i,r,\alpha,\beta}$ are homologous to $a_{\beta}[E^{\beta} \cap \Sigma_{r,\alpha}]$ within any arbitrarily small neighborhood of $E^{\beta} \cap \Sigma_{r,\alpha}$ in $\Sigma_{r,\alpha}$. Thus, Remark \ref{remark-homotopy} and Proposition \ref{linkingequalsintersection} yield the fact that
$$\operatorname{link}_{\Sigma_{r,\alpha}}(\sigma_{i,r,\alpha,\beta_1},\sigma_{i,r,\alpha,\beta_2}) = a_{\beta_1} a_{\beta_2} \operatorname{link}_{\Sigma_{r,\alpha}}(E^{\beta_1}\cap \Sigma_{r,\alpha},E^{\beta_2} \cap \Sigma_{r,\alpha}) = a_{\beta_1} a_{\beta_2} E^{\beta_1} \cdot E^{\beta_2},$$
where we observe that both of $\sigma_{i,r,\alpha,\beta_1},\sigma_{i,r,\alpha,\beta_2}$ have been equipped with opposite the orientation from Proposition \ref{linkingequalsintersection}. As $E^{\beta_1},E^{\beta_2}$ have no common irreducible components, we read from \cite[p.17, property (iv)]{mumford} that $E^{\beta_1}\cdot E^{\beta_2}>0$. The existence of $S_{i,r,\alpha,\beta}$ satisfying the desired bound is a consequence of Lemma \ref{cappinglemma}(d). 
\end{proof}

Suppose that $r>0$ is small and satisfies the conclusion of Claim \ref{claim-upperboundonboundarycurve}. If $i\geq \underline{i}(r)$ is sufficiently large, then Claim \ref{claim-upperboundonboundarycurve} tells us that the smooth 2-chain
\[
\widehat{E}_{i,r}^{\beta}:=Q \psi_{i,-1}(E_{i,r}^{\beta}) -\left(\sum_{\{\alpha\,|\,x_{\alpha}\,\in\,E^{\beta}\}} \psi_{i,-1}(S_{i,r,\alpha,\beta})\right)
\]
in $\mathcal{V}$ satisfies $\partial \widehat{E}_{i,r}^{\beta} = 0$. 
If $P\geq 2$, then there exist $\alpha \in \{1,...,N\}$ and $\beta_1,\beta_2 \in \{1,...,P\}$ with $\beta_1 \neq \beta_2$ such that $x_{\alpha} \in E^{\beta_1} \cap E^{\beta_2}$. By Lemma \ref{intersectiongraph}, we then have $E^{\beta_1}\cap E^{\beta_2}=\{x_{\alpha}\}$. Choose $0<r_1 < r_2 < 1$ such that Claim \ref{claim-upperboundonboundarycurve} holds for both. By \cite[Exercise 3.3.14]{GuiPol}, the intersection number of $\psi_{i,-1}(S_{i,r_2,\alpha,\beta_2}) \hookrightarrow \widetilde{M}$ and $\psi_{i,-1}(\widehat{E}_{i,r_1}^{\beta_1})$ in $\widetilde{M}$ is equal to the intersection number of $\psi_{i,-1}(S_{i,r_2,\alpha,\beta_2})$ and $\psi_{i,-1}(E_{i,r_1}^{\beta_1})\cap \psi_{i,-1}(\Sigma_{r_2,\alpha})=\psi_{i,-1}(E_{i,r_1}^{\beta_1}\cap \Sigma_{r_2,\alpha})$ in $\psi_{i,-1}(\Sigma_{r_2,\alpha})$, or equivalently, the intersection number of $S_{i,r_2,\alpha,\beta_2}$ and $E_{i,r_1}^{\beta_1}\cap \Sigma_{r_2,\alpha}$ in $\Sigma_{r_2,\alpha}$, where the orientation of $\Sigma_{r_2,\alpha}$ is the boundary orientation induced from the connected component of $X\setminus \Sigma_{r_2,\alpha}$ containing $x_{\alpha}$, and the orientation of $E_{i,r_1}^{\beta}\cap \Sigma_{r_2,\alpha}$ is induced by the orientations on $E_{i,r_1}^{\beta_1}$ and $\Sigma_{r,\alpha}$ as explained in \cite[p.100]{GuiPol}. This induced orientation of $E_{i,r_1}^{\beta_1}\cap \Sigma_{r_2,\alpha}$ is opposite to that of $\sigma_{i,r_2,\alpha,\beta_1}$. 
Because $E_{i,r_2}^{\beta_2} \cap S_{i,r_1,\alpha,\beta_1}= \emptyset$ and $E_{i,r_1}^{\beta_1} \cap E_{i,r_2}^{\beta_2}=\emptyset$, the topological intersection number of $\widehat{E}_{i,r_1}^{\beta_1}$ and $\widehat{E}_{i,r_2}^{\beta_2}$ in $\mathcal{V}$ is therefore given by 
$$\widehat{E}_{i,r_1}^{\beta_1} \cdot \widehat{E}_{i,r_2}^{\beta_2} =  Q ( \sigma_{i,r_2,\alpha,\beta_1}\cdot_{\Sigma_{r_{2},\alpha}} S_{i,r_2,\alpha,\beta_2}),$$
where $\cdot_{\Sigma_{r,\alpha}}$ denotes the intersection product in $\Sigma_{r,\alpha}$ with its induced orientation. The intersection number $\sigma_{i,r_2,\alpha,\beta_1}\cdot_{\Sigma_{r,\alpha}} S_{i,r_2,\alpha,\beta_2}$ is equal to the linking number of $\sigma_{i,r_2,\alpha,\beta_1} = E_{i,r_1}^{\beta_1} \cap \Sigma_{r_2,\alpha}$ and $\partial S_{i,r_2,\alpha,\beta_2} = Q \sigma_{i,r_2,\alpha,\beta_2}$ in $\Sigma_{r,\alpha}$, which is positive by Claim \ref{claim-upperboundonboundarycurve}. On the other hand, Claim \ref{claim-upperboundonboundarycurve} gives us that for $j=1,2$,
\begin{equation*}
\begin{split}
\int_{\widehat{E}_{i,r_j}^{\beta_j}} \widetilde{\omega}_{i,-1} &= Q \int_{\psi_{i,-1}(E_{i,r_j}^{\beta_j})} \widetilde{\omega}_{i,-1} - \sum_{ \{\alpha \, | \, x_{\alpha}\in E^{\beta}\} } \int_{\psi_{i,-1}(S_{i,r_j,\alpha,\beta_j})} \widetilde{\omega}_{i,-1} \\
& = Q \underbrace{\int_{E_{i,r_j}^{\beta_j}} \psi_{i,-1}^{\ast} \widetilde{\omega}_{i,-1}}_{\xrightarrow[i\to \infty]{\textnormal{\eqref{currentconverge}}} a_{\beta_j} \int_{E^{\beta_j}\cap \Omega_r}\omega \,>\,0} - \sum_{ \{\alpha \, | \, x_{\alpha}\in E^{\beta}\} } \underbrace{\int_{S_{i,r_j,\alpha,\beta_j}} \psi_{i,-1}^{\ast} \widetilde{\omega}_{i,-1}}_{<\,Cr^2} \\ & >0
\end{split}
\end{equation*}
for sufficiently small $r>0$ and $i\geq \underline{i}(r)$ large enough. Because $H_{2}(\mathcal{V};\mathbb{Z})$ is spanned
by $[\widetilde{E}]$, there exist $m_1,m_2 \in \mathbb{Z}$ such that $[\widehat{E}_{i,r_1}^{\beta_1}]=m_1 [\widetilde{E}]$ and $[\widehat{E}_{i,r_2}^{\beta_2}]=m_2 [\widetilde{E}]$. Since
$$0< \int_{\widehat{E}_{i,r_j}^{\beta_j}} \widetilde{\omega}_{i,-1} = m_j \int_{\widetilde{E}} \widetilde{\omega}_i =2\pi m_j$$
for $j=1,2$, we must have $m_1,m_2> 0$. We therefore arrive at the fact that
$$0< \widehat{E}_{i,r_1}^{\beta_1} \cdot \widehat{E}_{i,r_2}^{\beta_2} = m_1 m_2 \widetilde{E} \cdot \widetilde{E} = -m_1 m_2,$$
a contradiction.

\item  Fix $\alpha \in \{1,\ldots,N\}$ and let $\mathcal{C}_{i,r,\alpha} \subseteq \mathcal{V}$ denote the connected component of $\widetilde{M}\setminus \psi_{i,-1}(\Sigma_{r,\alpha})$ which does not intersect $\psi_{i,-1}(\Omega_r)$. Equip $\psi_{i,-1}(\Sigma_{r,\alpha})$ with the pushforward $\mathcal{D}_{i,r,\alpha}:= (\psi_{i,-1})_{\ast} \mathcal{D}_{r,\alpha}$ of the contact structure from Lemma \ref{gooddomain}(iii). For $r>0$ sufficiently small and $i\geq \underline{i}(r)$ large, it follows that $\widetilde{\omega}_{i,-1}>0$ on $\mathcal{D}_{i,r,\alpha}$, where $\widetilde{\omega}_{i,t}$ is the K\"ahler form of $\widetilde{g}_{i,t}$ as defined before \eqref{Fconverge}. Thus, $(\mathcal{C}_{i,r,\alpha},\widetilde{\omega}_{i,-1})$ is a symplectic filling of the contact three-manifold $(\psi_{i,-1}(\Sigma_{i,r,\alpha}),\mathcal{D}_{i,r,\alpha})$, which itself is contactomorphic to $\mathbb{S}^3/\Gamma_{\alpha}$ with its standard contact structure.

Next, arguing as in the proof of Claim \ref{claim-upperboundonboundarycurve} and applying Lemma \ref{cappinglemma}(d) to the rescalings $r^{-2}\widetilde{\omega}_{i,-1}$, keeping in mind \eqref{areasmall}, we can find $r>0$ arbitrarily small and a smooth 2-chain $S_{i,r,\alpha} \subseteq \Sigma_{r,\alpha}$ such that $\partial S_{i,r,\alpha}=Q (E_{i} \cap\Sigma_{r,\alpha})$ and
$$\left| \int_{S_{i,r,\alpha}} \psi_{i,-1}^{\ast} \widetilde{\omega}_{i,-1} \right| \leq Cr^2$$
for $i\geq \underline{i}(r)$ sufficiently large. Set 
\[
\widehat{E}_{i,r,\alpha}:=Q( \widetilde{E}\setminus (\widetilde{E}\cap \mathcal{C}_{i,r,\alpha}))-\psi_{i,-1}(S_{i,r,\alpha}).
\]
This is a smooth 2-chain in $\mathcal{V}$ satisfying $\partial \widehat{E}_{i,r,\alpha}=0$. For sufficiently small $r>0$ and $i\geq \underline{i}(r)$ large enough, we then have that
\begin{equation*}
\begin{split}
\int_{\widehat{E}_{i,r,\alpha}}\widetilde{\omega}_{i,-1}&= Q\int_{\widetilde{E} \setminus \mathcal{C}_{i,r,\alpha}}\widetilde{\omega}_{i,-1} - \int_{\psi_{i,-1}(S_{i,r,\alpha})} \widetilde{\omega}_{i,-1}\\
&=Q\underbrace{\int_{E_{i,r}} \psi_{i,-1}^{\ast}\widetilde{\omega}_{i,-1}}_{\xrightarrow[i\to \infty]{\textnormal{\eqref{currentconverge}}}\,\sum_{\beta=1}^{P} a_{\beta}\int_{E^{\beta}\setminus B(x_{\alpha},r)}\omega\,>\,0}-\underbrace{\int_{S_{i,r,\alpha}} \psi_{i,-1}^{\ast}\widetilde{\omega}_{i,-1}}_{<\,Cr^{2}}\\
&>0.
\end{split}
\end{equation*}
In particular, $\widehat{E}_{i,r,\alpha}$ represents a non-trivial homology class in $H_{2}(\mathcal{V};\mathbb{Z})$. (This also shows that $\widehat{E}_{i,r,\alpha}$ represents a non-trivial class in $H_2(\mathcal{V}\setminus \mathcal{C}_{i,r,\alpha};\mathbb{Q})$.) If $(\mathcal{C}_{i,r,\alpha},\widetilde{\omega}_{i,-1})$ were not a minimal symplectic filling of $\psi_{i,-1}(\Sigma_{r,\alpha}) \cong \mathbb{S}^3/\Gamma_{\alpha}$, then $\mathcal{C}_{i,r,\alpha}$ would contain an embedded smooth surface $\Delta \cong \mathbb{S}^2$ with self-intersection $-1$, disjoint from $\widehat{E}_{i,r,\alpha}$. In other words, $[\Delta],[\widehat{E}_{i,r,\alpha}]\in H_2(\mathcal{V};\mathbb{Z})$ would define two non-trivial homology classes whose intersection is zero, which would again contradict the fact that $H_2(\mathcal{V};\mathbb{Z})$ is spanned by a single class with  self-intersection $-1$. $(\mathcal{C}_{i,r,\alpha},\widetilde{\omega}_{i,-1})$ is therefore a minimal symplectic filling of $\Sigma_{r,\alpha}$.

Suppose by way of contradiction that $\Gamma_{\alpha}$ were not of the claimed type for some $\alpha\in \{1,...,N\}$. By \cite[Main Theorem 4]{milnorfibers}, $\mathcal{C}_{i,r,\alpha}$ would then be diffeomorphic to the Milnor fiber of a smoothing of $\mathbb{C}^2/\Gamma_{\alpha}$. By \cite[Remark 5.10]{LooWahl}, we would then have $H_2(\mathcal{C}_{i,r,\alpha};\mathbb{Q})\neq 0$. Now, because $H_k(\partial \mathcal{C}_{i,r,\alpha};\mathbb{Q})\cong H_k(\mathbb{S}^3/\Gamma;\mathbb{Q}) =0$ for $k=1,2$, the Mayer-Vietoris sequence tells us that 
$$H_2(\mathcal{C}_{i,r,\alpha};\mathbb{Q}) \oplus H_2 (\mathcal{V}\setminus \mathcal{C}_{i,r,\alpha};\mathbb{Q}) \to H_2(\mathcal{V};\mathbb{Q})$$
is an isomorphism, and as remarked above, $\widehat{E}_{i,r,\alpha}$ represents a non-trivial homology class in \linebreak $H_2(\mathcal{V}\setminus \mathcal{C}_{i,r,\alpha};\mathbb{Q})$ so that $\operatorname{rank}(H_2(\mathcal{V}\setminus \mathcal{C}_{i,r,\alpha};\mathbb{Q})) \geq 1$. It then follows from the fact that $H_2(\mathcal{C}_{i,r,\alpha};\mathbb{Q})\neq 0$ and $H_2(\mathcal{V};\mathbb{Q})\cong \mathbb{Q}$ that
$$2\leq \operatorname{rank}(H_2(\mathcal{C}_{i,r,\alpha};\mathbb{Q}))+\operatorname{rank}(H_2(\mathcal{V}\setminus \mathcal{C}_{i,r,\alpha};\mathbb{Q})) = \operatorname{rank}(H_2(\mathcal{V};\mathbb{Q})) = 1,$$
a contradiction.
\end{enumerate}
\end{proof}

We now know that $E$ is irreducible and is the unique exceptional divisor of $X\to \mathbb{C}^2$. Furthermore, we also know that the orbifold singularities of $X$ are contained in $E$ and are of a special type. We will now use these facts to show that $X$ in fact has no orbifold singularities, hence is the FIK shrinker. 

The proof proceeds by computing the trace of the intersection matrix associated to the exceptional set of the resolution $\widetilde{X}\to X\to \mathbb{C}^2$, that is, the composition of the minimal resolution $\widetilde{X}\to X$ with the Remmert reduction $X\to\mathbb{C}^{2}$, in two different ways. On one hand, $\widetilde{X}\to \mathbb{C}^2$ is a birational morphism between smooth varieties, and so is a composition of blowdown maps. On the other hand, a description due to \cite{Kollar} of the minimal resolutions of cyclic singularities of type $\frac{1}{n^2}(1,na-1)$ for $n>a \geq 1$ with $\operatorname{gcd}(a,\,n)=1$ gives us another way to compute the intersection matrix. We then argue that the results of these two computations are incompatible unless $X$ has no orbifold singularities at all. 

\begin{theorem} \label{ShrinkerisFIK} $X$ is the FIK shrinker.
\end{theorem}

\begin{proof} By Proposition \ref{irreducible}(ii), any singularity of $X$ is of the form $\frac{1}{n^2}(1,na-1)$ for some $n>a\geq 1$ with $\textnormal{gcd}(n,a)=1$. Let $N$ be the number of orbifold points of $X$ (which are contained in $E$) and let $\pi:\widetilde{X}\to X$ be the minimal resolution of $X$. By \cite[Theorem 7.4.16]{ishii}, the exceptional set of $\pi$ is a disjoint union of $N$ chains of $m_{i}$-curves $\{E_{i,1},...,E_{i,m_i}\}$, where $1\leq i\leq N$. This implies that 
$$E_{i,j}\cdot E_{i,k} = \left\{ \begin{array}{cc}
    0 & |j-k|>1, \\
    1 & |j-k|=1,\\
    -b_{i,j} & j=k,
\end{array} \right.$$
for some integers $b_{i,j}\geq 2$,
and the number of exceptional curves is $\sum_{i=1}^N m_i$.

Thanks to the singularity type, by \cite[Proof of Proposition 3.11]{Kollar} we can obtain each tuple $(b_{i,1},...,b_{i,m_i})$ from the $1$-tuple $(4)$ using operations of the following two types:
\begin{enumerate}
    \item Given $(a_1,...,a_k)$, output $(2,a_1,...,a_{k-1},a_k+1)$;
    \item Given $(a_1 ,...,a_k)$, output $(a_1+1,a_2,...,a_k,2)$.
\end{enumerate}
In particular, the trace of the intersection matrix $(E_{i,j}\cdot E_{i',j'})$ is given by
$$-\sum_{i=1}^N \left( 4 + 3(m_i-1) \right) = -N - 3\sum_{i=1}^N m_i.$$
Recall from Proposition \ref{nocones} that the exceptional set of the Remmert reduction $\rho:X\to \mathbb{C}^2$ is precisely $E$. By Proposition \ref{minusone}, the proper transform $\widehat{E}$ of $E$ with respect to $\pi$ is a $(-1)$-curve. The exceptional curves of $\rho \circ \pi$ are $\widehat{E}$ and the $E_{i,j}$, giving $1+\sum_{i=1}^N m_i$ exceptional curves of $\rho \circ \pi$. Moreover, it follows that the trace of the intersection matrix is $-\left( N+1+3\sum_{i=1}^N m_i \right)$. 

On the other hand, $\rho \circ \pi : \widetilde{X} \to \mathbb{C}^2$, being a birational morphism between quasiprojective smooth surfaces, decomposes as a composition of birational morphisms between smooth varieties \cite[Theorem 1-8-2]{matsuki} 
$$\widetilde{X} = X_m \xrightarrow{\pi_m} X_{m-1} \to \cdots \to X_1 \xrightarrow{\pi_1} \mathbb{C}^2,$$
where each $\pi_j$ is the blowdown along a $(-1)$-curve in $X_j$ and $\pi_1 \circ \cdots \circ \pi_m = \rho \circ \pi$. In particular, the exceptional set of $\rho \circ \pi$ comprises precisely $m$ exceptional curves. In addition, at each stage except for $\pi_1$ and $\pi_2$, blowing up along a point of intersection of two exceptional curves decreases the trace of the intersection matrix by 3. Otherwise, the trace of the intersection matrix decreases by 2. As a result, the trace of the intersection matrix of $X_j \xrightarrow{\pi_j} X_{j-1} \to \cdots \to \mathbb{C}^2$ for $j\geq 3$ is at least $-3$ plus that of $X_{j-1} \to \cdots \to \mathbb{C}^2$. Clearly, the trace of the intersection matrix of $\pi_1 \circ \pi_2:X_2 \to \mathbb{C}^2$ is $-3$.

By Proposition \ref{nocones}, we know that $m\geq 1$. Moreover, $m=1$ if and only if $X$ is smooth,
so suppose that $N\geq 1$. Then $m\geq 2$, the trace of the intersection matrix of $\rho \circ \pi$ is bounded below by $-3(m-2)-3=-3(m-1)$, and $m=1+\sum_{i=1}^N m_i$. Combining estimates, we find that
$$-3\sum_{i=1}^N m_i  \leq - \left( N+1+3\sum_{i=1}^N m_i \right),$$
that is, $N+1\leq 0$, a contradiction. We therefore conclude that $N=0$, i.e., $X$ has no singular points. By \cite[Theorem E(3)]{CDS}, we can now identify $X$ is the FIK shrinker up to pullback by an element of $GL(2,\,\mathbb{C})$.
\end{proof}

We finally prove that $(\widetilde{M},(\widetilde{g}_t)_{t\in [0,T)})$ has a Type I curvature bound by combining Theorem \ref{ShrinkerisFIK} and Proposition \ref{NoncollapsedTypeI}.

\begin{proof}[Proof of Theorem \ref{noncollapsedtheorem}] By Proposition \ref{NoncollapsedTypeI}, it suffices to find $C>0$ such that for all $(x,t)\in \widetilde{M} \times [0,T)$ with $P_{\widetilde{g}}^{\ast-}(x,t,\sqrt{T-t})\cap \widetilde{E}\neq \emptyset$, we have 
$$|\Rm_{\widetilde{g}}|_{\widetilde{g}}(x,t)\leq \frac{C}{T-t}.$$
Suppose by way of contradiction that there are $(x_i,t_i)\in \widetilde{M}\times [0,T)$ such that 
\begin{equation} \label{eq:curvatureblowsup} \lim_{i\to \infty}|\Rm_{\widetilde{g}}|_{\widetilde{g}}(x_i,t_i)(T-t_i)=\infty\end{equation}
and $P_{\widetilde{g}}^{\ast -}(x_i,t_i,\sqrt{T-t_i}) \cap (\widetilde{E}\times [0,T)) \neq \emptyset$. Set $\widetilde{g}_{i,t}:=(T-t_i)^{-1}g_{T+(T-t_i)t}$ and fix $x_0 \in \widetilde{E}$. We have shown in Theorem \ref{ShrinkerisFIK} that the unique tangent flow $X$ corresponding to any conjugate heat kernel based at $(x_0,T)$ is modeled on the FIK shrinker, hence \eqref{Fconverge} holds. Because $X$ is smooth, we know that the convergence is in fact smooth and $E$ is the unique $(-1)$ curve of $X$. Fix a bounded neighborhood $U$ of $E$ in $X$. As $(\psi_{i,t}^{\ast}\widetilde{g}_{i,t},\psi_{i,t}^{\ast} \widetilde{J}) \to (g_{t},J)$ in $C_{\operatorname{loc}}^{\infty}(X \times (-\infty,0))$, we can apply \cite[Corollary 2.4]{CCD} to obtain $\psi_{i,t}^{\ast}\widetilde{J}$-holomorphic curves $\widetilde{E}_{i,t} \subseteq U \times \{t\}$ with self-intersection $(-1)$. Recalling from Lemma \ref{containedinV} that $\psi_{i,t}(U \times [-2,-1])\subseteq \mathcal{V}$ for sufficiently large $i\in \mathbb{N}$, this means that $\psi_{i,t}(\widetilde{E}_{i,t})$ is a $(-1)$-curve in $(\mathcal{V},\widetilde{J})$. However, the only such $(-1)$-curve is $\widetilde{E}$, so that $\widetilde{E}=\psi_{i,t}(\widetilde{E}_{i,t})$ for all $t\in [-2,-1]$. In particular, we see that $\widetilde{E} \subseteq \psi_{i,t}(U)$ for all $t\in [-2,-1]$ if $i\in \mathbb{N}$ is sufficiently large. Thus, it follows that $$P_{\widetilde{g}_i}^{\ast -}(x_i,-1;1)\cap \psi_{i}(U\times [-2,-1])\neq \emptyset.$$
Set $K:= 2 \sup_X |{\Rm_g}|_g <\infty$.  Then for any $D>0$, the smooth Cheeger-Gromov convergence $\psi_i^{\ast}\widetilde{g}_i\to g$ implies that
$$\sup_{B_{\widetilde{g}_i}(\psi_{i,t'}(x'),t',2D)\times [-4,-1]} |\Rm_{\widetilde{g}_i}|_{\widetilde{g}_i} \leq K$$
for all $(x',t')\in U\times [-2,-1]$ whenever $i\geq \underline{i}(D)$ is sufficiently large. Choose $(x',t')\in \linebreak\psi_i(U\times[-2,-1])\cap P_{\widetilde{g}_i}^{\ast-}(x_i,-1,1)$ and apply \cite[Proposition 9.4(c) \& Corollary 9.6(b)]{Bam1} to obtain the inclusion
$$(x_i,-1)\in P_{\widetilde{g}_i}^{\ast}(x',t';2,0,-1-t')\subseteq B_{\widetilde{g}_i}(x',t';2D) \times [-2,-1]$$
for some $D=D(K)>0$ large. This means that $|\Rm_{\widetilde{g}_i}|_{\widetilde{g}_i}(x_i,-1)\leq K$, contradicting \eqref{eq:curvatureblowsup}. Theorem \ref{noncollapsedtheorem} now follows.
\end{proof}

\appendix
\section{Holomorphic coordinates for converging complex structures}

In this appendix, we show that a convergent sequence of K\"ahler structures near a given point give rise to a convergent sequence of holomorphic coordinates. This is used in Lemma \ref{Hausdorff} and Proposition \ref{ExcSet}.

\begin{lemma} \label{coordlemma}
Let $B\subset\mathbb{C}^{n}$ be the unit ball centred at the origin, let $\widehat{J}$ denote the standard complex structure on $B$, and let $(z_{1},\ldots,z_{n})$ denote $\widehat{J}$-holomorphic coordinates on $B$. Suppose that $\{J_{i}\}_{i}$ is a sequence of complex structures converging locally smoothly to $\widehat{J}$ on $B$. Then there exists a ball $\hat{B}\subset B$ centered at the origin such that for all $i$ large enough, there exist $J_{i}$-holomorphic coordinates $(z_{1}^{(i)},\ldots,z_{n}^{(i)})$ such that $z_{j}^{(i)}\to z_{j}$ locally smoothly on $\hat{B}$ as $i\to\infty$. 
\end{lemma}

\begin{proof}
Consider the $\widehat{J}$-holomorphic coordinates $(z_{1},\ldots,z_{n})$ on $B$. $B$ is strictly pseudoconvex with respect to $\widehat{J}$ and the function $|z|^2$. For each $i$ large enough, we want to find solutions $f_i^{\alpha}$ to the equation $\bar{\partial}_{J_{i}}(z_{\alpha}+f_{i}^{\alpha})=0$ converging locally smoothly to zero as $i\to \infty$. For any smaller ball $B'\subset B$ centred at the origin, 
it is clear that $B'$ will be pseudoconvex with respect to $J_{i}$ with the function $|z|^{2}$ for $i$ large enough. Set $$\omega_{i}:=\frac{\sqrt{-1}}{2}\partial_{J_{i}}\bar{\partial}_{J_{i}}\left(\sum_{j\,=\,1}^{n}e^{-|z_{j}|^{2}}\right),\qquad \widehat{\omega} := \frac{\sqrt{-1}}{2} \partial_{\widehat{J}} \bar{\partial}_{\widehat{J}}\left(\sum_{j\,=\,1}^{n}e^{-|z_{j}|^{2}}\right).$$
Then $\omega_{i}\to\widehat{\omega}>0$ as $i\to\infty$ smoothly in $B'$.   Consider values of $i$ large enough for which $\omega_{i}$ is positive-definite. 
Since $\Ric(\widehat{\omega})>0$ so that $\Ric(\omega_i)>0$ for $i$ large enough, H\"ormander's $L^{2}$-estimate
\cite[Theorem 3.8(2)]{SunZhangNoSemistability} (see also \cite[Theorem 8.6.1]{demailly})
gives us a solution $f_i^{\alpha}$ to $\bar{\partial}_{J_{i}}f_{i}^{\alpha}=-\bar{\partial}_{J_i} z_{\alpha}$ satisfying
$$\int_{B'}|f_{i}^{\alpha}|^{2}\,e^{-|z|^{2}}\omega_{i}^{n}\leq C(n)\int_{B'}|\bar{\partial}_{J_{i}}z_{\alpha}|^{2}_{\omega_{i}}\,e^{-|z|^{2}}\omega_{i}^{n},$$
hence also
$$\int_{B'}|f_{i}^{\alpha}|^{2}\,\omega_{i}^{n}\leq C(n)\int_{B'}|\bar{\partial}_{J_{i}}
z_{\alpha}|^{2}_{\omega_{i}}\,\omega_{i}^{n}.$$
Since $\omega_{i}\to\widehat{\omega}$ smoothly in $B'$ as $i\to\infty$, elliptic regularity tells us that $f_{i}^{\alpha}$ converge to zero locally smoothly on a smaller ball $B''\subset B'$. The functions
$\{z_{\alpha}+f_{i}^{\alpha}\}_{i=1}^{n}$, restricted to an even smaller ball $\hat{B}\subset B''$, are the desired $J_{i}$-holomorphic coordinates.
\end{proof}

\section{Capping Holomorphic Curves}

In this appendix, we prove a technical result required in the proof of Proposition \ref{irreducible}. The statements mostly follow from well-known estimates for $J$-holomorphic curves and compactness theorems for immersed submanifolds.

\begin{lemma} \label{cappinglemma}
Let $\Gamma \subseteq U(2)$ be a finite subgroup acting freely on $\mathbb{C}^2\setminus \{0\}$. Set $\mathcal{A}:=\{z\in\mathbb{C}^{2}/\Gamma \: |\: 1<|z|<2\}$ and $\mathbb{S}_r:=\{ z\in \mathbb{C}^2/\Gamma \: |\:|z|=r\}$. Let $(J_i,g_i)$ be a sequence of K\"ahler structures on $\mathcal{A}$ converging locally smoothly to a K\"ahler structure $(J,g)$ on $\mathcal{A}$ satisfying $\frac{1}{2}g_0 \leq g \leq 2g_0$ and $|\Rm_{g}|_g \leq 10$, where $g_0$ is the flat metric on $\mathcal{A}$. Let $E_{i}\subseteq\mathcal{A}$ be a sequence of properly embedded (possibly non-compact) $J_i$-holomorphic curves such that the following hold for some $B>0$.

\begin{enumerate}
\item $E_{i}$ converge in the Hausdorff sense (with respect to the metric $d_g$) as $i\to\infty$ to a smooth $J$-holomorphic curve $E_{\infty} \subseteq \mathcal{A}$ that intersects $\mathbb{S}_r$ transversely for all $r\in (1,\,2)$;

\item $\limsup_{i\to\infty}\mathcal{H}_{g_i}^{2}(E_{i}) \leq B$;

\item $\textnormal{genus}(E_{i})=0$;

\item The currents of integration $[E_i]$ converge to $k[E_{\infty}]$ in the weak sense of currents for some $k\in \mathbb{N}_{>0}$. 
\end{enumerate}

\noindent Then, after passing to a subsequence, there exist $C=C(B)>0$
and $r\in (1,2)$ such that the following hold.
\begin{enumerate}[label=\textnormal{(\alph{*})}, ref=(\alph{*})]
\item $\limsup_{i\to\infty}\mathcal{H}_{g_i}^{1}(E_{i}\cap \mathbb{S}_r)\leq C(B)$;
    \item For all $i\in\mathbb{N}$, $E_{i}\cap \mathbb{S}_r$ is
a disjoint union of simple closed curves;
    \item There is a disjoint union of circles $T$ together with diffeomorphisms $\vartheta_i:T\to E_{i}\cap \mathbb{S}_r$ converging smoothly to a $k$-sheeted covering map $\vartheta_{\infty}: T\to E_{\infty}\cap \mathbb{S}_r$ as $i\to\infty$;
    \item There is a sequence of smooth 2-chains $S_{i,r}\subseteq \mathbb{S}_r$ which satisfy $\partial S_{i,r}=|\Gamma| (E_{i}\cap \mathbb{S}_r)$
as 1-cycles, as well as $\limsup_{i\to\infty}\mathcal{H}_{g_i}^{2}(S_{i,r})\leq C(B)$.
\end{enumerate}
\end{lemma}

\begin{proof}
Define $\rho (z) := |z|$ for $z \in \mathcal{A}$. By Sard's theorem applied to $\rho|_{E_i}$, there is a subset $\mathcal{I}_{1}\subseteq(1,\,2)$ of full measure such that (b) holds for all $r\in\mathcal{I}_{1}$. By Fatou's lemma and the coarea formula, we have
that
\begin{align*}
\int_1^2 \liminf_{i\to \infty} \mathcal{H}_{g_i}^1(E_i \cap \mathbb{S}_r)\,dr &\leq \liminf_{i\to \infty}\int_{1}^{2}\mathcal{H}_{g_i}^{1}(E_{i}\cap \mathbb{S}_r)\,dr \\ &\leq  \liminf_{i\to \infty} \int_{E_i} |\nabla^{g_i} (\rho|_{E_i})|_{g_i} \,d\mathcal{H}_{g_i}^{2}\\ & \leq \sup_{\mathcal{A}} |\nabla^g \rho|_g \cdot \liminf_{i\to \infty} \mathcal{H}_{g_i}^2(E_i) \\ & \leq 2B.
\end{align*}
Hence there is a subset $\mathcal{I}_{2}\subseteq\mathcal{I}_{1}$ such that $|\mathcal{I}_2|\geq \frac{1}{2}$ and
$$\liminf_{i\to \infty} \mathcal{H}_{g_i}^1(E_i \cap \mathbb{S}_r) \leq 4B$$
for all $r\in \mathcal{I}_2$. 
For all such $r$, by considering a subsequence converging to 
$\liminf_{i\to \infty} \mathcal{H}_{g_i}^1(E_i \cap \mathbb{S}_r)$, we can further ensure that
$$\liminf_{i\to \infty} \mathcal{H}_{g_i}^1(E_i \cap \mathbb{S}_r)= \limsup_{i\to \infty} \mathcal{H}_{g_i}^1(E_i \cap \mathbb{S}_r).$$
In particular, for any choice of $r\in \mathcal{I}_2$,
after passing to a subsequence,
(a) and (b) hold.

Let $u_i: E_i \hookrightarrow \mathcal{A}$ be the inclusion map and let $\mathcal{K} \subseteq \mathcal{A}$ be any compact set. Let $\widetilde{\mathcal{K}} \subseteq \mathcal{A}$ be another compact set satisfying $\mathcal{K} \subseteq \textnormal{Int}(\widetilde{\mathcal{K}})$. Then for any compact subset $\widehat{\mathcal{K}} \subseteq \widetilde{\mathcal{K}}$, $u_i^{-1}(\widehat{\mathcal{K}}) = \widehat{\mathcal{K}} \cap E_i$ is compact. Because $E_i$ has no boundary in $\mathcal{A}$, it follows that the sequence $(u_i)$ is ``robustly $\mathcal{K}$-proper'' in the sense of \cite[Definition 2.3]{fish}. Because of this fact and assumptions (ii) and (iii), we can apply \cite[Proposition 3.11]{fish} to obtain the fact that
\[
\limsup_{i\to\infty}\int_{u_{i}^{-1}(\textnormal{Int}(\mathcal{K}))}|A_{i}|_{g_i}^{2}\,d\mathcal{H}_{g_i}^{2}<\infty
\]
for any compact subset $\mathcal{K} \subseteq\mathcal{A}$, where $A_{i}$ is
the second fundamental form of the embedding \linebreak$u_{i}:E_{i}\hookrightarrow  (\mathcal{A},g_i)$. We consider compact subsets $\mathcal{K}$ of the form
$\mathcal{A}_{\sigma}:=\{ 1+\sigma \leq  |z| \leq 2-\sigma \}$ for $\sigma>0$ small. By \cite[Theorem 5]{fishestimates}, for any $\delta>0$ sufficiently small, there exists $\hbar=\hbar(\delta) \in (0,\delta)$ such that for all $x\in \mathcal{A}_{2\delta}$, if $r\in (0,\hbar)$ satisfies 
$$\int_{u_i^{-1}(B_{g_i}(x,r))} |A_i|_{g_i}^2 d\mathcal{H}_{g_i}^2 \leq \hbar,$$
then $\sup_{u_i^{-1}(B_{g_i}(x,\frac{r}{2}))} |A_i|_{g_i} \leq r^{-1}$. An elementary covering argument then yields $N_0 \in \mathbb{N}$ such that for any $i\in \mathbb{N}$, there are $x_1^i,...,x_{N_0}^i \in \mathcal{A}_{2\delta}$ such that 
\[
\limsup_{i\to\infty}\sup_{u_{i}^{-1}\left(\mathcal{A}_{2\delta}\setminus\bigcup_{j=1}^{N_0}B_{g_i}(x_{j}^i,\sigma)\right)}|A_{i}|_{g_i}<\infty
\]
for any $\sigma>0$. By passing to a subsequence so that each sequence $(x_j^i)_{i\in \mathbb{N}}$ converges to some $x_j\in \mathcal{A}_{2\delta}$, we therefore have that
$$\limsup_{i\to \infty} \sup_{u_i^{-1}(\mathcal{A}_{3\delta} \setminus \bigcup_{j=1}^{N_0} B_g(x_j,\frac{1}{2}\sigma))} |A_i|_{g_i}<\infty.$$
By taking $\sigma\searrow 0$, then $\delta \searrow 0$, a diagonal argument gives a subset $\mathcal{I}_{3}\subseteq\mathcal{I}_{2}$
of full measure such that for each $r\in\mathcal{I}_{3}$, there exists
$\epsilon>0$ such that
\[
\limsup_{i\to\infty}\sup_{u_{i}^{-1}(\{r-\epsilon<|z|<r+\epsilon\})}|A_{i}|_{g_i}<\infty.
\]
By standard elliptic regularity (cf.~\cite[Proposition 5.2]{fishestimates}), there are also uniform bounds on $|(\nabla^{g_i|_{E_i}})^{\ell} A_i|_{g_i}$ for all $\ell \in \mathbb{N}$. For each $r\in \mathcal{I}_3$, we can therefore pass to a subsequence and apply the Breuning-Langer compactness theorem \cite[Corollary 7.13]{breuning}
to obtain $\epsilon_i \nearrow \epsilon$ and a smooth two-dimensional manifold $W$ along with a precompact open exhaustion $\{W_i\}_{i}\subseteq W$ and diffeomorphisms  
\[
\varphi_{i}:W_i \to u_{i}^{-1}(\{r-\epsilon_i <|z|<r+\epsilon_i \})
\]
such that $u_{i}\circ\varphi_{i}:W_i \to E_i \cap \{r-\epsilon_i <|z|<r+\epsilon_i \}$ converges in $C_{\text{loc}}^{\infty}(W)$
to a proper immersion 
$$u_{\infty}:W\to E_{\infty} \cap \{r - \epsilon < |z|< r+\epsilon\}.$$  The fact that $E_i \to E_{\infty}$ in the Hausdorff sense implies that $u_{\infty}$ is surjective.

Let $\omega_i$ denote the K\"ahler form of $(g_i,J_i)$. Because $u_{\infty}$ is a proper immersion, the number $\#u_{\infty}^{-1}(x)$ of points in $u_{\infty}^{-1}(\{x\})$ is locally constant as a function of $x\in \{r -\epsilon<|z|<r+\epsilon\} \cap E_{\infty}$. For any $\chi \in C_c^{\infty}(\{r-\epsilon<|z|<r+\epsilon\})$, we use (iv) to compute that
\begin{align*} k\int_{E_{\infty}}\chi\,d\mathcal{H}_{g}^{2}&=	\lim_{i\to\infty}\int_{E_{i}}\chi \,d\mathcal{H}_{g_{i}}^{2}\\ &=\lim_{i\to\infty}\int_{W_i}(\chi\circ u_{i}\circ\varphi_{i})((u_{i}\circ\varphi_{i})^{\ast}\omega_{i})
\\&=	\int_{W}(\chi\circ u_{\infty})(u_{\infty}^{\ast}\omega)
\\&= \int_{E_{\infty}} \#u_{\infty}^{-1}(x)\chi(x) \omega(x)
\\&=	\int_{E_{\infty}}\# u_{\infty}^{-1}(x)\chi(x) d\mathcal{H}_{g}^{2}(x).
\end{align*}
It follows that $\# u_{\infty}^{-1}(x)=k$ for $\mathcal{H}_g^2$-almost every $x\in E_{\infty}$, and because $\#u_{\infty}^{-1}(x)$ is locally constant, it is everywhere equal to $k$. $u_{\infty}$ is therefore a $k$-sheeted covering map. In particular, it restricts to a $k$-sheeted covering map $$u_{\infty}:T:=u_{\infty}^{-1}(E_{\infty} \cap \mathbb{S}_r)\to E_{\infty}\cap \mathbb{S}_r.$$

By (i), $u_{\infty}$ intersects $\mathbb{S}_r$ transversely for each $r\in(1,2)$. This means that any such $r$ is a regular value of $|u_{\infty}|$. Fix $r\in \mathcal{I}_3$. We then further have the convergence $u_{i}\circ\varphi_{i} \to u_{\infty}$ locally smoothly in a small annulus centered at $\mathbb{S}_{r}$. In particular, $r$ is a regular value of $|u_i \circ \varphi_i|$ for sufficiently large $i\in \mathbb{N}$, the corresponding level sets converge smoothly, and there exist diffeomorphisms $\zeta_i: T \to (u_i\circ \varphi_i)^{-1}(\mathbb{S}_r)$ such that $\vartheta_i:=u_i \circ \varphi_i \circ \zeta_i$ converge smoothly to $\vartheta := u_{\infty}|_T$.  This implies (c).

For each $r\in\mathcal{I}_{3}$, let $\sigma_{i,r} \subseteq \mathbb{S}_r^3$ be the preimage of $E_i \cap \mathbb{S}_r$ under the universal cover $\pi:\widehat{\mathcal{A}} \to \mathcal{A}$, where $\widehat{\mathcal{A}}:=\{ z \in \mathbb{C}^2 \: | \: 1<|z|<2\}$ and $\mathbb{S}_r^3:= \{z\in \mathbb{C}^2 \,|\, |z|=r\}$. Let $g_{\mathbb{S}_r^3}$ denote the round metric on $\mathbb{S}_r^3$. In the following claim, we show that there exists a point of $\mathbb{S}_r^3$ whose distance to $E$ has a definite lower bound.

\begin{claim} \label{claim:elementarycoveringargument} Fix $r\in[\frac{11}{10},\frac{19}{10}]$. Then for any $B'>0$, there exists $\epsilon=\epsilon(B')>0$ such that for all properly embedded $J$-holomorphic curves $E\subseteq\widehat{\mathcal{A}}$ with $\mathcal{H}_{\pi^{\ast}g}^{2}(E)\leq B'$, there exists $\zeta_{0}\in\mathbb{S}_{r}^{3}$ such that\linebreak $B_{g_{\mathbb{S}_r^3}}(\zeta_{0},\epsilon)\cap E=\emptyset.$ 
\end{claim}

\begin{proof}[Proof of Claim
\ref{claim:elementarycoveringargument}] 
We prove this claim by contradiction, so assume that there exists $B'>0$
such that for any given
$\epsilon>0$, one can find 
a $J$-holomorphic curve $E\subseteq\widehat{\mathcal{A}}$ with $\mathcal{H}_{g}^{2}(E)\leq B'$ such that $B_{g_{\mathbb{S}_{r}^{3}}}(\zeta,\epsilon)\cap E\neq\emptyset$ for all
$\zeta\in\mathbb{S}_{r}^{3}$.
Because $\pi^{\ast}g\leq 2\pi^{\ast}g_{0}$ for any $\zeta_{1},\zeta_{2}\in\mathbb{S}_{r}^{3}$, we know that (in $\mathcal{A}$) $d_{\pi^{\ast}g}(\zeta_{1},\zeta_{2})\leq \sqrt{2}d_{\pi^{*}g_{0}}(\zeta_{1},\zeta_{2}).$ Recall that 
\begin{equation} \label{eq:elementaryequivalence} d_{\pi^{\ast}g_0}(\zeta_1,\zeta_2)\leq d_{g_{\mathbb{S}_r^3}}(\zeta_1,\zeta_2)\leq 2d_{\pi^{\ast}g_0}(\zeta_1,\zeta_2)\end{equation}
for all $\zeta_1,\zeta_2\in \mathbb{S}_r^3$. Let $\{\zeta_{i}\}_{i=1}^{N}$ be a maximal subset of $\mathbb{S}_{r}^{3}$ such that $d_{g_{\mathbb{S}^{3}_r}}(\zeta_{i},\zeta_{j})\geq 100\epsilon$ whenever $i\neq j$. For each $i\in\{1,...,N\}$, there exists $\zeta_{i}'\in B_{g_{\mathbb{S}_{r}^{3}}}(\zeta_{i},\varepsilon)\cap E$. Then $d_{\pi^{\ast}g}(\zeta_i,\zeta_i')\leq \sqrt{2}\epsilon$ so that an application of \cite[Comment 1, p.178, and Proposition 4.3.1(ii)]{holocurves}
(see also \cite[Lemma 5.2]{CJL}) gives a universal constant $c>0$ such that $$\mathcal{H}_{\pi^{\ast}g}^{2}(B_{\pi^{\ast}g}(\zeta_{i},2\sqrt{2}\epsilon)\cap E)\geq\mathcal{H}_{\pi^{\ast}g}^{2}(B_{\pi^{\ast}g}(\zeta_{i}',\sqrt{2}\epsilon)\cap E)\geq c\epsilon^{2}.$$
Moreover, since $\{B_{g_{\mathbb{S}_{r}^{3}}}(\zeta_{i},300\epsilon)\}_{i=1}^{N}$ is a cover of $\mathbb{S}_{r}^{3}$, there exists $c_{0}>0$ universal such that $N\geq c_{0}\epsilon^{-3}$. On the other hand, if $\epsilon\in (0,\frac{1}{200})$, then $B_{\pi^{\ast}g}(\zeta_i,2\sqrt{2}\epsilon) \subseteq \widehat{\mathcal{A}}$ for all $i\in \{1,...,N\}$. Moreover, as $\pi^{\ast}g \geq \frac{1}{2}\pi^{\ast}g_0$ by assumption, we know that the collection $\{B_{\pi^{\ast}g}(\zeta_{i},2\sqrt{2}\epsilon)\}_{i=1}^{N}$ forms a pairwise-disjoint subset of $\widehat{\mathcal{A}}$. We therefore find that $$B'\geq\sum_{i=1}^{N}\mathcal{H}_{\pi^{\ast}g}^{2}(B_{\pi^{\ast}g}(\zeta_{i},2\sqrt{2}\epsilon)\cap E)\geq c\epsilon^{2}N\geq cc_{0}\epsilon^{-1}.$$ This leads to a contradiction if we let $\epsilon:=\frac{1}{2}cc_{0}(B')^{-1}$. 
\end{proof}

Now $\pi$ restricts to a covering map $\pi:\sigma_{i,r} \to E_i \cap \mathbb{S}_r$ with $|\Gamma|$ sheets. Let $g_{\mathbb{R}^{3}}$ denote the flat metric on $\mathbb{R}^{3}$. Applying Claim \ref{claim:elementarycoveringargument} to $\pi^{-1}(E_{\infty})$ gives $\zeta_{0}\in\mathbb{S}_r^{3}$ and $\epsilon = \epsilon(B)>0$ such that
\[
B_{g_{\mathbb{S}_r^3}}(\zeta_0,\epsilon)\cap \pi^{-1}(E_{\infty}) = \emptyset.
\]
By (i), it follows that 
\[
B_{g_{\mathbb{S}_r^3}}\left(\zeta_0,\frac{\varepsilon}{2}\right)\cap \pi^{-1}(E_i) = \emptyset
\]
for sufficiently large $i\in\mathbb{N}$.
Thus, letting $\varphi:\mathbb{S}_r^{3}\setminus\{\zeta_{0}\}\to\mathbb{R}^{3}$
denote the stereographic projection map based at $\zeta_{0}$, we see that there exists $R=R(B)>0$ such that 
\begin{equation} \label{eq:ContainedinRball} \varphi(\pi^{-1}(E_i)\cap \mathbb{S}_r^3) \subseteq B_{g_{\mathbb{R}^3}}(0,R)\end{equation} for sufficiently large $i\in \mathbb{N}$. Choose $\Lambda=\Lambda(B)>0$ such that
\begin{equation} \label{eq:stereographicequivalence}
\Lambda^{-1}(\varphi^{-1})^{\ast}g_{\mathbb{S}_r^{3}}\leq g_{\mathbb{R}^{3}}\leq\Lambda(\varphi^{-1})^{\ast}g_{\mathbb{S}_r^{3}}
\end{equation}
on $B_{g_{\mathbb{R}^3}}(0,R)$. By the solution to Plateau's problem (cf.~\cite[Section 4.5, Main Theorem]{plateau} for existence and \cite[Chapter 2.3, Theorem 1]{plateau2} for boundary regularity), there exists an area-minimizing smooth surface $\widetilde{S}_{i,r}\subseteq \mathbb{R}^3$
(with respect to the flat metric) which is a disjoint union of smoothly immersed disks, with boundaries satisfying $\partial \widetilde{S}_{i,r}=\varphi(\sigma_{i,r})$ as 1-cycles. By \eqref{eq:ContainedinRball} and \cite[Section 4.14, Theorem 2]{plateau}, we know that
\begin{equation} \label{eq:isoperimetric}\mathcal{H}_{g_{\mathbb{R}^3}}^2(\widetilde{S}_{i,r}) \leq R \mathcal{H}_{g_{\mathbb{R}^3}}^1(\partial \widetilde{S}_{i,r}).\end{equation}
Combining expressions \eqref{eq:elementaryequivalence}--\eqref{eq:isoperimetric}, we find that $S_{i,r}:=\varphi^{-1}(\widetilde{S}_{i,r})$ satisfies
\[
(100\Lambda)^{-1}\mathcal{H}_{\pi^{\ast}g_i}^{2}(S_{i,r})\leq \Lambda^{-1}\mathcal{H}_{g_{\mathbb{S}_r^3}}^2(S_{i,r})\leq \mathcal{H}_{g_{\mathbb{R}^{3}}}^{2}(\varphi(S_{i,r}))\leq R\mathcal{H}_{g_{\mathbb{R}^{3}}}^{1}(\varphi_{r}(\sigma_{i,r}))\leq 10R\Lambda^{\frac{1}{2}}\mathcal{H}_{\pi^{\ast}g_i}^{1}(E_{i}\cap \mathbb{S}_r^{3})
\]
for sufficiently large $i\in \mathbb{N}$, where we used the equivalence of $g_i$ and $g_0$ in the first and last inequalities. The fact that $\partial \left( \pi_{\ast}[S_{i,r}]\right) = |\Gamma| \cdot [E_i \cap \mathbb{S}_r]$ as 1-cycles now gives (d).
\end{proof}

\bibliographystyle{amsalpha}

\bibliography{ref2}

\end{document}